\documentclass[a4paper,12pt]{article}

\usepackage{amsmath, amssymb, float, mathtools, mathrsfs, xcolor, stmaryrd}
\usepackage{dblfloatfix} 
\usepackage{dynkin-diagrams}
\usepackage[shortlabels]{enumitem}

\usepackage{amsthm}

\numberwithin{equation}{section}
\newtheorem{lemma}{Lemma}[section]
\newtheorem{corollary}[lemma]{Corollary}
\newtheorem{proposition}[lemma]{Proposition}
\newtheorem{theorem}[lemma]{Theorem}
\theoremstyle{definition}

\newtheorem{remark}[lemma]{Remark}

\def\th@plain{%
  \thm@notefont{}
  \itshape 
}
\def\th@definition{%
  \thm@notefont{}
  \normalfont 
}

\renewcommand{\qedsymbol}{$\blacksquare$}

\newcommand{\C}{\mathbb{C}}
\newcommand{\Z}{\mathbb{Z}}
\newcommand{\bH}{\mathbb{H}}

\renewcommand{\a}{\alpha}
\newcommand{\e}{\varepsilon}

\newcommand{\cA}{\mathcal{A}}

\newcommand{\cD}{\mathcal{D}}

\newcommand{\cM}{\mathcal{M}}

\newcommand{\cR}{\mathcal{R}}
\newcommand{\cT}{\mathcal{T}}
\newcommand{\cZ}{\mathcal{Z}}

\newcommand{\fA}{\mathfrak{A}}

\newcommand{\fS}{\mathfrak{S}}
\newcommand{\fsl}{\mathfrak{sl}}
\newcommand{\fgl}{\mathfrak{gl}}

\newcommand{\ta}{\widetilde{a}}
\newcommand{\tb}{\widetilde{b}}
\newcommand{\tc}{\widetilde{c}}
\newcommand{\tD}{\widetilde{D}}

\newcommand{\tY}{\widetilde{Y}}

\newcommand{\hy}{\widehat{y}}
\newcommand{\htt}{\widehat{t}}

\newcommand{\HH}{H\hspace{-0.2em}H}

\DeclareMathOperator{\id}{id}
\DeclareMathOperator{\trig}{trig}

\DeclareMathOperator{\Res}{Res}

\DeclareMathOperator{\Mat}{Mat}
\DeclareMathOperator{\rk}{rank}
\DeclareMathOperator{\gr}{gr}

\usepackage{ulem}

\usepackage[T1]{fontenc}
\usepackage[utf8]{inputenc}

\usepackage{hyperref}

\usepackage{tikz-cd}

\allowdisplaybreaks

\usepackage{authblk}
\title{$q$-Analogue of the degree zero part of a rational Cherednik algebra and generalised Van Diejen's Hamiltonians}
\author[1]{Misha Feigin
}
\author[2]{Martin Vrabec
}
\date{\today}
\affil[1]{School of Mathematics and Statistics, University of Glasgow, University Place, Glasgow G12 8QQ, UK}
\affil[2]{Max Planck Institute for Mathematics, Vivatsgasse 7, Bonn 53111, Germany}

\AtEndDocument{\bigskip{\footnotesize%
  \textit{E-mail addresses:} \, \texttt{misha.feigin@glasgow.ac.uk, } 
  \texttt{martinvrabec222@gmail.com} \par
}}

\begin{document}
\maketitle

\begin{abstract}
   Inside the double affine Hecke algebra of type $GL_n$, which depends on two parameters $q$ and $\tau$, we define a subalgebra~$\mathbb{H}^{\mathfrak{gl}_n}$ that may be thought of as a 
    $q$-analogue 
    of the degree zero part of the corresponding rational Cherednik algebra. We prove that the algebra $\mathbb{H}^{\mathfrak{gl}_n}$ is a flat $\tau$-deformation of the crossed product of the group algebra of the symmetric group with the image of the Drinfeld--Jimbo quantum group~$U_q(\mathfrak{gl}_n)$ under the $q$-oscillator (Jordan--Schwinger) representation. We find all the defining relations and an explicit PBW basis for the algebra~$\mathbb{H}^{\mathfrak{gl}_n}$. We describe its centre and establish a double centraliser property.
    
    As an application, we also obtain new integrable generalisations of 
    a Macdonald--Ruijsenaars system with an external Morse (exponential) potential introduced by Van Diejen.  
    In particular, we extend Van Diejen's Hamiltonian to the case of a system with two different types of interacting particles, and thus generalise the deformed Macdonald--Ruijsenaars operators of Chalykh--Sergeev--Veselov to the setting with an external field.
    

\vspace{1em}

\noindent \textbf{MSC classification code:} 81R12

\end{abstract}

\section{Introduction}
Rational Cherednik algebras (RCA) are a remarkable class of algebras associated with finite Coxeter groups $W$~\cite{EG}. They have deep connections to integrable systems, geometry, and combinatorics, as well as a rich representation theory (see e.g.~\cite{Etingof}).  

They admit a faithful representation on a space of polynomials. In this representation, and in the case of the symmetric group~$W=\fS_n$, 
the corresponding RCA $H_n = H_{n,c}$ ($c \in \C$) acts on  $\C[X_1, \dots, X_n]$, and it is generated by the transpositions $s_{ij} = (i,j) \in \fS_n$, multiplication operators~$X_i$, and the rational Dunkl operators~\cite{Dunkl}
\begin{equation*}
    \nabla_i = \partial_{X_i} - \sum_{\substack{j=1 \\ j \neq i}}^n \frac{c}{X_i-X_j}(1-s_{ij}).
\end{equation*}
Here $\partial_{X_i} = \frac{\partial}{\partial X_i}$ is a partial derivative.

The RCA $H_n$ is a graded algebra, where the grading is determined by assigning degree 0 to elements of the group $\fS_n$, degree~$1$ to the multiplication operators~$X_i$, and degree $-1$ to Dunkl operators.  
The degree zero subalgebra~$H^{\fgl_n} = H_c^{\fgl_n}$ is also interesting in its own right from various points of view. 
It is generated by~$\fS_n$ and the operators $X_i \nabla_j$ ($1 \leq i, j \leq n$). 

The algebra $H^{\fgl_n}$ enjoys, as the notation for it suggests, a link to Lie theory. More precisely, it is a flat $c$-deformation of the crossed product of the group algebra~$\C \fS_n$ with a certain quotient $U(\fgl_n)/I$ of the universal enveloping algebra~$U(\fgl_n)$ of the Lie algebra~$\fgl_n$ over a two-sided ideal~$I$, as was established by Hakobyan and one of the authors in~\cite{FH}. 
The quotient~$U(\fgl_n)/I$ is the image of $U(\fgl_n)$ under the so-called oscillator (also known as Jordan--Schwinger) representation $\rho_{\text{JS}}$ that maps the standard generators of~$\fgl_n$ to the operators~$X_i \partial_{X_j}$.

Similarly to the RCA itself, the algebra 
$H^{\fgl_n}$ is a quadratic algebra of Poincar\'e--Birkhoff--Witt (PBW) type. In contrast to the RCA, the defining relations of $H^{\fgl_n}$ include relations that are not of a commutator type. The associated graded algebra is the crossed product of $\C \fS_n$ with the algebra of polynomial functions on the space of $n\times n$ complex matrices of rank at most one~\cite{FH}.

The centre of the RCA is trivial~\cite{BG}, but the RCA has a commutative subalgebra which acts (in the polynomial representation) on symmetric polynomials as the rational Calogero--Moser operator and its quantum integrals~\cite{Heckman}.
On the other hand, the centre of the degree zero subalgebra $H^{\fgl_n}$ is generated by the Euler operator~$eu$, which can be related to the rational quantum Calogero--Moser Hamiltonian with an additional harmonic potential term by an automorphism of the RCA~\cite{FH}. The central quotient $\overline{H}^{\fgl_n} = H^{\fgl_n} / (eu + \, \text{const})$ is isomorphic to the algebra of global sections of a sheaf of Cherednik algebras on the projective space~\cite{Etingof2, BM}. Further properties of this algebra and its `$t=0$' (classical) version were studied recently in \cite{BFH}. The spherical subalgebra of the classical version of the algebra $\overline{H}^{\fgl_n}$ gives a deformation of the conic symplectic singularity $\overline {\mathcal O}_{min}/\fS_n$, where ${\mathcal O}_{min}$ is 
the minimal nilpotent orbit in~$\fgl_n$ which carries  a natural action of the symmetric group $\fS_n$. In turn, the spherical subalgebra of $\overline{H}^{\fgl_n}$ is a quantisation of the algebra of functions on the quotient $\overline {\mathcal O}_{min}/\fS_n$.

In this paper, we generalise the main parts of the above theory to the $q$-deformed setting by introducing and studying a certain subalgebra $\bH^{\fgl_n}$ inside Cherednik's double affine Hecke algebra (DAHA) $\bH_n = \bH_{n,q,\tau}$ of type $GL_n$. We note that even though the DAHA $\bH_n$ has a natural grading, the subalgebra $\bH^{\fgl_n}\subset \bH_n$ is in general  strictly smaller than the degree zero part. Another important difference with the RCA case is that the algebra~$\bH^{\fgl_n}$ contains the $Y$-elements of the DAHA. The main idea behind the definition of $\bH^{\fgl_n}$ is to replace the role of $U(\fgl_n)$ by the Drinfeld--Jimbo quantum group~$U_q(\fgl_n)$.
The algebra $U_q(\fgl_n)$ admits a representation $\rho$ which is a $q$-multiplicative generalisation of the Jordan--Schwinger map $\rho_{\text{JS}}$. 

We consider the image $A = \rho(U_q(\fgl_n))$ and define an algebra $\cA = \C \fS_n \ltimes A$, where the symmetric group acts in a natural way. We then define inside the DAHA $\bH_n$ a subalgebra~$\bH^{\fgl_n}$ whose generators are $\tau$-deformations of those of~$\cA$. Moreover, in a suitable $q \to 1$ limit, the algebra~$\bH^{\fgl_n}$ reduces to the degree zero part $H^{\fgl_n}$ of the RCA.
The following diagram summarises the relationships between the various algebras:

\vspace{1em}
\centerline{
                \begin{minipage}{.6\linewidth}
                     \begin{tikzcd}[ampersand replacement=\&] 
                        \bH^{\fgl_n} \dar{\tau = q^{-c/2}, \ q \to 1} \rar{\tau \to 1} \& \cA = \C \fS_n \ltimes \rho(U_q(\fgl_n)) \dar{q \to 1} \\ 
                         H^{\fgl_n} \rar{c \to 0} \& \C \fS_n \ltimes \rho_{\text{JS}}(U(\fgl_n)).
                     \end{tikzcd}
                \end{minipage}
}
\vspace{1em}

We give all the defining relations of the algebra $\bH^{\fgl_n}$, and show that it is an algebra of PBW type by explicitly constructing a PBW basis. We show that $\bH^{\fgl_n}$ is a flat $\tau$-deformation of the algebra $\cA$. 
We prove that the centre of~$\bH^{\fgl_n}$ is generated by a single invertible element $\tY$. When $q \to 1$, the central element $(1-q)^{-1}(1-\tY)$ reduces to the generator $eu$ of the centre~$\cZ(H^{\fgl_n})$. We also prove a double centraliser property that is related to the $(\fgl_n, \fgl_1)$ Howe duality.

The DAHA $\bH_n$ contains pairwise-commuting elements $D_i$ which can be thought of as a $q$-generalisation of the Dunkl operators, and which we use to define the algebra~$\bH^{\fgl_n}$. Similar but different commuting elements appear in the definition of a cyclotomic DAHA inside $\bH_n$ by Braverman, Etingof, and Finkelberg~\cite{BEF}. 
We show that the algebra $\bH^{\fgl_n}$ is isomorphic to the subalgebra of degree zero elements in this cyclotomic DAHA. The (spherical subalgebra of the) cyclotomic DAHA is realised as a ($K$-theoretic) quantised Coulomb branch of a framed quiver gauge theory in \cite{BEF}, and it would be interesting if a similar relation can be found for the algebra~$\bH^{\fgl_n}$.

We also consider pairwise-commuting elements~$\cD_i = \cD^{(l_1, l_2)}_i \in \bH_n$ of a more general form than $D_i$. The former depend on parameters $l_1, l_2 \in \Z_{\geq 0}$, and $a_j \in \C$ ($j=-l_1, \dots, l_2$). In the case $l_2 = 0$, they are equivalent to certain generators of a general cyclotomic DAHA~\cite{BEF}. These elements allow us to make a connection to integrable systems as follows.

By looking at the action of symmetric combinations of~$\cD_i$ on the space of symmetric Laurent polynomials, we arrive at families of new commuting $q$-difference operators related to a Macdonald--Ruijsenaars system with Morse potential terms introduced by Van Diejen~\cite{vD, vDE}. 
Van Diejen's model in question consists of Ruijsenaars' relativistic Calogero--Moser system (relativistic in the sense that it leads to a representation of the Poincar\'e algebra in 1+1 dimensions \cite{Ruij}) of hyperbolic type coupled to an external field. The differential limit of Van Diejen's system is given by the Hamiltonian
\begin{equation}\label{eqn: Inozemtsev}
    \sum_{i=1}^n \partial_{x_i}^2 - \sum_{i<j}^{n} \frac{c(c+1)}{2\sinh^2(\frac{x_i-x_j}{2})}  + \sum_{i=1}^{n}(A e^{x_i} + Be^{2x_i}), 
\end{equation}
where $x_i$ are the coordinates of the particles on a line, $c$ is the coupling constant for the Calogero--Moser inter-particle interaction, and $A$ and $B$ are coupling constants for the Morse potential.
In the case $B=0$, the classical version of this Hamiltonian was introduced by Adler~\cite{Adler}, who found a Lax pair and investigated the scattering behaviour of the system. The classical version of~\eqref{eqn: Inozemtsev} appeared in the work of Inozemtsev~\cite{Inozemtsev}, who also considered more general external potentials with 4 exponentials and constructed corresponding Lax pairs. In the rational limit, $\sinh z$ is replaced with $z$ and the external potential becomes, in general, quartic. In the case of Morse terms as in the Hamiltonian~\eqref{eqn: Inozemtsev}, the rational limit is the rational Calogero--Moser system with a harmonic term.  
The importance of Morse potential is that, in addition to preserving integrability in the sense of a large family of commuting conserved quantities, it also leads to a Hamiltonian which has a discrete spectrum in the quantum case \cite{IM}.
For some further aspects of Calogero--Moser systems in external fields, at the classical level, we refer the reader to~\cite{Polychronakos2, Avan, KP}. In the quantum relativistic case considered by Van Diejen, the Hamiltonian takes the form of a difference operator, which is due to the dependence on the canonical momenta operators being exponential.

Let us now illustrate the type of systems which we obtain with the use of the elements $\cD_i^{(l_1, l_2)}$.
For example, in the case $l_1=l_2 =1$ we obtain the following Hamiltonian
\begin{align*}
            M &= \a \sum_{i=1}^n \frac{1}{X_i} \left( \prod_{\substack{j = 1 \\ j \neq i}}^n \frac{\tau^2 X_i - X_j}{X_i - X_j} \right)t_i  + \beta \sum_{i=1}^n \frac{1}{X_i} \left( \prod_{\substack{j = 1 \\ j \neq i}}^n \frac{ X_i - \tau^2 X_j}{X_i - X_j} \right)t_i^{-1} + \gamma \sum_{i=1}^n \frac{1}{X_i},
\end{align*}
where $t_i$ is the $q$-shift operator in the coordinate $X_i$, and $\a,\beta, \gamma$ are independent parameters.  
Here $\tau$ is the coupling constant governing a Macdonald--Ruijsenaars type inter-particle interaction, and $X_i$ are exponentiated coordinates of the particles. The differential limit is a $q \to 1$ limit. In the case $\beta=0$, the differential limit of the operator $M$ is gauge-equivalent to the Hamiltonian~\eqref{eqn: Inozemtsev} with $B=0$ after changing the sign of the coordinates.

Relations to the known integrable difference Hamiltonians are as follows.
In the case of $\a=0$ (corresponding to $l_2 = 0$ and $l_1 = 1$), the operator $M$ appeared in the paper~\cite{BF} by Baker and Forrester. A more general version of their $q$-difference operator was found earlier by Van Diejen without using $q$-Dunkl operators~\cite{vD} by taking a limit of Koornwinder's operator~\cite{Koorn} in which the centre of mass is sent to infinity in a special way. Van Diejen's operator has a limit to the operator $M$ with an extra constraint on the parameters $\a$, $\beta$, and $\gamma$ \cite{vDE}.
Generalisations to higher $l_1$ (with $l_2=0$) of the Baker--Forrester operator 
were considered in~\cite{BEF}, which recover as a special case Chalykh's operators from~\cite{Chalykh}; see  \cite{ChalykhFairon} for an explicit form of such a Hamiltonian for $l_1=2$.  

Furthermore, Macdonald--Ruijsenaars operators of type $A$ admit integrable generalisations to systems with two types of particles \cite{Ch'00, SV}, known in the literature as deformed Macdonald--Ruijsenaars systems. The non-relativistic version of the corresponding Hamiltonians goes back to the papers \cite{CFV, SVdifferential}, and it corresponds to the generalised Calogero--Moser Hamiltonians where particles can have two different masses. 
These deformed Macdonald--Ruijsenaars systems are related to submodules of the polynomial representation of DAHA at special values of the parameters \cite{FS} (see \cite{F} for the non-relativistic rational version which corresponds to representations of rational Cherednik algebras). For a possible relation of these systems to quantum sine-Gordon field theory,  
see \cite{AHL} and references therein. 

We generalise Van Diejen's operator from~\cite{vD} to a Hamiltonian involving two types of particles, and we explain how to obtain quantum integrals for it. 
This is an extension to the Morse-perturbed setting of the deformed Macdonald--Ruijsenaars operators.
This also leads to a generalisation of the above operator $M$ and the Inozemtsev's operator~\eqref{eqn: Inozemtsev} to the case of two types of particles. 


The structure of the paper is as follows.
In Section~\ref{sec: RCA}, we recall the definitions and properties of the rational and trigonometric Cherednik algebras of type $GL_n$. In Section~\ref{sec: quantum group}, we review 
the definition of the quantum group~$U_q(\fgl_n)$ and its Jordan--Schwinger representation~$\rho$. We define the algebras~$A$ and~$\cA$, and study their properties. In Section~\ref{sec: DAHA}, we recall the definition of the DAHA~$\bH_n$. In Section~\ref{sec: Hgln}, we define the algebra~$\bH^{\fgl_n}$. In Section~\ref{sec: D}, we study the properties of the commuting elements~$D_i$. In Section~\ref{sec: presentation}, which is the most technical part of the paper, we give all the defining relations of~$\bH^{\fgl_n}$ and a basis for it. In Section~\ref{sec: centre}, we describe its centre and establish a double centraliser property that $\bH^{\fgl_n}$ satisfies as a subalgebra of a cyclotomic DAHA~\cite{BEF}. In Section~\ref{sec: integrable systems}, we derive new generalisations of Van Diejen's and related systems.

\section{Rational and trigonometric Cherednik algebras of type $GL_n$}\label{sec: RCA}

Let $c \in \C$ be a parameter. The RCA $H_n = H_{n, c}$ of type~$GL_n$ is the (unital, associative) algebra over $\C$ generated by the simple transpositions $s_k = (k, k+1)$ from the symmetric group $\fS_n$ ($1 \leq k \leq n-1$), and elements~$X_i$ and $y_i$ ($1 \leq i \leq n$) subject to the following relations \cite{EG}:
\begin{align}
    &[y_i, y_j] = 0 = [X_i, X_j], \nonumber \\
    &s_k X_k s_k = X_{k+1}, \quad \ [s_k, X_i] = 0 \textnormal{ for } i \neq k, k+1, \nonumber \\
    &s_k y_k s_k = y_{k+1},  \quad \ \ \ [s_k, y_i] = 0 \textnormal{ for } i \neq k, k+1, \nonumber \\
    &S_{ij} \coloneqq [y_i, X_j] = \begin{cases}
        1 - c \sum_{l \neq i} s_{il} \textnormal{ if } i = j,\\
        c s_{ij} \textnormal{ if } i \neq j. 
    \end{cases} \label{eqn: S_ij}
\end{align}
Here $j,l \in \{1, \dots, n\}$, and $s_{ij}$ denote the transpositions $(i,j) \in \fS_n$. The bracket~$[\cdot, \cdot]$ denotes the commutator. For notational convenience, we assume that $n\ge 2$ throughout.

The algebra $H_n$ admits a faithful representation, called the polynomial representation, on the space of polynomials $\C[X_1, \dots, X_n]$. 
The elements~$s_k$ act by swapping $X_k$ and $X_{k+1}$, the elements~$X_i$ act by multiplication, and~$y_i$ act as the standard commuting rational Dunkl operators $\nabla_i$ introduced in~\cite{Dunkl}, which $c$-deform~$\partial_{X_i}$ $= \tfrac{\partial}{\partial X_i}$ and are given by
\begin{equation*}
    \nabla_i = \partial_{X_i} - \sum_{\substack{j=1 \\ j \neq i}}^n \frac{c}{X_i-X_j }(1-s_{ij}).
\end{equation*}

The trigonometric Cherednik algebra (also known as degenerate double affine Hecke algebra) $\bH_n^{\trig} = \bH_{n,c}^{\trig}$ of type~$GL_n$ is the (unital, associative) algebra over $\C$ generated by~$s_k \in \fS_n$ ($1 \leq k \leq n-1$), and elements~$X_i^{\pm 1}$ and $\hy_i$ ($1 \leq i \leq n$) subject to the following relations:
\begin{align*}
    &[\hy_i, \hy_j] = 0 = [X_i, X_j], \qquad X_iX_i^{-1} = X_i^{-1}X_i = 1, \\
    &s_k X_k s_k = X_{k+1}, \qquad \quad \quad  [s_k, X_i] = 0 \textnormal{ for } i \neq k, k+1, \\
    &s_k\widehat{y}_{k+1} - \widehat{y}_k s_k = c,  \qquad \quad [s_k, \hy_i] = 0 \textnormal{ for } i \neq k, k+1,\\
    &(\widehat{y}_1 + \cdots + \widehat{y}_n)X_i = X_i (1 + \widehat{y}_1 + \cdots + \widehat{y}_n), \\
    & \widehat{y}_2 - X_1\widehat{y}_2X_1^{-1} = cs_1,
\end{align*}
where $1 \leq j \leq n$.
These can be obtained by taking the relations of the DAHA of type~$GL_n$ from \cite[p.~100]{Cherednik} and applying the trigonometric degeneration described in Section~\ref{sec: DAHA} below. 
Slightly different but equivalent sets of generators and relations are used in~\cite{AST} 
(see also~\cite[Section~3]{Suzuki}). 

The algebra $\bH_n^{\trig}$ admits a faithful representation, called the polynomial representation, on the space of Laurent polynomials $\C[X_1^{\pm 1}, \dots, X_n^{\pm 1}]$. The elements $s_k$ act by swapping $X_k$ and $X_{k+1}$, while~$X_i^{\pm 1}$ act by multiplication, and $\hy_i$ act as Cherednik's commuting trigonometric Dunkl operators~\cite{Ch'91}
\begin{equation*}
    \nabla_i^{\trig} = X_i \partial_{X_i} - \sum_{\substack{j=1 \\ j \neq i}}^n \frac{c}{1-X_j X_i^{-1}}(1-s_{ij}) - c\sum_{\substack{j=1 \\ j > i}}^n s_{ij} = X_i \nabla_i -  c\sum_{\substack{j=1 \\ j > i}}^n s_{ij}.
\end{equation*}

From the respective polynomial representations, one sees that there is an algebra embedding of $H_n$ into $\bH_n^{\trig}$ given as follows \cite[Proposition~4.1(i)]{Suzuki}:   
\begin{equation}
    \begin{aligned}\label{eqn: embedding of H_n}
        &s_k \mapsto s_k, \qquad X_i \mapsto X_i, \\
        &y_i \mapsto X_i^{-1}\left( \hy_i + c\sum_{\substack{j=1 \\ j > i}}^n s_{ij} \right) 
        = X_i^{-1} (s_i s_{i+1} \cdots s_{n-1}) \hy_n (s_{n-1} \cdots s_{i+1} s_i).
    \end{aligned}
\end{equation}
On the other hand, the degenerate \textit{affine} Hecke algebra that is the subalgebra of~$\bH_n^{\trig}$ generated by $\fS_n$ and $\C[\hy_1, \dots, \hy_n]$ embeds into $H_n$ via $s_k \mapsto s_k$, $\hy_i \mapsto X_i y_i -  c\sum_{j > i} s_{ij}$ \cite[Proposition~4.3(ii)]{EG}.

\subsection{Degree zero part of $H_n$}\label{sec: deg-0 part of RCA} 
The RCA $H_n$ admits a grading in which $\deg s_k = 0$, $\deg X_i = 1$, and $\deg y_i = -1$. The subalgebra $H^{\fgl_n}= H_n^{(0)}$ of those elements that have degree zero was studied in \cite{FH}. It is generated by the elements~$s_k$ and the products $E_{ij}\coloneqq X_iy_j$ ($i,j \in \{1, \dots, n\}$) with relations
\begin{align*}
    &s_k E_{ij} = E_{s_k(i), \, s_k(j)} s_k, \\
    &E_{ij} E_{kl} - E_{il} E_{kj} =  E_{il} S_{jk} - E_{ij} S_{kl},  \\
    &E_{ij} E_{kl} - E_{kj} E_{il} = S_{jk} E_{il} - S_{ij} E_{kl}. 
\end{align*}
Here $S_{ij}$ is given by formula~\eqref{eqn: S_ij}.
Equivalently, the third set of relations could be replaced by the commutator-type relations
\begin{equation*}
    [E_{ij}, E_{kl}] = E_{il}S_{jk} - S_{il}E_{kj} + [S_{kl}, E_{ij}].
\end{equation*}
The algebra $H_n$ has a filtration with $\deg X_i = \deg y_i = 1$ and $\deg s_k = 0$, which induces a filtration on the subalgebra $H^{\fgl_n}$.
The associated graded algebra is $\gr H^{\fgl_n} = \C \fS_n \ltimes \C[\cM]$, where
    \begin{equation*}
        \cM = \{ M \in \Mat_n(\C) \colon \rk M \leq 1 \}.
    \end{equation*}

The algebra $H^{\fgl_n}$ admits a PBW-type basis consisting of monomials
\begin{equation*}
     wE_{i_1 j_1}^{k_1} \cdots E_{i_r j_r}^{k_r},
\end{equation*}
where $w \in \fS_n$, $r \in \Z_{\geq 0}$, $k_u \in \Z_{>0}$, $1 \leq i_1 \leq \cdots \leq i_r \leq n$, and $1 \leq j_1  \leq \cdots \leq j_r \leq n$ with $i_u = i_{u+1} \Rightarrow j_u < j_{u+1}$. 
The algebra is a flat $c$-deformation of~$\C \fS_n \ltimes \rho_{\text{JS}}(U(\fgl_n))$, where $\rho_{\text{JS}}$ is the oscillator (also known as Jordan--Schwinger) representation of the universal enveloping algebra $U(\fgl_n)$ mapping the standard basis of the Lie algebra~$\fgl_n$ to the operators $X_i \partial_{X_j}$ ($i,j \in \{1, \dots, n\}$). 
The algebra $H^{\fgl_n}$ is an example of a non-homogeneous quadratic algebra over $\C\fS_n$ of PBW type (cf.~\cite{BGe}).

The element $eu = \sum_{i=1}^n X_i y_i - c\sum_{i < j} s_{ij}$ (which in the polynomial representation is equal up to a constant to the Euler operator $\sum_{i=1}^n X_i \partial_{X_i}$) generates the centre~$\cZ(H^{\fgl_n})$. 

The algebra~$\bH_n^{\trig}$ also has a grading, given by $\deg s_k = \deg \hy_i = 0$ and $\deg X_i^{\pm 1} = \pm 1$ \cite{BEF}. The embedding~\eqref{eqn: embedding of H_n} preserves the respective gradings. The degree zero part~$\bH_n^{\trig, \, (0)}$ is generated by the set of elements $s_k$, $\hy_i$, and $X_iX_j^{-1}$. The algebra~$H^{\fgl_n} = H_n^{(0)}$ embeds into~$\bH_n^{\trig, \, (0)}$ via a restriction of the mapping~\eqref{eqn: embedding of H_n}.

\section{Quantum group $U_q(\fgl_n)$} \label{sec: quantum group}
Let $q \in \C^\times$, not a root of unity. 
The quantum group $U_q(\fgl_n)$ is the (unital, associative) algebra over $\C$ generated by $g_i^{\pm 1}$ ($1 \leq i \leq n$) and $e_k$, $f_k$ ($1 \leq k \leq n-1$) subject to the following relations~\cite[p.~163--164]{KS} (cf.~\cite{Jimbo}):
\begin{align*}
    &g_i g_i^{-1} = g_i^{-1} g_i = 1, \qquad [g_i, g_j] = 0, \\
    &g_ie_kg_i^{-1} = q^{\delta_{ik} - \delta_{i, k+1}}e_k, \qquad g_i f_k g_i^{-1} = q^{\delta_{i, k+1} - \delta_{ik}}f_k,\\
    &[e_k, f_l] = \delta_{kl}\frac{g_k g_{k+1}^{-1} - g_k^{-1} g_{k+1}}{q - q^{-1}}, \\
    &[e_k, e_l] = 0 = [f_k, f_l] \textnormal{ if } |k-l| > 1, \\
    &e_k^2e_l - (q + q^{-1})e_k e_l e_k + e_l e_k^2 = 0 \textnormal{ if } |k-l| = 1,  \\
    &f_k^2f_l - (q + q^{-1})f_k f_l f_k + f_l f_k^2 = 0 \textnormal{ if } |k-l| = 1.
\end{align*}
Here $1 \leq j \leq n$, $1 \leq l \leq n-1$, and $\delta_{ij}$ is the Kronecker delta.
It follows from these relations that $\prod_{i=1}^n g_i$ lies in the centre $\cZ(U_q(\fgl_n))$ of this algebra, and \cite[Proposition~4]{Jimbo} provides some further central elements.

In this section, we recall a representation of $U_q(\fgl_n)$ on the space of Laurent polynomials $\C[X_1^{\pm 1},\dots, X_n^{\pm 1}]$. Let us firstly set up some notations.
Let~$t_i = q^{X_i \partial_{X_i}}$ be the~$q$-shift operator which acts on functions~$f$ by~$(t_if)(X_1, \dots, X_n) = f(X_1, \dots, qX_i, \dots, X_n)$. Let us also consider the following operator
\begin{equation}\label{eqn: d_i}
   d_i = \frac{X_i^{-1}(t_i - t_i^{-1})}{q-q^{-1}}.
\end{equation}
In the $q \to 1$ limit, it satisfies $d_i \to \partial_{X_i}$.

The next lemma collects some properties of $d_i$ and $t_i$, which can be checked by a direct computation. We use the notation $[a, b]_\lambda$ for $\lambda \in \C$ to mean $ab - \lambda ba$.
\begin{lemma}\label{lemma: properties of d_i at the quantum group level}
For all $i,j \in \{1, \dots, n\}$, we have
\begin{enumerate}[(1)]
    \item 
    $[t_i, X_j]_{q^{\delta_{ij}}} = 0$,
    \item $[d_i, d_j] = 0 = [t_i, t_j]$,
    \item $[d_i, t_j]_{q^{\delta_{ij}}} = 0$, 
    \item 
    $d_i X_i = (q-q^{-1})^{-1}(qt_i - q^{-1}t_i^{-1})$, 
    and $[d_i, X_j]_{q^{\pm \delta_{ij}}} = \delta_{ij} t_i^{\mp 1}$. 
\end{enumerate}
\end{lemma}

In terms of $d_i$, $t_i$, and the multiplication operators $X_i$, one can write down a representation of $U_q(\fgl_n)$ on $\C[X_1^{\pm 1}, \dots, X_n^{\pm 1}]$ as follows. It is called the Jordan--Schwinger or $q$-oscillator representation.

\begin{proposition}\textnormal{\textbf{\cite{Hayashi}}}
    There is a representation $\rho$ of $U_q(\fgl_n)$ given on the generators by $\rho \colon g_i^{\pm 1} \mapsto t_i^{\pm 1}$, $e_k \mapsto X_k d_{k+1}$, and $f_k \mapsto X_{k+1} d_k$.
\end{proposition}

This representation $\rho$ has a submodule $\C[X_1, \dots, X_n]$, which we revisit later in Section~\ref{sec: centre}.
The physical interpretation of the above is that $t_i$, $X_i$, $d_i$ represent the action of the generators of the $q$-Weyl algebra~\cite{Hayashi} (or $q$-boson algebra) on the Fock-like space $\C[X_1, \dots, X_n]$ for a system of $q$-bosons on a 1-dimensional lattice with $n$ sites. 
The operators $X_i$ and $d_i$ respectively create and annihilate a $q$-boson at site $i$. The preceding proposition realises an action of $U_q(\fgl_n)$ on this Fock space.

 Let us consider the algebra $A = \rho(U_q(\fgl_n))$, 
 the image 
 of $U_q(\fgl_n)$ under the map $\rho$: 
 $$
    A = \langle t_i^{\pm 1} \ (1 \leq i \leq n), \ X_k d_{k+1}, \ X_{k+1}d_k \ (1 \leq k \leq n-1) \rangle.
 $$ 
By the isomorphism theorems, $A$ is isomorphic to $U_q(\fgl_n)/I_q$ for~$I_q = \ker(\rho)$.  
We next describe the algebra $A$ abstractly by (a different set of) generators and relations.

Let $E_{ij}^q = X_i d_j$ ($i,j \in \{1, \dots ,n \}$).
Then~$E_{ii}^q = (q-q^{-1})^{-1}(t_i - t_i^{-1}) \in A$.
The operators~$E_{ij}^q$ for $|i-j| > 1$ are related to Jimbo's analogue of the non-simple root vectors of $\fgl_n$ from \cite[Proposition~1]{Jimbo}. The following formulae hold for all $1 \leq i<j \leq n-1$ \cite[(3.3)]{DKKV}: 
     \begin{align*}
        &E_{i,j+1}^q = \rho\left([e_i,[e_{i+1}, \cdots [e_{j-1},e_j]_q\cdots]_q]_q \ g_{i+1} g_{i+2} \cdots g_j\right), \\
        &E_{j+1, i}^q = \rho\left([f_j,[f_{j-1},\cdots[f_{i+1},f_i]_{q^{-1}}\cdots]_{q^{-1}}]_{q^{-1}}g_{i+1}^{-1} \ g_{i+2}^{-1} \cdots g_j^{-1}\right).
    \end{align*}
It follows that $E_{ij}^q \in A$ for all $i, j$, and
$A = \langle t_i^{\pm 1}, \  E_{ij}^q \ (i \neq j) \rangle$ as an algebra.

In the $q \to 1$ limit, the representation $\rho$ recovers the oscillator representation of the universal enveloping algebra~$U(\fgl_n)$, since $\lim_{q \to 1} E_{ij}^q = X_i \partial_{X_j}$ for all $i, j$.

The next two propositions describe relations satisfied by the generators~$E_{ij}^q$ and $t_i$. 
Let us introduce the notation 
$$
    S_{ij}^q \coloneqq [d_i, X_j] = \begin{cases}
         (q+1)^{-1}\left(qt_i + t_i^{-1}\right) \textnormal{ if } i = j, \\
         0 \textnormal{ if } i \neq j,
    \end{cases}
$$
where the second equality follows from Lemma~\ref{lemma: properties of d_i at the quantum group level}~(4).
Then the following proposition is a straightforward consequence of the definitions 
and the commutativity of the operators $d_i$. (It will also follow from the $\tau = 1$ limit of the more general discussions presented in the next section.) 
\begin{proposition}\label{prop: relations among E_ij^q}
For all $i,j,k,l \in \{1, \dots, n\}$,
\begin{equation}\label{swapping j,l or i,k}
    \begin{aligned}
        &E_{ij}^qE_{kl}^q - E_{il}^qE_{kj}^q =  E_{il}^q S_{jk}^q - E_{ij}^q S_{lk}^q,   \\
        &E_{ij}^qE_{kl}^q - E_{kj}^qE_{il}^q = S_{jk}^q E_{il}^q - S_{ji}^q E_{kl}^q. 
    \end{aligned}
\end{equation}
\end{proposition}

The following statement holds as a result of Lemma~\ref{lemma: properties of d_i at the quantum group level} (1) and (3).
\begin{proposition}\label{prop: relations between E_ij^q and t_i}
    For all $i,j,k \in \{1, \dots, n\}$, 
    \begin{equation}\label{moving t_i to the right}
        t_i E_{jk}^q t_i^{-1} = q^{\delta_{ij} - \delta_{ik}}E_{jk}^q.
    \end{equation}
\end{proposition}

The preceding two propositions lead to a PBW-type basis and a presentation for the algebra $A$.
\begin{proposition}\label{prop: PBW basis for A}
A linear basis of the algebra $A$ is formed by elements 
\begin{align}\label{eqn: PBW basis for A}
        (E_{i_1 j_1}^q)^{k_1} \cdots (E_{i_r j_r}^q)^{k_r} \prod_{l = 1}^n  t_l^{m_l}, 
\end{align}
where $r \in \Z_{\geq 0}$, $k_u \in \Z_{>0}$, $m_l \in \Z$, $1 \leq i_1 \leq \cdots \leq i_r \leq n$, $1 \leq j_1  \leq \cdots \leq j_r \leq n$ with $i_u = i_{u+1} \Rightarrow j_u < j_{u+1}$, and none of the indices $i_u$ equal any of the indices $j_v$. 

The algebra $A$ has a presentation by generators~$t_i^{\pm 1}$, $E_{ij}^q \ (i \neq j)$ and relations \eqref{swapping j,l or i,k} with $i \neq j$ and $k \neq l$, \eqref{moving t_i to the right} with~$j \neq k$, and the Laurent relations for $t_i^{\pm 1}$, namely $t_i t_i^{-1} = t_i^{-1} t_i = 1$ and $[t_i, t_j] = 0$ for all $i,j$.
\end{proposition}
\begin{proof}
    It follows from relations~\eqref{swapping j,l or i,k} and~\eqref{moving t_i to the right} that any element of $A$ can be written as a linear combination of elements of the form~\eqref{eqn: PBW basis for A}, thus they span~$A$. We now show that they are linearly independent over $\C$ as operators on~$\C[X_1^{\pm 1}, \dots, X_n^{\pm 1}]$.

    For any $k \in \Z_{> 0}$ and $i \neq j$, by using Lemma~\ref{lemma: properties of d_i at the quantum group level} we get
    \begin{align*}
        (E_{ij}^q)^k =(q-q^{-1})^{-k} X_i^k X_j^{-k} \prod_{l = 0}^{k-1}(q^{-l}t_j- q^l t_j^{-1}).
    \end{align*}
    More generally, for elements of the form~\eqref{eqn: PBW basis for A} we have 
    \begin{equation*}
        (E_{i_1 j_1}^q)^{k_1} \cdots (E_{i_r j_r}^q)^{k_r} \prod_{l = 1}^n  t_l^{m_l} \sim X_{i_1}^{k_1} \cdots X_{i_r}^{k_r}X_{j_1}^{-k_1} \cdots X_{j_r}^{-k_r}t^{k_1}_{j_1} \cdots t^{k_r}_{j_r} \prod_{l = 1}^n  t_l^{m_l} + \, \dots,
    \end{equation*}
    where $\dots$ denotes terms in which the overall sum of the exponents on the~$t_i$'s is lower than in the above leading term, and $\sim$ denotes proportionality by a non-zero factor, which may depend on $q$. 

    Assume a non-trivial linear dependence of some terms of the form~\eqref{eqn: PBW basis for A}. This implies a non-trivial linear dependence of their corresponding leading (with highest degree in $t_i$'s) terms 
    \begin{equation}\label{leading terms}
        X_{i_1}^{k_1} \cdots X_{i_r}^{k_r}X_{j_1}^{-k_1} \cdots X_{j_r}^{-k_r}t^{k_1}_{j_1} \cdots t^{k_r}_{j_r} \prod_{l = 1}^n  t_l^{m_l}.
    \end{equation}
    By the assumptions on the indices of the monomials~\eqref{eqn: PBW basis for A}, their leading terms~\eqref{leading terms} are different, and
    since $\prod_{l=1}^n X_l^{n_l} \prod_{l=1}^n t_l^{n'_l}$ ($n_l, n'_l \in \Z$) are linearly independent over $\C$ as operators on $\C[X_1^{\pm 1}, \dots, X_n^{\pm 1}]$ (we are using here that~$q$ is not a root of unity), we get a contradiction. The statement follows.
\end{proof}

The above basis allows for the following proof of what the centre of $A$ is.
\begin{proposition}\label{prop: centre of A}
    The centre $\cZ(A)$ is generated by $(\prod_{i=1}^n t_i)^{\pm 1}$.
\end{proposition}
\begin{proof}
    For the monomial~\eqref{eqn: PBW basis for A}, 
    let $\widetilde r = \min (\{ u \in \{1, \dots, r-1\}\colon i_u \neq i_{u+1}\} \cup \{r\})$, that is,
    $i_1 = i_2 = \cdots = i_{\widetilde r} < i_{\widetilde r + 1} \leq \cdots \leq i_r$. Then
    \begin{equation*}
        t_{i_1}(E_{i_1 j_1}^q)^{k_1} \cdots (E_{i_r j_r}^q)^{k_r} \left(\prod_{l = 1}^n  t_l^{m_l}\right) t_{i_1}^{-1} = q^{k_1 + k_2 + \cdots + k_{\widetilde r}} (E_{i_1 j_1}^q)^{k_1} \cdots (E_{i_r j_r}^q)^{k_r} \prod_{l = 1}^n  t_l^{m_l}.
    \end{equation*}
    Together with our assumption that $q$ is not a root of unity, this implies for any~$f \in \cZ(A)$ that its expansion in the PBW basis from Proposition~\ref{prop: PBW basis for A} cannot involve any basis elements for which $r > 0$.

    Similarly, since for all $1 \leq k \leq n-1$ we have
    \begin{equation*}
       \left(\prod_{l = 1}^n  t_l^{m_l}\right) E_{k, k+1}^q = q^{m_k - m_{k+1}} E_{k, k+1}^q\prod_{l = 1}^n  t_l^{m_l},
    \end{equation*}
    we get that the expansion of $f$ can contain only terms of the form $\prod_{l = 1}^n  t_l^{m_l}$ where all $m_l$ are equal. Conversely, all such terms do belong to the centre. The statement follows.
\end{proof}

Additionally, let us consider the crossed (equivalently, semi-direct or smash) product algebra $\cA \coloneqq \C\fS_n \ltimes A$. As a vector space, $\cA\cong \C\fS_n\otimes A$ with the algebra structure defined by the natural action of the symmetric group~$\fS_n$ on $A$ 
given by
 \begin{equation}\label{eqn: moving s_k to the right}
     s_kt_i^{\pm 1} = t_{s_k(i)}^{\pm 1}s_k, \qquad s_k E_{ij}^q = E_{s_k(i), s_k(j)}^qs_k,
 \end{equation}
$1 \leq k \leq n-1$, where $s_k = (k, k+1) \in \fS_n$.
This action is well-defined as it preserves the defining relations of~$A$ given in Proposition~\ref{prop: PBW basis for A}.
The algebra~$\cA$ has a presentation analogous to that of $A$, just with the extra generators $s_k$ and the extra relations~\eqref{eqn: moving s_k to the right} along with the Coxeter relations among $s_k$ that hold in $\fS_n$. 
A basis of PBW type for $\cA$ consists of elements of the form 
\begin{align}\label{eqn: PBW basis for cA}
        w(E_{i_1 j_1}^q)^{k_1} \cdots (E_{i_r j_r}^q)^{k_r} \prod_{l = 1}^n  t_l^{m_l},  \qquad\qquad (w \in \fS_n)
\end{align}
with the same restrictions on the indices as above in Proposition~\ref{prop: PBW basis for A}.

\begin{proposition}\label{prop: centre of cA}
    The centre of $\cA$ satisfies $\cZ(\cA) = \cZ(A) = \langle (\prod_{i=1}^n t_i)^{\pm 1} \rangle$.
\end{proposition}
\begin{proof}
    Since elements of $\cZ(A)$, described in Proposition~\ref{prop: centre of A}, are $\fS_n$-invariant, we have $\cZ(A) \subseteq \cZ(\cA)$. Since $A \subset \cA$, we have $\cZ(\cA) \cap A \subseteq \cZ(A)$. 
It is now sufficient to show that     $\cZ(\cA) \subseteq A$.

    Denote the elements of the basis~\eqref{eqn: PBW basis for cA} schematically as $wET$ where $E = (E_{i_1 j_1}^q)^{k_1} \cdots (E_{i_r j_r}^q)^{k_r}$ and $T = \prod_{l = 1}^n  t_l^{m_l}$. We have $t_iwET \sim wETt_{w^{-1}(i)}$, where the proportionality factor is a power of $q$. For any~$f \in \cZ(\cA)$, in its expansion in the basis~\eqref{eqn: PBW basis for cA} let us group terms that have the same~$w$ and~$E$ parts. Each such group has the form $wE\sum_k \a_k T^{(k)}$, where $\a_k \in \C$ and $T^{(k)} = \prod_{l = 1}^n  t_l^{m_{kl}}$. Any~$t_i$ has to commute individually with each of the groups. Suppose $w \neq \id$. Take $i, j$ such that $j = w^{-1}(i) \neq i$. We have $t_jE = q^aEt_j$ for some $a$. Commutativity requires
    \begin{equation*}
        wE\sum_k \a_k T^{(k)}t_i = t_i wE \sum_k \a_k T^{(k)} = q^a wE \sum_k \a_k T^{(k)} t_j,
    \end{equation*}
    or equivalently, $wE (q^a - t_it_j^{-1}) \sum_k \a_k T^{(k)} = 0$, which forces $\sum_k \a_k T^{(k)} = 0$.

    This showed for any~$f \in \cZ(\cA)$ that its expansion in the basis~\eqref{eqn: PBW basis for cA} cannot involve any basis elements for which $w \neq \id$, as was required.
\end{proof}

In the next section, we define inside the DAHA of type $GL_n$ a subalgebra that deforms the algebra $\cA$ in a natural way.

\section{A subalgebra of a double affine Hecke algebra} \label{sec: subalgebra}
\subsection{DAHA of type $GL_n$}\label{sec: DAHA}
We start by recalling the definition of the DAHA of type $GL_n$. 
Let~$\tau$ be a formal parameter, and $q \in \C^\times$ not a root of unity. Let $\C_\tau$ $=\C[\tau^{\pm 1}]$ be the ring of Laurent polynomials in the variable $\tau$. The DAHA~$\bH_n = \bH_{n, q,\tau}$ of type~$GL_n$ is the (unital, associative) $\C_\tau$-algebra generated by $T_k$ ($1 \leq k \leq n-1$), $X_i^{\pm 1}$, and $Y_i^{\pm 1}$ ($1 \leq i \leq n$) with the following relations~\cite[p.~100]{Cherednik}:
\begin{align}
    &(T_k-\tau)(T_k +\tau^{-1}) = 0, \label{eqn: Hecke relations} \\
    &T_l T_{l+1} T_l = T_{l+1} T_l T_{l+1} \quad (1 \leq l \leq n-2), \quad \ [T_k, T_l] = 0 \textnormal{ if } |k-l| > 1, \label{eqn: braid relations} \\
    &T_kX_kT_k = X_{k+1}, \qquad [T_k, X_i] = 0 \textnormal{ for } i \neq k, k+1, \label{eqn: T_k and X_i}\\
    &T_k^{-1} Y_k T_k^{-1} = Y_{k+1}, \quad \, [T_k, Y_i] = 0 \textnormal{ for } i \neq k, k+1, \label{eqn: T_k and Y_i}\\
    &\tY X_i = q X_i \tY,  \nonumber \\ 
    &Y_2^{-1} X_1 Y_2 X_1^{-1} = T_1^2,  \nonumber \\
    &\textnormal{and the Laurent relations for } X_i^{\pm 1}, \ Y_i^{\pm 1} \textnormal{ (that is, } X_iX_i^{-1} = X_i^{-1}X_i = 1, \nonumber \\
    &[X_i, X_j] = 0, \ 1 \leq j \leq n; \textnormal{ similarly for } Y_i^{\pm 1} \textnormal{)}, \nonumber 
\end{align}
where $\tY = \prod_{i=1}^n Y_i$. Relations~\eqref{eqn: T_k and Y_i} imply that $\tY$ commutes with all $T_k$, which generate a subalgebra isomorphic to the Hecke algebra of type $A_{n-1}$.

The DAHA $\bH_n$ admits a grading in which $\deg T_k = \deg Y_i^{\pm 1} = 0$ and $\deg X_i^{\pm 1} = \pm 1$. The degree zero part $\bH_n^{(0)}$ is generated by the set of elements~$T_k$, $Y_i^{\pm 1}$, and $X_i X_j^{-1}$.

As in \cite[(1.4.57)]{Cherednik}, let $\pi = Y_1^{-1} T_1 \cdots T_{n-1}$. Relations~\eqref{eqn: T_k and Y_i} imply that
\begin{equation}\label{eqn: Y_i in terms of pi}
    Y_i = T_i T_{i+1} \cdots T_{n-1} \pi^{-1} T_1^{-1}T_2^{-1} \cdots T_{i-1}^{-1}
\end{equation}
for all $1 \leq i \leq n$ (for $i=1$ and $i=n$, this is to be interpreted as $Y_1 = T_1 \cdots T_{n-1} \pi^{-1}$ and $Y_n = \pi^{-1} T_1^{-1} \cdots T_{n-1}^{-1}$, respectively). 

The algebra $\bH_n$ admits a faithful representation, known as the polynomial representation, on the space of Laurent polynomials $\C_\tau[X_1^{\pm 1}, \dots, X_n^{\pm 1}]$. This representation is determined by
\begin{align}
    &T_k \mapsto \tau s_k + \frac{\tau - \tau^{-1}}{X_kX_{k+1}^{-1} - 1}(s_k - 1), \label{eqn: action of T_k} \\
    &\pi^{-1}(X_1^{a_1} X_2^{a_2} \cdots X_n^{a_n}) = q^{a_1}X_1^{a_2} \cdots X_{n-1}^{a_n} X_n^{a_1} \qquad (a_i \in \Z),
    \label{eqn: action of pi}
\end{align}
and the action of $X_i^{\pm 1}$ by multiplication \cite[p.~101]{Cherednik}. Thus at $\tau = 1$, the elements $T_k$ act as~$s_k \in \fS_n$; and by equality~\eqref{eqn: Y_i in terms of pi} and formula~\eqref{eqn: action of pi}, the elements $Y_i$ act at $\tau = 1$ as the $q$-shift operators $t_i$ from Section~\ref{sec: quantum group}. 
 
The trigonometric degeneration of $\bH_n$, obtained by putting 
\begin{equation}\label{eqn: trigonometric degeneration}
    Y_i = e^{\hbar \widehat{y}_i}, \quad q = e^\hbar, \quad \tau = e^{-\hbar c/2}, \quad T_k = s_ke^{-\hbar c s_k/2}
\end{equation}
and letting $\hbar \to 0$ (equivalently $q \to 1$), recovers the trigonometric Cherednik algebra $\bH_n^{\trig}$ of type $GL_n$ from Section~\ref{sec: RCA}.

In the next subsection, we describe a subalgebra, which we denote~$\bH^{\fgl_n}$. As we explain, this subalgebra is a $q$-deformation of the degree zero part~$H^{\fgl_n}$ of the RCA of type $GL_n$, and it is a $\tau$-deformation of the algebra $\cA \cong \C\fS_n \ltimes (U_q(\fgl_n)/I_q)$ from Section~\ref{sec: quantum group}.

\subsection{Subalgebra $\bH^{\fgl_n}$} \label{sec: Hgln}
We will use throughout the following shorthand notations
\begin{align*}
    &T_{ij}^+ = \begin{cases}
        T_i T_{i+1} \cdots T_j \textnormal{ if } i \leq j, \\
        1 \textnormal{ if } i > j,
    \end{cases}
    \qquad 
    (T^{-1})_{ij}^+ = \begin{cases}
        T_i^{-1} T_{i+1}^{-1} \cdots T_j^{-1} \textnormal{ if } i \leq j, \\
        1 \textnormal{ if } i > j,
    \end{cases} \\
    &T_{ij}^- = \begin{cases}
        T_i T_{i-1} \cdots T_j \textnormal{ if } i \geq j, \\
        1 \textnormal{ if } i < j,
    \end{cases}
    \qquad 
    (T^{-1})_{ij}^- = \begin{cases}
        T_i^{-1} T_{i-1}^{-1} \cdots T_j^{-1} \textnormal{ if } i \geq j, \\
        1 \textnormal{ if } i < j,
    \end{cases} 
    \\
            &(\cR^\e)_{ij}^{\pm} \coloneqq \begin{cases}
                (T^\e)_{i-1, j+1}^- T_j^{\pm 2} (T^{-\e})_{j+1, i-1}^+ \textnormal{ if } i > j, \\
            1 \textnormal{ if } i \leq  j,
            \end{cases}
            \textnormal{ with } \e \in \{1, -1\}. 
\end{align*}
We write $\cR$ for $\cR^1$ and $T$ for $T^1$. We note that $T_{ij}^+ (T^{-1})_{ji}^- = 1 = T^-_{ij}(T^{-1})^+_{ji}$. In particular, in the definition of $(\cR^\e)_{ij}^{\pm}$, the factors $(T^\e)_{i-1, j+1}^-$ and $(T^{-\e})_{j+1, i-1}^+$ are inverses of each other when $\e$ is fixed, and
hence $(\cR^\e)_{ij}^- (\cR^\e)_{ij}^+ = 1$.

Let us also note that relations~\eqref{eqn: T_k and X_i}, \eqref{eqn: T_k and Y_i} of the DAHA imply the following (and similar) identities, used in several places throughout the text.
\begin{lemma}\label{lem: DAHA identity}
 The identities
\begin{align*}
    &T^+_{i,n-1} X_n^{-1} = X_i^{-1} (T^{-1})^+_{i,n-1}, \quad \quad \quad  T^+_{i,n-1} Y_n = Y_i (T^{-1})^+_{i,n-1}, \\[0.4em]
    &X_n^{-1} T^-_{n-1,i}  = (T^{-1})^-_{n-1,i} X_i^{-1}, \quad \quad \quad  Y_n T^-_{n-1,i} = (T^{-1})^-_{n-1,i} Y_i 
\end{align*}
hold in the DAHA $\bH_n$.
\end{lemma}

Let $D_n = (q-q^{-1})^{-1}X_n^{-1}(Y_n - Y_n^{-1})$, and let 
\begin{equation}\label{eqn: definition of D_i}
    D_i = T_{i, n-1}^+ D_n T_{n-1, i}^- = 
(q-q^{-1})^{-1}X_i^{-1}(T^{-1})_{i, n-1}^+ (Y_n - Y_n^{-1})T_{n-1, i}^-
\end{equation}
for $1 \leq i \leq n-1$.
With the assignments~\eqref{eqn: trigonometric degeneration}, upon performing the trigonometric degeneration~$q \to 1$,
we get $D_i \to y_i$, where we implicitly use the embedding~\eqref{eqn: embedding of H_n}. 
At~$\tau = 1$, the elements $D_i$ act in the polynomial representation of the DAHA as the operators $d_i$ from Section~\ref{sec: quantum group}.

Let  
$e_{ij}= X_iD_j$ ($i,j \in \{1, \dots, n\}$). We now define the main object of this paper.
Inside $\bH_n$, we define~$\bH^{\fgl_n} = \bH^{\fgl_n}_{q, \tau}$ as the following subalgebra:
\begin{equation*}
    \bH^{\fgl_n} = \langle T_k, \, Y_i^{\pm 1}, \, e_{ij} \colon 1 \leq k \leq n-1, \, 1 \leq i,j \leq n, \, i \neq j \rangle \subset \bH_n.
\end{equation*}
Note that, by equality~\eqref{eqn: definition of D_i}, we get
\begin{equation*}
    e_{ii} = (q-q^{-1})^{-1}(T^{-1})_{i, n-1}^+ (Y_n - Y_n^{-1})T_{n-1, i}^- \in \bH^{\fgl_n}.
\end{equation*}

At $\tau = 1$, the generators
$T_k, \, Y_i^{\pm 1}$, and $e_{ij}$ 
of $\bH^{\fgl_n}$ act (in the polynomial representation) respectively as $s_k$, $t_i^{\pm 1}$, and $E_{ij}^q$, which generate the algebra $\cA$ from Section~\ref{sec: quantum group}.
In the trigonometric limit $q \to 1$, we get $T_k \to s_k$,  $Y_i^{\pm 1}\to 1$, and $e_{ij} \to X_iy_j$ for all $i,j$, which are the generators of $H^{\fgl_n}$.

We note that $\bH^{\fgl_n} \subset \bH_n^{(0)}$, and $\bH^{\fgl_n} \neq \bH_n^{(0)}$. 
Indeed, in the limit $q \to 1$ we do not get, for example, the elements $X_iX_j^{-1}$ for $i \ne j$.

In the next remark, we explain that the algebra $\bH^{\fgl_n}$ is isomorphic to a subalgebra of a cyclotomic DAHA introduced in~\cite{BEF}.

\begin{remark}\label{rem: BEF's D_i}
Elements similar to but different from $D_i$ appear in the definition of the cyclotomic DAHA $\HH_{n,t}^l(Z,q^{-1})$ for $l=2$, $Z_1 = 1$, $Z_2 = -1$, $Z = (Z_1, Z_2)$ \cite[Section~3.6]{BEF}, where we assume that $t$ is a formal parameter and $q$ is numerical. Let us make the relation more precise. The following elements~$D_i^{\text{BEF}} \equiv D_i^{(2)}$ were considered in \cite{BEF}: 
\begin{equation}\label{eqn: D_i^(2)}
    D_i^{\text{BEF}} = (T^{-1})_{i-1,1}^- X_1^{-1}(Y_1^2 - 1)(T^{-1})_{1,i-1}^+.
\end{equation}
The DAHA $\HH_{n,t}(q)$ considered in \cite{BEF} is isomorphic to the DAHA~$\bH_n$ considered in this paper via an isomorphism $g \colon \HH_{n,t}(q) \to \bH_n$ given by 
\begin{equation*}
    g(T_k) = T_k, \quad g(X_i) = Y_i^{-1}, \quad g(Y_i) = X_i, \quad g(\mathbf{t}) = \tau,
\end{equation*}
where $t = \mathbf{t}^2$, and $q$ from \cite{BEF} corresponds to our~$q$.
According to \cite{BEF}, there is an isomorphism $\varphi\colon \HH_{n,t}(q^{-1}) \to \HH_{n,t}(q)$ given by
\begin{equation*}
    \varphi(T_k) = T_k^{-1}, \quad \varphi(X_i) = Y_i^{-1}, \quad 
    \varphi(Y_i) = X_i^{-1}, \quad  
    \varphi(\mathbf{t}) = \mathbf{t}^{-1}.
\end{equation*}
Also, it is straightforward to check that the DAHA $\bH_n$ has an automorphism~$h$ given by 
\begin{equation*}
    h(T_k) = T_{n-k}, \quad h(X_i) = X_{n-i+1}^{-1}, \quad h(Y_i) = Y_{n-i+1}^{-1}, \quad h(\tau) = \tau.
\end{equation*}
By combining these isomorphisms and applying them to $D_{n-i+1}^{\text{BEF}} \in \HH_{n,t}^l(Z,q^{-1})$, we get
 \begin{align}
    (h \circ g \circ \varphi)(D_{n-i+1}^{\text{BEF}}) &= T_{i, n-1}^+ X_n^{-1}(Y_n^{-1} - Y_n)Y_n^{-1}T_{n-1, i}^-  \nonumber \\
    &=  (q^{-1}-q) D_i (T^{-1})^+_{i, n-1}Y_n^{-1}T_{n-1, i}^- \nonumber \\
    &= (q^{-1}-q) D_i  Y_i^{-1} T_{i,n-1}^+
    T_{n-1, i}^-, \label{eqn: D_i vs D_i^(2)}
\end{align}
where we used Lemma~\ref{lem: DAHA identity}.

It follows that $(\varphi^{-1} \circ g^{-1} \circ h)(D_i) \in \HH_{n,t}^2((1,-1), q^{-1})$. This implies that $\bH^{\fgl_n}$ is isomorphic to a subalgebra of $\HH_{n,t}^2((1,-1), q^{-1}) \subset \HH_{n,t}^1(1, q^{-1})$. 

To summarise, we have the following commutative diagram:

\vspace{1.5em}
                \begin{minipage}{.6\linewidth}
                     \begin{tikzcd}[ampersand replacement=\&] 
                           \HH_{n,t}(q^{-1}) \dar{D_{n-i+1}^{\text{BEF}}} \rar{\varphi} \& \HH_{n,t}(q) \rar{g} \& \bH_n \rar{h} \& \bH_n  \dar{(q^{-1}-q) D_i  Y_i^{-1} T_{i,n-1}^+T_{n-1, i}^-} \\ 
                            \HH_{n,t}(q^{-1}) \rar{\varphi} \& \HH_{n,t}(q) \rar{g} \& \bH_n \rar{h} \& \bH_n
                     \end{tikzcd}
                \end{minipage}
\vspace{1.5em}

\noindent where the downward arrows denote action by multiplication on the left.
\hfill \qedsymbol
\end{remark}

The choice~\eqref{eqn: definition of D_i} of the elements $D_i$ is needed in order to be able to make the connection of the subalgebra $\bH^{\fgl_n}$ with the quantum group $U_q(\fgl_n)$. We now derive some properties of $D_i$ for later use.

\subsection{Properties of $D_i$}\label{sec: D}
We begin by some technical preliminaries. 
The following lemma and its corollary record some braid group identities.
\begin{lemma}\label{lemma: braid identities}
For all $n-1 \geq k > j\geq i \geq 1$ and $\varepsilon \in \{\pm 1\}$, we have
\begin{align*}
    T_{j+1}^{\varepsilon} T_{ik}^+  = T_{ik}^+ T_j^{\varepsilon}, \qquad  
    T_{ki}^- T_{j+1}^{\varepsilon}  = T_j^{\varepsilon} T_{ki}^-.
\end{align*}
\end{lemma}
\begin{proof}
    By using the braid relations, we compute
    \begin{equation*}
        T_{j+1}^{\varepsilon} T_{ik}^+ = T_{i, j-1}^+ T_{j+1}^{\e} T_j T_{j+1} T_{j+2, k}^+ = T_{i, j-1}^+ T_j T_{j+1} T_j^{\e} T_{j+2, k}^+ 
        = T_{ik}^+ T_j^{\varepsilon},
    \end{equation*}
    as required. Similarly for the other relation.
\end{proof}

The following is a straightforward corollary of the preceding lemma.
\begin{corollary}\label{cor: braid identities}
For all $n \geq j > i \geq 1$ and $\varepsilon \in \{\pm 1\}$, we have
    \begin{enumerate}[(i)]
        \item $(T^{\e})_{j, n-1}^+ T_{i, n-1}^+ = T_{i, n-1}^+(T^{\e})_{j-1, n-2}^+$,
        \item $(T^{\e})_{n-1, j}^- T_{i, n-1}^+ = T_{i, n-1}^+(T^{\e})_{n-2, j-1}^-$,
        \item $T_{n-1, i}^- (T^\e)_{n-1, j}^- = (T^\e)_{n-2, j-1}^-T_{n-1, i}^-$,
        \item $T_{n-1, i}^- (T^\e)_{j, n-1}^+ = (T^\e)_{j-1, n-2}^+T_{n-1, i}^-$.
    \end{enumerate}
\end{corollary}

Also, Corollary~\ref{cor: braid identities} (i) and (iii) (with $i=1$, $j=2$, $\e = 1$) respectively imply the following two useful identities.
\begin{corollary}\label{cor: braid identities 2}
    We have that
    \begin{align*}
        &(T^+_{1,n-1})^2 = T^+_{2,n-1}T^+_{1,n-2} T_{n-1}^2, \\
        &(T^-_{n-1,1})^2 = T^2_{n-1} T^-_{n-2,1}T^-_{n-1,2}.
    \end{align*}
\end{corollary}

The next lemma gives some identities for $(\cR^\e)_{ji}^{\pm}$ in the Hecke algebra.
\begin{lemma}\label{lemma: Hecke algebra identity}
    For all $n \geq j > i \geq 1$ and $\varepsilon \in \{\pm 1\}$, we have
    $$(\cR^\e)_{ji}^{\pm} = (T^{-\e})_{i,j-2}^+ T_{j-1}^{\pm 2} (T^\e)_{j-2, i}^-.$$
\end{lemma}
\begin{proof}
    Let $\e = -1$.  
    The claim trivially holds if $j = i+1$, so let $j>i+1$. We want to show for $\delta \in \{\pm 1\}$ that   
        \begin{align}\label{equivalent claim}
            T_i^{2\delta} T_{i+1, j-1}^+ T_{i, j-2}^+ = T_{i+1, j-1}^+ T_{i, j-2}^+T_{j-1}^{2\delta}.
        \end{align}
        Since $T_i^{2\delta} = 1 + \delta (\tau - \tau^{-1}) T_i^{\delta}$, the left-hand side of equality \eqref{equivalent claim} equals 
        \begin{equation}
            T_{i+1, j-1}^+ T_{i, j-2}^+ + \delta (\tau - \tau^{-1}) T_i^{\delta} T_{i+1, j-1}^+ T_{i, j-2}^+.   \label{intermediate step}
        \end{equation}
        By the braid relations, for any $n \geq j > l > i \geq 1$, we have
        \begin{equation*}
            (T_{i, l-2}^+T_{l-1}^\delta T_{l, j-1}^+)T_{l-1} = T_l(T_{i, l-1}^+T_l^\delta T_{l+1, j-1}^+),
        \end{equation*}
        (proved similarly to Lemma~\ref{lemma: braid identities}) which upon repeated application (for $l=i+1, \dots, j-1$) gives that $T_i^{\delta} T_{i+1, j-1}^+ T_{i, j-2}^+ = T_{i+1, j-1}^+ T_{i, j-2}^+ T_{j-1}^{\delta}$. Hence the expression~\eqref{intermediate step} equals
        \begin{equation*}
            T_{i+1, j-1}^+ T_{i, j-2}^+ + \delta (\tau - \tau^{-1}) T_{i+1, j-1}^+ T_{i, j-2}^+ T_{j-1}^{\delta} =  T_{i+1, j-1}^+ T_{i, j-2}^+T_{j-1}^{2\delta},
        \end{equation*}
        as required. The case when $\e = 1$ can be proved similarly.
\end{proof}

The next lemma is an analogue of relations~\eqref{eqn: T_k and X_i} and \eqref{eqn: T_k and Y_i} for $T_k$ and $D_i$.
\begin{lemma}\label{lemma: T_k and D_i}
    We have $[T_k, D_i] = 0$ for $i \neq k, k+1$, and $T_k^{-1}D_k T_k^{-1} = D_{k+1}$.
\end{lemma}
\begin{proof}
    That $T_k^{-1}D_k T_k^{-1} = D_{k+1}$ is clear from the definition. If $i \neq k$, $k+1$, then either $i \geq k+2$, in which case $[T_k, D_i] = 0$ because $T_k$ commutes with $X_n^{-1}$, $Y_n^{\pm 1}$, and with both $T_{i,n-1}^+$ and $T_{n-1, i}^-$; or $i < k$, in which case
    $T_k T_{i,n-1}^+ = T_{i,n-1}^+ T_{k-1}$ and $T_{k-1}T_{n-1, i}^- = T_{n-1, i}^-T_k$
    by Lemma~\ref{lemma: braid identities}, and so $[T_k, D_i]=0$ follows, as~$[T_{k-1}, X_n^{-1}] =[T_{k-1}, Y_n^{\pm 1}] = 0$.
\end{proof}

The following lemma is a $\tau$-deformed version of Lemma~\ref{lemma: properties of d_i at the quantum group level}.
\begin{lemma} 
\label{lemma: properties of D_i}
    The following relations are satisfied. 
    \begin{enumerate}[(1)]
        \item (Relations between $Y_i$ and $X_j$) For $n \geq i \neq j \geq 1$, $ n \geq l \geq 1$, we have
            \begin{align*}
                &Y_i \cR_{ij}^+ X_j 
                = X_j (\cR^{-1})_{ji}^- Y_i, \\
                &Y_lX_l = q T_{l, n-1}^+ T_{n-1, l}^- X_lY_l T_{l-1, 1}^- T_{1,l-1}^+.
            \end{align*}
        \item $[D_i, D_j] = 0$ for all $i, j$. 
        \item (Relations between $Y_i$ and $D_j$) For $n \geq i \neq j \geq 1$, $ n \geq l \geq 1$, we have
            \begin{align*}
                &Y_iD_j 
                = (\cR^{-1})_{ji}^+ D_jY_i\cR_{ij}^+, \\
                &Y_l T_{l-1, 1}^- T_{1,l-1}^+ D_l = q^{-1} D_l (T^{-1})_{l, n-1}^+ (T^{-1})_{n-1, l}^- Y_l.
            \end{align*}
        \item (Relations between $X_i$ and $D_j$) For $n \geq l \geq 1$, we have
            \begin{align}
                &X_l D_l = (q-q^{-1})^{-1} (T^{-1})_{l, n-1}^+ (Y_n - Y_n^{-1})T_{n-1, l}^-, \label{eqn: X_l D_l} \\
                &D_l X_l = (q-q^{-1})^{-1} (T^{-1})_{l-1,1}^- (q Y_1 - q^{-1}Y_1^{-1})T_{1,l-1}^+. \label{eqn: D_l X_l}
            \end{align}
            For $n \geq j > i \geq 1$, we have 
            \begin{align}
                &[D_j, X_i] = \frac{\tau^{-1} - \tau}{q-q^{-1}} (T^{-1})_{i, n-1}^+ (T^{-1})_{j-2, 1}^- (qY_1 + Y_n^{-1})T_{n-1,j}^- T_{1, i-1}^+, \label{eqn: [D_j, X_i]} \\
                &[D_i, X_j] =  \frac{\tau^{-1} - \tau}{q-q^{-1}}(T^{-1})_{i-1,1}^- (T^{-1})_{j, n-1}^+ (q^{-1}Y_1^{-1} + Y_n)T_{1,j-2}^+ T_{n-1, i}^-. \label{eqn: [D_i, X_j]}
            \end{align}
    \end{enumerate}
\end{lemma}

The trigonometric degeneration $q \to 1$ of the above relations 
\eqref{eqn: X_l D_l}--\eqref{eqn: [D_i, X_j]} recovers the commutator relation~\eqref{eqn: S_ij} that holds in the RCA~$H_n$. 
We now proceed to prove  
each part of Lemma~\ref{lemma: properties of D_i} in turn.
Later, in Section~\ref{sec: integrable systems}, parts~(2) and (3) will be generalised to a more general version of the elements~$D_i$. Those considerations also allow to deduce part (3) from part (1) in a straightforward way. 

\begin{proof}[Proof of Lemma~\ref{lemma: properties of D_i} (1)]
    This follows from \cite[(1.4.64) and (1.4.68)]{Cherednik} and the duality between $X$ and~$Y$ described in \cite[Theorem~1.4.8]{Cherednik}.
\end{proof}

\begin{proof}[Proof of Lemma~\ref{lemma: properties of D_i} (2)]
    As $Y_{n-1}X_n = X_nT_{n-1}^{-2}Y_{n-1}$ by part~(1), we have
    \begin{equation}
    \begin{aligned}\label{claim 1}
        Y_nT_{n-1}X_n^{-1} &= T_{n-1}^{-1}Y_{n-1}X_n^{-1} = T_{n-1}X_n^{-1}Y_{n-1}, \\[0.7em] 
        Y_n^{-1}T_{n-1}X_n^{-1} &= 
        T_{n-1}Y_{n-1}^{-1}T_{n-1}^2 X_n^{-1} = T_{n-1}X_n^{-1} Y_{n-1}^{-1}.
    \end{aligned}
    \end{equation}  
    Secondly, we have
    \begin{equation}\label{claim 2}
        [(Y_{n-1} - Y_{n-1}^{-1})(Y_n - Y_n^{-1}),T_{n-1}] = 0.
    \end{equation}
    This follows from the fact that $[T_{n-1}, (Y_{n-1}Y_n)^{\pm 1}] = 0 = [T_{n-1}, Y_{n-1}+Y_n]$
    and that $(Y_{n-1} - Y_{n-1}^{-1})(Y_n - Y_n^{-1})$ is symmetric in $Y_{n-1}$, $Y_n$, hence can be expressed through $(Y_{n-1}Y_n)^{\pm 1}$ and $Y_{n-1}+Y_n$. 
    
    By using relations~\eqref{claim 1}, \eqref{claim 2}, and that $D_{n-1} = T_{n-1}D_nT_{n-1}$, we get
    \begin{align*}
        (q-q^{-1})^2&[D_{n-1}, D_n] \\
        &= T_{n-1}X_n^{-1} (Y_n - Y_n^{-1})T_{n-1}X_n^{-1} (Y_n - Y_n^{-1}) \\
            &\qquad -  X_n^{-1} (Y_n - Y_n^{-1})T_{n-1}X_n^{-1} (Y_n - Y_n^{-1})T_{n-1} \\
        &\stackrel{\eqref{claim 1}}{=} X_{n-1}^{-1}X_n^{-1}(Y_{n-1} - Y_{n-1}^{-1})(Y_n - Y_n^{-1}) \\
            &\qquad-X_n^{-1}X_{n-1}^{-1}T_{n-1}^{-1}(Y_{n-1} - Y_{n-1}^{-1})(Y_n - Y_n^{-1})T_{n-1} \stackrel{\eqref{claim 2}}{=} 0. 
    \end{align*}
    For $1 \leq i \leq  n-2$, since $[D_n, T_{i,n-2}^+] = [D_n, T_{n-2, i}^-] = 0$, we get
    \begin{align*}
        [D_i, D_n] &= [T_{i,n-2}^+ D_{n-1} T_{n-2, i}^-, D_n] =  T_{i,n-2}^+ [D_{n-1}, D_n]T_{n-2, i}^- = 0.
    \end{align*}
    For $n-1 \geq j > i \geq 1$, 
    $[D_i, T_{j, n-1}^+] = [D_i, T_{n-1, j}^-] = 0$ by Lemma~\ref{lemma: T_k and D_i}, hence
    \begin{equation*}
        [D_i, D_j] = [D_i, T_{j, n-1}^+ D_n T_{n-1, j}^-] = T_{j,n-1}^+ [D_i, D_n] T_{n-1, j}^- = 0.
    \end{equation*}
    This completes the proof.
\end{proof}

\begin{proof}[Proof of Lemma~\ref{lemma: properties of D_i} (3)]
    Let $n \geq j > i \geq 1$.
    Firstly, by using equality~\eqref{eqn: definition of D_i}, Lemma~\ref{lemma: properties of D_i}~(1), and that 
    $[Y_i, (T^{-1})_{j, n-1}^+] = [Y_i, T_{n-1, j}^-] = 0$, we get
    \begin{align*}
        (q-q^{-1}) Y_i D_j &= Y_iX_j^{-1}(T^{-1})_{j, n-1}^+ (Y_n - Y_n^{-1})T_{n-1, j}^- \\
                               &= (\cR^{-1})_{ji}^+ X_j^{-1}Y_i (T^{-1})_{j, n-1}^+ (Y_n - Y_n^{-1})T_{n-1, j}^- \\
                               &= (q-q^{-1})(\cR^{-1})_{ji}^+ D_jY_i,
    \end{align*}
    as required. 
    Secondly, by using equality~\eqref{eqn: definition of D_i}, Lemmas~\ref{lemma: properties of D_i} (1) and~\ref{lemma: Hecke algebra identity}, and that 
   $Y_j (T^{-1})_{i, j-2}^+ T_{j-1} (T^{-1})_{j, n-1}^+ = $   $(T^{-1})_{i, n-1}^+ Y_{j-1}$, we get
    \begin{align*}
        (q-q^{-1})Y_jD_i &= Y_jX_i^{-1}(T^{-1})_{i, n-1}^+ (Y_n - Y_n^{-1})T_{n-1, i}^- \\
                         &= X_i^{-1} Y_j (T^{-1})_{i, j-2}^+ T_{j-1} (T^{-1})_{j, n-1}^+ (Y_n - Y_n^{-1})T_{n-1, i}^- \\
                         &= X_i^{-1} (T^{-1})_{i, n-1}^+ (Y_n - Y_n^{-1})Y_{j-1}T_{n-1, i}^- \\
                         &= (q-q^{-1})D_i (T^{-1})_{i, n-1}^+ Y_{j-1}T_{n-1, i}^-\\
                         &= (q-q^{-1})D_iY_j (T^{-1})_{i, j-2}^+ T_{j-1}^2 T_{j-2, i}^-
                         \\
                         &= (q-q^{-1})D_iY_j \cR_{ji}^+,
    \end{align*}
    as required.
    Thirdly, using that Lemma~\ref{lemma: properties of D_i} (1) gives $$Y_l T^-_{l-1,1}T^+_{1,l-1}X_l^{-1} = q^{-1}X_l^{-1} (T^{-1})^+_{l,n-1} (T^{-1})^-_{n-1,l}Y_l,$$ we get
    \begin{align*}
        (q-q^{-1})Y_l &T_{l-1, 1}^- T_{1,l-1}^+ D_l \stackrel{\eqref{eqn: definition of D_i}}{=} Y_l T_{l-1, 1}^- T_{1,l-1}^+ X_l^{-1}(T^{-1})_{l, n-1}^+ (Y_n - Y_n^{-1})T_{n-1, l}^- \\
                    &= q^{-1}X_l^{-1} (T^{-1})_{l, n-1}^+ (T^{-1})_{n-1, l}^- Y_l (T^{-1})_{l, n-1}^+ (Y_n - Y_n^{-1})T_{n-1, l}^- \\
                    &= q^{-1} X_l^{-1} (T^{-1})_{l, n-1}^+ (Y_n - Y_n^{-1})(T^{-1})_{n-1, l}^- Y_l \\
                    &\stackrel{\eqref{eqn: definition of D_i}}{=}q^{-1}(q-q^{-1})D_l (T^{-1})_{l, n-1}^+ (T^{-1})_{n-1, l}^- Y_l, 
    \end{align*}
    as required, where for the penultimate equality we used Lemma~\ref{lem: DAHA identity}.
\end{proof}

\begin{proof}[Proof of Lemma~\ref{lemma: properties of D_i} (4)]
Relation~\eqref{eqn: X_l D_l} follows from equality~\eqref{eqn: definition of D_i}.

Next, using 
Lemma~\ref{lemma: properties of D_i} (1) with $l=n$, we compute
\begin{align}
    &(q-q^{-1})D_1X_1 
                     = T_{1, n-1}^+ X_n^{-1}(Y_n - Y_n^{-1})X_n (T^{-1})_{n-1, 1}^- \nonumber \\
                     &= T_{1, n-1}^+ \left(q Y_n T_{n-1,1}^- T_{1, n-1}^+ - q^{-1}(T^{-1})_{n-1,1}^- (T^{-1})_{1, n-1}^+Y_n^{-1}\right)(T^{-1})_{n-1, 1}^- \nonumber \\
                     &= q T_{1, n-1}^+ Y_n T_{n-1, 1}^- - q^{-1}(T^{-1})_{1, n-1}^+ Y_n^{-1}(T^{-1})_{n-1, 1}^- \nonumber \\
                     &= q Y_1 - q^{-1} Y_1^{-1}, \label{eqn: D_1X_1}
\end{align}
which proves relation~\eqref{eqn: D_l X_l} for $l=1$. For $2 \leq l \leq n$, we have 
\begin{align*}
    D_lX_l &= (T^{-1})_{l-1,1}^- D_1 (T^{-1})_{1,l-1}^+ X_l = (T^{-1})_{l-1,1}^-  D_1 X_1 T_{1,l-1}^+, 
\end{align*}
which combined with equality~\eqref{eqn: D_1X_1} completes the proof of relation~\eqref{eqn: D_l X_l}. 

Next, by using Lemma~\ref{lemma: properties of D_i} (1) with $i = n$, $j=1$, and the fact that $\cR_{n,1}^+ = 1 + (\tau - \tau^{-1})T_{n-1, 1}^- (T^{-1})_{2, n-1}^+$,
we get
\begin{align}
    &(q-q^{-1})[D_n, X_1] = X_n^{-1}[Y_n - Y_n^{-1}, X_1] \nonumber \\
    &= X_n^{-1}\left(Y_n(1- \cR_{n,1}^+)X_1 + (1- \cR_{n,1}^+)X_1 Y_n^{-1}\right) \nonumber \\
    &= (\tau^{-1}-\tau)X_n^{-1}\left(Y_n T_{n-1, 1}^- (T^{-1})_{2, n-1}^+ X_1 + T_{n-1, 1}^- (T^{-1})_{2, n-1}^+X_1Y_n^{-1}\right). \label{eqn: [D_n, X_1] intermediate} 
\end{align}
Here $T_{n-1, 1}^- (T^{-1})_{2, n-1}^+ X_1 = X_n (T^{-1})_{n-1, 1}^- (T^{-1})_{2, n-1}^+$, and we use~Lemma~\ref{lemma: properties of D_i}~(1) with $l=n$
to get that the expression~\eqref{eqn: [D_n, X_1] intermediate} equals
\begin{align}
    &(\tau^{-1}-\tau)(qY_n T_{n-1,1}^- (T^{-1})_{2, n-1}^+ + (T^{-1})_{n-1, 1}^- (T^{-1})_{2, n-1}^+ Y_n^{-1}) \nonumber \\
    &=  (\tau^{-1}-\tau) (T^{-1})_{n-1, 1}^- (T^{-1})_{2, n-1}^+ (qY_1+Y_n^{-1}). \label{eqn: [D_n, X_1]} 
\end{align}
We then note that, for $n \geq j > i \geq 1$, we have
\begin{align}
     [D_j, X_i] &= [T_{j, n-1}^+ D_n T_{n-1, j}^-, T_{i-1, 1}^- X_1 T_{1,i-1}^+] = T_{i-1, 1}^- T_{j, n-1}^+ [D_n, X_1] T_{n-1, j}^- T_{1,i-1}^+,
     \label{eqn: D_j, X_i}
\end{align}
as $[D_n, T_k] = 0$ if $k \leq n-2$, and $[X_1, T_k] = 0$ if $k \geq 2$.
Relation~\eqref{eqn: [D_j, X_i]} then follows from~\eqref{eqn: [D_n, X_1]}, \eqref{eqn: D_j, X_i}, and the fact that
$(T^{-1})_{j-1,2}^- (T^{-1})_{1, n-1}^+$ $= (T^{-1})_{1, n-1}^+ (T^{-1})_{j-2, 1}^-$. For the latter, we use the inverse of one of the relations in Lemma~\ref{lemma: braid identities} to move successively $T_2^{-1}, \dots, T_{j-1}^{-1}$ to the right of $(T^{-1})_{1, n-1}^+$.

It now only remains to prove relation~\eqref{eqn: [D_i, X_j]}.

By Corollary~\ref{cor: braid identities 2}, we have $(T_{1,n-1}^+)^2 = T_{2,n-1}^+ T_{1, n-2}^+ T_{n-1}^2$ 
and $(T_{n-1, 1}^-)^2 = T_{n-1}^2 T_{n-2,1}^- T_{n-1,2}^-$.
We will additionally use that $T_{n-1}^2 = 1 + (\tau - \tau^{-1})T_{n-1}$, that $[D_n, T_{n-2,1}^-] = 0 = [D_n, T_{1,n-2}^+]$, and
that $[X_1, T_{n-1,2}^-] = 0 = [X_1, T_{2,n-1}^+]$. We also apply Corollary~\ref{cor: braid identities}~(ii), (iv), relations~\eqref{eqn: [D_n, X_1]}, 
\eqref{eqn: X_l D_l} (with $l=1$), and~\eqref{eqn: D_1X_1}.
Then we get
\begin{align}
    [D_1, X_n] &= [T_{1,n-1}^+ D_n T_{n-1, 1}^-, T_{n-1, 1}^- X_1 T_{1,n-1}^+] \nonumber\\
               &= T_{1,n-1}^+D_nT_{n-1}^2 T_{n-2,1}^- T_{n-1,2}^- X_1T_{1,n-1}^+ - T_{n-1, 1}^- X_1 T_{2,n-1}^+ T_{1, n-2}^+ T_{n-1}^2 D_nT_{n-1, 1}^- \nonumber\\
               &=  T_{n-1,2}^- T_{1,n-1}^+ [D_n,  X_1] T_{n-1,1}^- T_{2,n-1}^+ \nonumber\\
               &\qquad + (\tau^{-1} - \tau) \big(T_{1,n-1}^+T_{n-2,1}^- X_1D_1 - D_1X_1 T_{n-1,1}^- T_{2,n-1}^+\big)  \nonumber\\
               &= \frac{\tau^{-1} - \tau}{q-q^{-1}}(q^{-1}Y_1^{-1} + Y_n) T_{n-1,1}^- T_{2,n-1}^+. \label{eqn: D_1, X_n}
\end{align}
We then note that, for $n \geq j > i \geq 1$, we have
\begin{align}
     [D_i, X_j] &= [(T^{-1})_{i-1, 1}^- D_1 (T^{-1})_{1, i-1}^+, (T^{-1})_{j,n-1}^+ X_n (T^{-1})_{n-1, j}^-] \nonumber \\
     &= (T^{-1})_{i-1, 1}^- (T^{-1})_{j,n-1}^+ [D_1, X_n](T^{-1})_{n-1, j}^- (T^{-1})_{1, i-1}^+,
     \label{eqn: D_i, X_j}
\end{align}
as $[D_1, (T^{-1})_{j,n-1}^+] = [D_1, (T^{-1})_{n-1, j}^-] = 0$ by Lemma~\ref{lemma: T_k and D_i}, and $[X_j, (T^{-1})_{i-1, 1}^-]$ $= [X_j, (T^{-1})_{1, i-1}^+] = 0$. 
Relation~\eqref{eqn: [D_i, X_j]} then follows from~\eqref{eqn: D_1, X_n} and~\eqref{eqn: D_i, X_j} because we have 
$T_{n-1, 1}^- T_{2,j-1}^+ = T_{1,j-2}^+ T_{n-1, 1}^-$, which is seen by using a relation from Lemma~\ref{lemma: braid identities} to move successively $T_2, \dots, T_{j-1}$ to the left of $T_{n-1, 1}^-$.
\end{proof}

\subsection{A presentation of $\bH^{\fgl_n}$ and a basis} \label{sec: presentation}

We begin this section by describing relations among the generators of~$\bH^{\fgl_n}$.
The goal is to write them in a form that makes it apparent that they $\tau$-deform the relations of the algebra $\cA$ from Section~\ref{sec: quantum group} (Propositions~\ref{prop: relations among E_ij^q}, \ref{prop: relations between E_ij^q and t_i}, and relations~\eqref{eqn: moving s_k to the right}) and to deduce below in Theorem~\ref{thm: PBW basis for H_gl} a PBW-type basis for $\bH^{\fgl_n}$ analogous to~\eqref{eqn: PBW basis for cA}.

Define $S_{ij}^\tau \coloneqq [D_i, X_j]$. Explicit formulae for $S_{ij}^\tau$ follow from Lemma~\ref{lemma: properties of D_i}~(4). In particular, one can see that $S_{ij}^\tau \in \bH^{\fgl_n}$.
The following statement is a consequence of the commutativity of the elements $D_i$. 
These relations in the special case of $\tau = 1$ appear earlier in Proposition~\ref{prop: relations among E_ij^q}. The relations below look formally the same as those in~\eqref{swapping j,l or i,k} just with $E_{ij}^q$ replaced by $e_{ij}$ and $S_{ij}^q$ by $S_{ij}^\tau$.
\begin{proposition} \label{prop: relations among e_ij}
For all $1 \leq i \neq j \leq n$, $1 \leq k \neq l \leq n$, we have
    \begin{align*}
        &e_{ij}e_{kl} - e_{il}e_{kj} =  e_{il}S_{jk}^\tau - e_{ij}S_{lk}^\tau,  \\
        &e_{ij}e_{kl} - e_{kj}e_{il} = S_{jk}^\tau e_{il} - S_{ji}^\tau e_{kl}. 
    \end{align*}
\end{proposition}

\begin{proof}
The second relation is proved similarly to the first. We have
\begin{align*}
        e_{ij}e_{kl} &= e_{ij}\left(D_lX_k - S_{lk}^\tau \right) =  e_{il}D_jX_k - e_{ij}S_{lk}^\tau = e_{il} \left( e_{kj} + S_{jk}^\tau \right) - e_{ij}S_{lk}^\tau. \quad \qedhere
    \end{align*}
\end{proof}

Further relations are as follows. One can move $T_k$ to the left through~$Y_i^{\pm 1}$ thanks to the defining relations of the DAHA $\bH_n$. 
The relations that enable us to move $T_k$ to the left through $e_{ij}$ ($i \neq j$) are given in the next proposition. These relations at $\tau = 1$ coincide with those from~\eqref{eqn: moving s_k to the right} between $s_k$ and $E_{ij}^q$.

\begin{proposition}\label{prop: moving T_k to the left}
    Let $i \in \{1, \dots, n-1\}$. We have
    \begin{equation*}
        T_i e_{i, i+1} T_i = e_{i+1, i} + (\tau^{-1} - \tau)(q-q^{-1})^{-1}(T^{-1})_{i+1, n-1}^+ (Y_n - Y_n^{-1})T_{n-1, i}^-.
    \end{equation*}
    Also, for any $j \in \{1, \dots, n\} \setminus \{i,i+1\}$, we have
    \begin{equation*}
        T_i e_{ij} T_i =  e_{i+1, j}, \qquad T_i e_{j, i+1} T_i = e_{ji}.
    \end{equation*}
    Finally, for any $j \in \{1, \dots, n\} \setminus \{i,i+1\}$ and $k \in \{1, \dots, n\} \setminus \{j,i,i+1\}$, we have 
    \begin{equation*}
        e_{jk}T_i = T_ie_{jk}.
    \end{equation*}
\end{proposition}
\begin{proof}
    For $1 \leq i \leq n-1$, we have
    \begin{equation*}
        T_i e_{i, i+1} T_i = T_i X_i T_i^{-1} D_i = X_{i+1}D_i 
        + (\tau^{-1} - \tau)T_i X_iD_i,
    \end{equation*}
    and the first relation follows from Lemma~\ref{lemma: properties of D_i} (4).
    For $j \in \{1, \dots, n\} \setminus \{i, i+1 \}$, by Lemma~\ref{lemma: T_k and D_i} we have $[D_j, T_i] = 0$, and so $T_ie_{ij} T_i = T_iX_i T_i D_j = X_{i+1}D_j = e_{i+1, j}$, as required. The third relation is proved similarly, since
    $[X_j, T_i] = 0$ and $T_iD_{i+1}T_i = D_i$.
    If also $k \in \{1, \dots, n\} \setminus \{i, i+1 \}$, then $[e_{jk}, T_i]=0$ as $[D_k, T_i] = 0$, too.
\end{proof}

The relations that help us move $Y_i^{\pm 1}$ to the right through $e_{jk}$ ($j \neq k$) can be split into three cases: 
$Y_i^{\pm 1}$ with $e_{ij}$ for $i \neq j$; $Y_i^{\pm 1}$ with $e_{ji}$ for $i \neq j$; and~$Y_i^{\pm 1}$ with $e_{jk}$ for $i \neq j \neq k \neq i$.
We have the following statement, which at~$\tau = 1$ reproduces (parts of)~Proposition~\ref{prop: relations between E_ij^q and t_i}. The only thing that will be important for us about the form of the expressions $C_1$ and $C_2$ below is that they can be written using just the generators $T$ and $Y$, without $X$ or~$D$, and we will not need their precise form.
\begin{proposition}\label{prop: moving Y_i to the right}
\begin{enumerate}[(1)]
    \item For $n \geq i \neq j \geq 1$, we have 
        \begin{flalign*}
                Y_i &e_{ij} (T^{-1})_{i-1, 1}^- (T^{-1})_{1, i-1}^+ Y_i^{-1} = q T_{i,n-1}^+ T_{n-1,i}^- e_{ij} 
                + (q-q^{-1})^{-1}(\tau - \tau^{-1}) C_1, & 
        \end{flalign*}
        where
        \begin{equation*}
            C_1 = \begin{cases}
                qT_{i, n-1}^+ (T^{-1})_{j-2, i}^- (Y_n - Y_n^{-1}) T_{n-1, j}^- \quad \textnormal{ if } j > i, \\
                T_{i,n-1}^+ (T^{-1})_{j-1,1}^- (T^{-1})_{1, n-2}^+ Y_{n-1}^{-1} (Y_n^2 - 1) T_{n-1, j}^- (T^{-1})_{n-1, i}^- \quad
                \textnormal{ if } i > j.
            \end{cases}
        \end{equation*}
    \item For $n \geq i \neq j \geq 1$, we have
            \begin{flalign*}
                Y_i &T_{i-1,1}^- T_{1,i-1}^+ e_{ji} Y_i^{-1} = q^{-1} e_{ji} (T^{-1})_{i,n-1}^+ (T^{-1})_{n-1,i}^- 
                + (q-q^{-1})^{-1}(\tau^{-1} - \tau)C_2, & 
            \end{flalign*}
        where 
        \begin{equation*}
            C_2 = \begin{cases}
                q^{-1}(T^{-1})_{j,n-1}^+ (Y_n - Y_n^{-1}) T_{i, j-2}^+ (T^{-1})_{n-1, i}^- \quad \textnormal{ if } j > i, \\
                T_{i,n-1}^+ (T^{-1})_{j,n-1}^+ Y_{n-1} (1 - Y_n^{-2}) T_{n-2, 1}^- T_{1,j-1}^+ (T^{-1})_{n-1, i}^-
                \quad \textnormal{ if } i > j.
            \end{cases}
        \end{equation*}
    \item For all $i,j,k \in \{1, \dots, n\}$ with $i \neq j \neq k \neq i$, we have
        \begin{align*}
            Y_i & \cR_{ij}^{+} e_{jk} \cR_{ik}^{-} Y_i^{-1} = (\cR^{-1})_{ki}^{+} e_{jk} (\cR^{-1})_{ji}^{-}.
        \end{align*}
\end{enumerate}
\end{proposition}
Throughout the proof of Proposition~\ref{prop: moving Y_i to the right}, we rely on Lemmas~\ref{lemma: T_k and D_i}, \ref{lemma: properties of D_i}. 

\begin{proof}[Proof of Proposition~\ref{prop: moving Y_i to the right} (1)]       
    Let us first consider the case $n \geq j > i \geq 1$. In this case, $D_j$ commutes with $(T^{-1})_{i-1, 1}^- (T^{-1})_{1, i-1}^+$ by Lemma~\ref{lemma: T_k and D_i}, and so we have
    \begin{align}
        Y_i &e_{ij} (T^{-1})_{i-1, 1}^- (T^{-1})_{1, i-1}^+ Y_i^{-1} = Y_i X_i (T^{-1})_{i-1, 1}^- (T^{-1})_{1, i-1}^+ D_j Y_i^{-1} \nonumber \\
        &= q T_{i, n-1}^+ T_{n-1, i}^- X_iY_iD_j Y_i^{-1} \nonumber \\
        &= q T_{i, n-1}^+ T_{n-1, i}^- X_i (\cR^{-1})_{ji}^+ D_j \nonumber \\
        &= q T_{i, n-1}^+ T_{n-1, i}^- e_{ij} 
        + q(\tau - \tau^{-1}) T_{i, n-1}^+ T_{n-1, i}^- X_i (T^{-1})_{j-1, i+1}^- T_{i, j-1}^+ D_j, \label{Y_i e_ij part 1}
    \end{align}
    where we used that $(\cR^{-1})_{ji}^+ = 1 + (\tau - \tau^{-1})(T^{-1})_{j-1, i+1}^- T_{i, j-1}^+$. Then 
    \begin{align}
        &T_{n-1, i}^- X_i (T^{-1})_{j-1, i+1}^- T_{i, j-1}^+ D_j = T_{n-1, i}^- (T^{-1})_{j-1, i+1}^- X_i D_i (T^{-1})_{i, j-1}^+ \nonumber \\
        &\qquad =  T_{n-1, i}^- (T^{-1})_{j-1, i+1}^- (T^{-1})_{i, n-1}^+ X_nD_n T_{n-1, j}^- \nonumber \\
        &\qquad =  (T^{-1})_{j-2, i}^- X_nD_n T_{n-1, j}^-, \label{Y_i e_ij part 2}
    \end{align}
    since
    $T_{n-1, i}^- (T^{-1})_{j-1, i+1}^-$ $= (T^{-1})_{j-2, i}^- T_{n-1, i}^-$ by using a relation from Lemma~\ref{lemma: braid identities} to move successively $T_{j-1}^{-1}, \dots, T_{i+1}^{-1}$ to the left of $T_{n-1, i}^-$.
    Inserting expression~\eqref{Y_i e_ij part 2} for $T_{n-1, i}^- X_i (T^{-1})_{j-1, i+1}^- T_{i, j-1}^+ D_j$ into~\eqref{Y_i e_ij part 1}, and using~\eqref{eqn: X_l D_l}, this proves the claim of Proposition~\ref{prop: moving Y_i to the right} (1) in the case $j>i$. It remains to deal with the case $i>j$. We do that in steps.

    Let us next establish the claim for $i=n$, $j=n-1$. 
    Using Lemma~\ref{lemma: properties of D_i} (1) with $l=n$ to move $Y_n$ to the right past $X_n$, and writing $D_{n-1}$ as $T_{n-1}D_nT_{n-1}$ (Lemma~\ref{lemma: T_k and D_i}), we get
    \begin{align}
        &Y_n e_{n, n-1} (T^{-1})_{n-1, 1}^- (T^{-1})_{1, n-1}^+ Y_n^{-1} \label{Y_n e_n,n-1 LHS} \\
        &= q X_nY_n T_{n-1, 1}^- T_{1,n-2}^+ T_{n-1}^2 D_n  (T^{-1})_{n-2, 1}^- (T^{-1})_{1, n-1}^+ Y_n^{-1} \nonumber\\
        &= q X_nY_n T_{n-1} D_n T_{n-1}^{-1} Y_n^{-1} + (\tau - \tau^{-1})X_n D_n Y_n (T^{-1})_{n-2, 1}^- (T^{-1})_{1, n-1}^+ Y_n^{-1}, \nonumber
    \end{align}
    where in the second equality we used that $T_{n-1}^2 = 1 + (\tau - \tau^{-1})T_{n-1}$, 
    as well as Lemma~\ref{lemma: properties of D_i} (3) with $l=n$, and that $[T^-_{n-2,1}T^+_{1,n-2},D_n]=0$ by Lemma~\ref{lemma: T_k and D_i}. The final line in~\eqref{Y_n e_n,n-1 LHS} can further be rearranged, by using that $(T^{-1})_{n-2, 1}^- (T^{-1})_{1, n-2}^+$ commutes with $X_nD_nY_n$ due to relations~\eqref{eqn: T_k and X_i} and~\eqref{eqn: T_k and Y_i} and Lemma~\ref{lemma: T_k and D_i}, into the following expression
    \begin{align}
        &q X_n Y_n D_{n-1} T_{n-1}^{-2} Y_n^{-1} + (\tau - \tau^{-1})(T^{-1})_{n-2, 1}^- (T^{-1})_{1, n-2}^+ X_nD_nY_n T_{n-1}^{-1} Y_n^{-1} \nonumber \\
        &= q X_nD_{n-1} + (\tau - \tau^{-1})(T^{-1})_{n-2, 1}^- (T^{-1})_{1, n-2}^+ X_nD_nY_n Y_{n-1}^{-1} T_{n-1}, \label{Y_n e_n,n-1}
    \end{align}
    where to get the first term in the last line we used Lemma~\ref{lemma: properties of D_i} (3). From this the claim for $i = n$, $j = n-1$ follows by replacing $X_nD_n$ in~\eqref{Y_n e_n,n-1} with $(q-q^{-1})^{-1}(Y_n - Y_n^{-1})$ using relation~\eqref{eqn: X_l D_l}.
     
    Next, consider any $1 \leq j \leq n-1$. Writing $D_j$ as $T_{j,n-2}^+ D_{n-1} T_{n-2, j}^-$ (Lemma~\ref{lemma: T_k and D_i}), we have
    \begin{align}
        &Y_n e_{nj} (T^{-1})_{n-1, 1}^- (T^{-1})_{1, n-1}^+ Y_n^{-1}  \nonumber \\ 
        &= Y_n X_n T_{j,n-2}^+ D_{n-1} T_{n-2, j}^- (T^{-1})_{n-1, 1}^- (T^{-1})_{1, n-1}^+ Y_n^{-1} \nonumber \\
        &=  T_{j, n-2}^+ Y_n e_{n, n-1} (T^{-1})_{n-1, 1}^- (T^{-1})_{1, n-1}^+  Y_n^{-1} T_{n-2, j}^-, \label{Y_n e_nj} 
    \end{align}
    where 
    we used that $[Y_n X_n, T_{j, n-2}^+] = 0$, and that $T_{n-2, j}^-$ commutes with $(T^{-1})_{n-1, 1}^- (T^{-1})_{1, n-1}^+$ as a consequence of Corollary~\ref{cor: braid identities} (i), (iv). Now we use the form~\eqref{Y_n e_n,n-1} for the expression~\eqref{Y_n e_n,n-1 LHS} to rearrange expression~\eqref{Y_n e_nj} as 
    \begin{equation*}
        q e_{nj} + (\tau - \tau^{-1}) (T^{-1})_{j-1, 1}^- (T^{-1})_{1, n-2}^+ X_nD_n Y_n Y_{n-1}^{-1} T_{n-1, j}^-,
    \end{equation*}
    which completes the proof of the claim for $i=n$.
    
    Finally, for any $n \geq i > j \geq 1$, we have
    \begin{equation}\label{eqn: Y_i e_ij}
            (T^{-1})_{n-1, i}^- Y_i e_{ij} (T^{-1})_{i-1, 1}^- (T^{-1})_{1, i-1}^+ Y_i^{-1} T_{i,n-1}^+ = Y_n e_{nj}  (T^{-1})_{n-1, 1}^- (T^{-1})_{1, n-1}^+ Y_n^{-1}. 
    \end{equation}
    The proof is completed by combining equality~\eqref{eqn: Y_i e_ij} with the claim for $i=n$, and using that 
    $e_{nj} (T^{-1})_{n-1, i}^- = T_{n-1, i}^- e_{ij}$.
\end{proof}
\begin{proof}[Proof of Proposition~\ref{prop: moving Y_i to the right} (2)]
     For $n \geq j > i \geq 1$, we have 
    \begin{align}
        Y_i &T_{i-1,1}^- T_{1,i-1}^+ e_{ji} Y_i^{-1} = Y_i X_j T_{i-1,1}^- T_{1,i-1}^+ D_i Y_i^{-1} \nonumber \\
        &=  X_j (\cR^{-1})_{ji}^- Y_i T_{i-1,1}^- T_{1,i-1}^+ D_i Y_i^{-1} \nonumber \\
        &= q^{-1}X_j (\cR^{-1})_{ji}^-  D_i (T^{-1})_{i, n-1}^+ (T^{-1})_{n-1, i}^- \nonumber \\
        &= q^{-1} e_{ji} (T^{-1})_{i, n-1}^+ (T^{-1})_{n-1, i}^- \nonumber \\
        &\quad + q^{-1}(\tau^{-1}-\tau) X_j (T^{-1})_{j-1, i}^- T_{i+1, j-1}^+  D_i (T^{-1})_{i, n-1}^+ (T^{-1})_{n-1, i}^-, \label{Y_i e_ji part 1}
    \end{align}
    where we used that $(\cR^{-1})_{ji}^- = 1 + (\tau^{-1} - \tau)(T^{-1})_{j-1, i}^- T_{i+1, j-1}^+$. Then
    \begin{align}
        X_j &(T^{-1})_{j-1, i}^- T_{i+1, j-1}^+  D_i (T^{-1})_{i, n-1}^+ =
        T_{j-1, i}^- T_{i+1, j-1}^+  X_iD_i (T^{-1})_{i, n-1}^+ \nonumber \\
        &= T_{j-1, i}^- T_{i+1, j-1}^+ (T^{-1})_{i, n-1}^+ X_nD_n  \nonumber \\
        &= (T^{-1})_{j, n-1}^+ X_nD_n T_{i, j-2}^+, \label{Y_i e_ji part 2}  
    \end{align}
    since
    $T_{i+1, j-1}^+ (T^{-1})_{i, n-1}^+$ $= (T^{-1})_{i, n-1}^+ T_{i, j-2}^+$ by using the inverse of a relation in Lemma~\ref{lemma: braid identities} to move successively $T_{j-1}, \dots, T_{i+1}$ to the right of $(T^{-1})_{i, n-1}^+$. Relations~\eqref{Y_i e_ji part 1} and \eqref{Y_i e_ji part 2} imply the claim for $j > i$. 
    
    Let us now establish the claim for $i=n$, $j=n-1$. We
    compute
         \begin{align}
            Y_n &T_{n-1,1}^- T_{1,n-1}^+ e_{n-1, n} Y_n^{-1} = Y_n T_{n-1}^2 X_{n-1} T_{n-1,1}^- T_{1,n-2}^+ T_{n-1}^{-1} D_n Y_n^{-1} \label{Y_n e_n-1,n LHS} \\ 
            &= X_{n-1} Y_n T_{n-1,1}^- T_{1,n-2}^+ T_{n-1}^{-1} D_n Y_n^{-1} \nonumber \\
            &= X_{n-1} Y_n T_{n-1,1}^- T_{1,n-1}^+ D_n Y_n^{-1} \nonumber \\
             &\qquad + (\tau^{-1} - \tau)  X_{n-1} Y_n T_{n-1,1}^- T_{1,n-2}^+ D_n Y_n^{-1} \nonumber
        \end{align}
    where we used that $T_{n-1}^{-1} = T_{n-1} + \tau^{-1} - \tau$. Now note $Y_n T_{n-1,1}^- T_{1,n-1}^+ D_nY_n^{-1} = q^{-1}D_n$ by Lemma~\ref{lemma: properties of D_i} (3); and to deal with $Y_n T_{n-1,1}^- T_{1,n-2}^+ D_n Y_n^{-1}$ in the term proportional to $(\tau^{-1} - \tau)$, we can write $Y_n$ as $T_{n-1}^{-1}Y_{n-1}T_{n-1}^{-1}$ and then use that $T_{n-2,1}^- T_{1,n-2}^+$ commutes with $D_nY_n^{-1}$. Thus, altogether, we arrive at  
    \begin{align}
        &q^{-1} X_{n-1} D_n + (\tau^{-1} - \tau)  X_{n-1} T_{n-1}^{-1} Y_{n-1} D_n Y_n^{-1}  T_{n-2,1}^- T_{1,n-2}^+ \nonumber \\
        &= q^{-1} e_{n-1, n} + (\tau^{-1} - \tau) T_{n-1}^{-1} X_n D_n Y_{n-1} Y_n^{-1} T_{n-2,1}^- T_{1,n-2}^+, \label{Y_n e_n-1,n}
    \end{align}
    where we used that $Y_{n-1}D_n = T^2_{n-1}D_nY_{n-1}$ by Lemma~\ref{lemma: properties of D_i} (3).
    From~\eqref{Y_n e_n-1,n}, the claim for $i = n$, $j = n-1$ follows.
    
    Next, consider any $1 \leq j \leq n-1$. Writing $X_j$ as $(T^{-1})_{j, n-2}^+ X_{n-1} (T^{-1})_{n-2, j}^-$ using relations~\eqref{eqn: T_k and X_i}, we get
    \begin{align}
        Y_n&T_{n-1, 1}^- T_{1,n-1}^+ e_{jn} Y_n^{-1} \nonumber \\
        &= Y_n T_{n-1, 1}^- T_{1,n-1}^+ (T^{-1})_{j, n-2}^+ X_{n-1} (T^{-1})_{n-2, j}^- D_n Y_n^{-1} \nonumber \\ 
        &= (T^{-1})_{j, n-2}^+ Y_n T_{n-1, 1}^- T_{1,n-1}^+  e_{n-1, n} Y_n^{-1} (T^{-1})_{n-2, j}^- \label{eqn: Y_n e_jn}
    \end{align}
    since $(T^{-1})_{j, n-2}^+$ commutes with $T_{n-1, 1}^- T_{1,n-1}^+$ as a consequence of Corollary~\ref{cor: braid identities}~(i) and~(iv), and it also commutes with $Y_n$ due to~\eqref{eqn: T_k and Y_i}; and we also used that $(T^{-1})_{n-2, j}^-$ similarly commutes with $Y_n^{-1}$ and with $D_n$ (Lemma~\ref{lemma: T_k and D_i}). Now we use the form~\eqref{Y_n e_n-1,n} for the left-hand side of equality~\eqref{Y_n e_n-1,n LHS} to rearrange expression~\eqref{eqn: Y_n e_jn} as 
    \begin{equation*}
        q^{-1} e_{jn} +  (\tau^{-1} - \tau) (T^{-1})_{j, n-1}^+ X_n D_n Y_{n-1} Y_n^{-1}  T_{n-2,1}^- T_{1,j-1}^+,
    \end{equation*}
    which completes the proof of the claim for $i=n$.
    
    For $n \geq i > j \geq 1$, we have
    \begin{equation}\label{eqn: Y_i e_ji}
        (T^{-1})_{n-1, i}^- Y_i T_{i-1,1}^- T_{1,i-1}^+ e_{ji} Y_i^{-1} T_{i,n-1}^+ 
        = Y_n T_{n-1, 1}^- T_{1,n-1}^+ e_{jn} Y_n^{-1}.
    \end{equation}
    The proof is completed by combining equality~\eqref{eqn: Y_i e_ji} with the claim for $i=n$, and using that 
    $T_{i, n-1}^+ e_{jn} = e_{ji} (T^{-1})_{i, n-1}^+$.
\end{proof}
\begin{proof}[Proof of Proposition~\ref{prop: moving Y_i to the right} (3)]       
    For $j>i$, $k > i$, $j \neq k$, we have
    \begin{align}
         Y_i&e_{jk}Y_i^{-1} = X_j (\cR^{-1})_{ji}^- Y_i D_k Y_i^{-1} = X_j (\cR^{-1})_{ji}^- (\cR^{-1})_{ki}^+ D_k. 
            \label{j neq k case}
    \end{align}
    By Lemma~\ref{lemma: Hecke algebra identity} applied to $(\cR^{-1})_{ji}^-$, $(\cR^{-1})_{ki}^+$, 
    if $k > j$ then the right-hand side of equality \eqref{j neq k case} equals
    \begin{align*}
             X_j &T_{i, j-2}^+ T_{j-1}^{-1} T_{j, k-2}^+ T_{k-1}^2 (T^{-1})_{k-2, i}^- D_k \\
             &= T_{i, k-2}^+ T_{k-1}^2 (T^{-1})_{k-2, j}^- X_{j-1} (T^{-1})_{j-1, i}^- D_k  \\
             &= T_{i, k-2}^+ T_{k-1}^2 (T^{-1})_{k-2, j-1}^- X_j T_{j-1}^{-2} (T^{-1})_{j-2, i}^- D_k  \\
             &= T_{i, k-2}^+ T_{k-1}^2 (T^{-1})_{k-2, j-1}^- e_{jk} T_{j-1}^{-2} (T^{-1})_{j-2, i}^- \\
             &= T_{i,k-2}^+ T_{k-1}^2 (T^{-1})_{k-2, i}^-e_{jk}T_{i, j-2}^+T_{j-1}^{-2}(T^{-1})_{j-2, i}^- \\
             &= (\cR^{-1})_{ki}^+ e_{jk}(\cR^{-1})_{ji}^-,
    \end{align*}
    as required. Similarly if $j > k$, the right-hand side of equality \eqref{j neq k case} equals
    \begin{align*}
            X_j &T_{i,j-2}^+ T_{j-1}^{-2} (T^{-1})_{j-2, k}^- T_{k-1} (T^{-1})_{k-2, i}^- D_k \\
                &= X_j T_{i, k-1}^+ D_{k-1} T_{k,j-2}^+ T_{j-1}^{-2} (T^{-1})_{j-2, i}^- \\
                &= X_j T_{i, k-2}^+ T_{k-1}^2 D_k T_{k-1,j-2}^+ T_{j-1}^{-2} (T^{-1})_{j-2, i}^- \\
                &= T_{i, k-2}^+ T_{k-1}^2 e_{jk} T_{k-1,j-2}^+ T_{j-1}^{-2} (T^{-1})_{j-2, i}^- \\
                &= T_{i,k-2}^+ T_{k-1}^2 (T^{-1})_{k-2, i}^-e_{jk} T_{i, j-2}^+T_{j-1}^{-2}(T^{-1})_{j-2, i}^- \\
                &= (\cR^{-1})_{ki}^+ e_{jk}(\cR^{-1})_{ji}^-,
    \end{align*}
    as required. 
    Next, for $i > j$, $i> k$, $j\neq k$, we have
    \begin{align*}
        Y_i \cR_{ij}^+ e_{jk} \cR_{ik}^- Y_i^{-1} = X_j Y_i D_k \cR_{ik}^- Y_i^{-1} = e_{jk}Y_i\cR_{ik}^+ \cR_{ik}^- Y_i^{-1} = e_{jk},
    \end{align*}
    as required.
    Next, for $k > i > j$, we have
    \begin{align*}
        Y_i \cR_{ij}^+ e_{jk}Y_i^{-1} = X_jY_iD_kY_i^{-1} = X_j  (\cR^{-1})_{ki}^+ D_k = (\cR^{-1})_{ki}^+ e_{jk},
    \end{align*}
    since $[X_j, (\cR^{-1})_{ki}^+] = 0$, as required.
    Finally, for $j > i > k$, we have
    \begin{align*}
        Y_i &e_{jk} \cR_{ik}^- Y_i^{-1} = X_j (\cR^{-1})_{ji}^- Y_i D_k \cR_{ik}^- Y_i^{-1} = X_j (\cR^{-1})_{ji}^- D_k = e_{jk} (\cR^{-1})_{ji}^-,
    \end{align*}
    since $[D_k, (\cR^{-1})_{ji}^-] = 0$, as required. This covered all the possibilities. 
\end{proof}

In Proposition~\ref{prop: moving Y_i to the right}, the relations in cases
\begin{equation}\label{eqn: cases without Ts}
(1); (2) \textnormal{ for } i = 1; \textnormal{ and } (3) \textnormal{ for } j > i
\end{equation}
have the elements~$Y_i$ in the left-hand side placed immediately before the corresponding elements~$e_{jk}$. On the other hand, the relations in cases~(2) when $i \neq 1$, and (3) for $i > j$ have Hecke algebra elements $T_l$ in between the~$Y_i$ and $e_{jk}$ in the left-hand side.  Therefore, to get to the PBW basis statement, we need two more technical lemmas.

In order to be able to move an arbitrary $Y_i$ to the right through an arbitrary directly adjacent $e_{jk}$, we also need the following Lemma~\ref{lemma: moving Y_i to the right}. Whenever we encounter $Y_i$ directly adjacent to some $e_{jk}$ with their indices not falling into one of the cases \eqref{eqn: cases without Ts}, we can expand such a monomial into a sum of terms each of which can be handled, in the sense of moving~$Y$'s to the right. 
The case $Y_i e_{jk}$ with $i > j$ and $j \neq k \neq i$ can be dealt with by Lemma~\ref{lemma: moving Y_i to the right}~(ii) and Proposition~\ref{prop: moving Y_i to the right} (1), (3).
The case $Y_i e_{ji}$ with $j > i \geq 2$ can be dealt with by Lemma~\ref{lemma: moving Y_i to the right} (i) and Proposition~\ref{prop: moving Y_i to the right}~(2), (3).
The final case to consider is $Y_i e_{ji}$ with $i > j$. The first step is to apply Lemma~\ref{lemma: moving Y_i to the right}~(i). By then applying Proposition~\ref{prop: moving Y_i to the right} (2), we are left to consider terms $Y_k e_{ji}$ with~$k < i$. For the terms with $k < j$, we apply Proposition~\ref{prop: moving Y_i to the right}~(3); for the term with $k = j$, we apply Proposition~\ref{prop: moving Y_i to the right}~(1); and we deal with the terms with~$k > j$ by applying Lemma~\ref{lemma: moving Y_i to the right} (ii) and Proposition~\ref{prop: moving Y_i to the right}~(1),~(3).

\begin{lemma}\label{lemma: moving Y_i to the right}
\begin{enumerate}[(i)]
    \item 
    For $n \geq i \geq 1$, we have
        \begin{align*}
             Y_i  = Y_i T_{i-1,1}^- T_{1, i-1}^+  + (\tau^{-1} - \tau)\sum_{k = 1}^{i-1} (T^{-1})_{i-1, k}^- T_{k+1, i-1}^+ Y_k. 
        \end{align*}   
    \item 
    For $n \geq i > j \geq 1$, we have
        \begin{equation*}
            Y_i =
            Y_i \cR_{ij}^{+}  
            + (\tau^{-1} - \tau)(T^{-1})_{i-1, j}^- (T^{-1})_{j+1, i-1}^+ Y_j.
        \end{equation*}
\end{enumerate}
\end{lemma}
\begin{proof}
    (i) The claim is trivial for $i=1$, so suppose $i>1$. By using $T_1^2 = 1 + (\tau - \tau^{-1})T_1$, we get
    \begin{align}
        Y_i T_{i-1,1}^- T_{1, i-1}^+ &= Y_i T_{i-1,2}^- T_{2, i-1}^+ + (\tau - \tau^{-1})Y_i T_{i-1,1}^- T_{2, i-1}^+ \nonumber \\
        &= Y_i T_{i-1,2}^- T_{2, i-1}^+ + (\tau - \tau^{-1})(T^{-1})_{i-1,1}^- T_{2, i-1}^+ Y_1. \label{eqn: i=2 case}
    \end{align}
    If $i = 2$, then the desired relation follows by rearranging equality~\eqref{eqn: i=2 case} for $Y_i T_{i-1,2}^- T_{2, i-1}^+ = Y_i$. 
    Assume $i > 2$, then we iterate $i-2$ times the manipulation of~\eqref{eqn: i=2 case}. Thus, at the next step we use $T_2^2 = 1 + (\tau - \tau^{-1})T_2$ to get
    that $Y_i T_{i-1,1}^- T_{1, i-1}^+$ equals
    \begin{align*}
         &Y_i T_{i-1,3}^- T_{3, i-1}^+  + (\tau - \tau^{-1}) Y_i T_{i-1,2}^- T_{3, i-1}^+ 
        + (\tau - \tau^{-1})(T^{-1})_{i-1,1}^- T_{2, i-1}^+ Y_1  \\
        &= Y_i T_{i-1,3}^- T_{3, i-1}^+ + 
        (\tau - \tau^{-1})\sum_{k=1}^2 (T^{-1})_{i-1,k}^- T_{k+1, i-1}^+ Y_k,
    \end{align*}
    and so forth until we obtain the desired relation.

    (ii) By using $T_j^2 = 1 + (\tau - \tau^{-1})T_j$, we get
    \begin{align*}
        Y_i \cR_{ij}^{+}  &=
        Y_i 
        + (\tau - \tau^{-1})Y_i T_{i-1, j}^- (T^{-1})_{j+1, i-1}^+ \\
        &= Y_i 
        + (\tau - \tau^{-1})(T^{-1})_{i-1, j}^- (T^{-1})_{j+1, i-1}^+ Y_j, 
    \end{align*}
    and the desired relation follows.
\end{proof}

In order to be able to move $Y_i^{-1}$ to the right past~$e_{jk}$, we use the relations from Proposition~\ref{prop: moving Y_i to the right} multiplied by~$Y_i^{-1}$ from the left and rearranged to find an expression for the first term in the right-hand side. We also need the next lemma serving an analogous purpose as Lemma~\ref{lemma: moving Y_i to the right} (proved similarly, too) to deal with the cases where we end up with Hecke algebra elements~$T_l$ in between the~$Y_i^{-1}$ and~$e_{jk}$.

\begin{lemma}\label{lemma: moving Y_i inverse to the right}
\begin{enumerate}[(i)]
    \item 
    For $n \geq i \geq 1$, we have
        \begin{align*}
           Y_i^{-1} = Y_i^{-1}T_{i, n-1}^+ T_{n-1, i}^-  + (\tau^{-1} - \tau)\sum_{k=0}^{n-i-1} (T^{-1})_{i, n-k-1}^+ T_{n-k-2,i}^- Y_{n-k}^{-1}. 
        \end{align*}
    \item For $n \geq k > i \geq 1$, we have 
    \begin{equation*}
        Y_i^{-1} =
            Y_i^{-1} (\cR^{-1})_{ki}^{+} 
            +(\tau^{-1} - \tau) (T^{-1})_{k-1, i+1}^- (T^{-1})_{i, k-1}^+ Y_k^{-1}.    
    \end{equation*}
\end{enumerate}
\end{lemma}

In view of the relations presented above, 
we arrive at a $\C_\tau$-basis for the algebra~$\bH^{\fgl_n}$. It may be thought of as a PBW-type basis.
At $\tau = 1$, it reduces to the one given by formula~\eqref{eqn: PBW basis for cA} for the algebra~$\cA$. The following theorem also implies that $\bH^{\fgl_n}$ is a flat $\tau$-deformation of the algebra $\cA$ from Section~\ref{sec: quantum group}.

\begin{theorem}\label{thm: PBW basis for H_gl}
The algebra $\bH^{\fgl_n}$ has a free basis over $\C_\tau$ consisting of the monomials
\begin{equation}
    \begin{aligned}\label{eqn: PBW basis for H_gl}
         T_w e_{i_1 j_1}^{k_1} &\cdots e_{i_t j_t}^{k_t} \prod_{l = 1}^n  Y_l^{m_l},                                                
\end{aligned}
\end{equation}
where $w \in \fS_n$, $t \in \Z_{\geq 0}$, $k_s \in \Z_{>0}$, $m_l \in \Z$,
$1 \leq i_1 \leq \cdots \leq i_t \leq n$, $1 \leq j_1 \leq \cdots \leq j_t \leq n$  with $i_s = i_{s+1} \Rightarrow j_s < j_{s+1}$, and none of the indices~$i_r$ equal any of the indices $j_s$. Here $T_w$ ($w \in \fS_n$) is the standard basis of the Hecke algebra of type~$A_{n-1}$. 

The algebra $\bH^{\fgl_n}$ has a presentation by generators $T_k$ ($1 \leq k \leq n-1$), $Y_i^{\pm 1}$, $e_{ij}$ ($1 \leq i\neq j \leq n$) and relations~\eqref{eqn: Hecke relations}, \eqref{eqn: braid relations}, \eqref{eqn: T_k and Y_i},  Laurent relations for $Y_i^{\pm 1}$, and relations from Propositions~\ref{prop: relations among e_ij}--\ref{prop: moving Y_i to the right}. 
Further, $\bH^{\fgl_n}/(\tau - 1)\bH^{\fgl_n}$ $ \cong \cA$.
\end{theorem}

\begin{proof}
    Consider any monomial in the generators: $T_k$, $Y_i^{\pm 1}$, $e_{jl}$ ($j \neq l$). In step one, we apply the relations from Proposition~\ref{prop: moving Y_i to the right} and Lemmas~\ref{lemma: moving Y_i to the right} and~\ref{lemma: moving Y_i inverse to the right} to move $Y_i^{\pm 1}$ to the right past any directly adjacent $e_{jl}$. All of those relations are linear in~$e$'s, so this does not increase the number of $e$'s in any single term. In step two, we move all $T_k$'s completely to the left end of each term by using Proposition~\ref{prop: moving T_k to the left} and relations~\eqref{eqn: T_k and Y_i}. This does not increase in any single term the number of $e$'s, and in those terms where the number of~$e$'s stayed the same, this did not increase the number of $Y$'s. We then repeat steps one and two until all $Y$'s are to the right of any $e$'s and all $T$'s are at the left end of each monomial term. We achieve this in finitely many steps. Since $T_w$ ($w \in \fS_n$) form a basis of the Hecke algebra of type~$A_{n-1}$ (which follows from relations~\eqref{eqn: Hecke relations} and~\eqref{eqn: braid relations}), we can now assume that all monomial terms take the form
    $T_w ($product of $e$'s$)\prod_{l = 1}^n  Y_l^{m_l}$.

    Next, we use the relations from Proposition~\ref{prop: relations among e_ij} for the product of $e$'s to order them in accordance with the conditions on the indices as in~\eqref{eqn: PBW basis for H_gl}. We can handle by induction the terms where the number of~$e$'s has decreased after an application of a relation from Proposition~\ref{prop: relations among e_ij}, so we care only about the quadratic terms in $e$ in those relations, and these do not introduce any new~$Y$'s nor~$T$'s.
    This proves that the monomials~\eqref{eqn: PBW basis for H_gl} span the algebra~$\bH^{\fgl_n}$. 
    
    It remains to show that these monomials are linearly independent over~$\C_\tau$. It suffices to show that for $\tau = 1$ they are linearly independent over $\C$. This holds because at $\tau = 1$ they coincide with the PBW basis for the algebra $\cA$ from Section~\ref{sec: quantum group} (see formula~\eqref{eqn: PBW basis for cA}).

    It is straightforward to see that under the correspondence $T_k \leftrightarrow s_k$, $Y_i^{\pm 1} \leftrightarrow t_i^{\pm 1}$, 
    $e_{ij} \leftrightarrow E_{ij}^q$, the defining relations of $\bH^{\fgl_n}$ are just deformations of those of the algebra $\cA$ (relations from Proposition~\ref{prop: PBW basis for A} together with~\eqref{eqn: moving s_k to the right} and Coxeter relations), which reduce to those of $\cA$ when $\tau = 1$. The last part of the statement follows.
\end{proof}

\begin{remark}
    One may also consider a non-formal version of the algebra~$\bH^{\fgl_n}$ where $\tau = \lambda \in \C^\times$, that is the $\C$-algebra $$\bH^{\fgl_n} / (\tau - \lambda)\bH^{\fgl_n} \cong \bH^{\fgl_n} \otimes_{\C_\tau} \C,$$ where we use a ring homomorphism $\varphi\colon \C_\tau \to \C$ given by $\varphi(\tau) = \lambda$. Then, it follows from Theorem~\ref{thm: PBW basis for H_gl} that (the coset representatives of) the elements~\eqref{eqn: PBW basis for H_gl} give a basis of this algebra. 
\end{remark}

\section{Centre and a double centraliser property}\label{sec: centre}
In this section, we consider the DAHA $\bH_n$ and its subalgebra $\bH^{\fgl_n}$ defined in an analogous way as in Section~\ref{sec: subalgebra}, but over the field $\C(\tau)$ of rational functions in the variable $\tau$ instead of $\C_\tau$. An analogous proof shows that this version of~$\bH^{\fgl_n}$ satisfies the direct analogue of Theorem~\ref{thm: PBW basis for H_gl} as well. We now study some further properties of this algebra.

\begin{proposition}\label{prop: tY in centre}
    The element $\tY = \prod_{i=1}^n Y_i$ belongs to the centre $\cZ(\bH^{\fgl_n})$.
\end{proposition}

\begin{proof}
    From the defining relations of $\bH_n$, it follows that $\tY$ commutes with all~$Y_i^{\pm 1}$ and $T_k$.
    Also,  
    $\tY X_i = q X_i \tY$, hence $[\tY, X_iX_j^{-1}] = 0$, and it follows that $\tY$ commutes with all
    \begin{equation*}
        e_{ij} = (q-q^{-1})^{-1} X_i X_j^{-1} (T^{-1})_{j, n-1}^+ (Y_n - Y_n^{-1})T_{n-1, j}^-.
    \end{equation*}
    Thus, $\tY$ commutes with the generators of $\bH^{\fgl_n}$, as required.
\end{proof} 

\begin{remark}
    With the assignments~\eqref{eqn: trigonometric degeneration}, in the limit $\hbar \to 0$, 
    we get
    \begin{equation*}
        \lim_{\hbar \to 0} \frac{1 - \tY}{1-q} = 
        \sum_{i=1}^n X_i y_i - c\sum_{i < j} s_{ij} = eu,
    \end{equation*}
    (where we implicitly use the embedding~\eqref{eqn: embedding of H_n})
    which is the generator of the centre~$\cZ(H^{\fgl_n})$ from Section~\ref{sec: deg-0 part of RCA}. 
    In the polynomial representation in the variables $x_i$ such that $X_i = e^{x_i}$, the action of the element $eu$ is given (up to a constant) by the operator $\sum_{i=1}^n \partial_{x_i}$. Thus, $\tY$ corresponds to the total momentum. 
    \hfill \qedsymbol
\end{remark}

Let us now consider an arbitrary element $f \in \cZ(\bH^{\fgl_n})$. We can expand $f$ in a $\C(\tau)$-basis of monomials of the same form as those from Theorem~\ref{thm: PBW basis for H_gl}. That is, there exist finitely many monomials $M_1, M_2, \dots, M_N$ of the form~\eqref{eqn: PBW basis for H_gl}, and some
    $\lambda_1(\tau), \dots, \lambda_N(\tau) \in \C(\tau) \setminus \{0 \}$ such that 
    \begin{equation*}
        f = \lambda_1(\tau) M_1 + \cdots + \lambda_N(\tau) M_N.
    \end{equation*}
    We want to show that $f$ is a Laurent polynomial in $\tY$. We proceed by induction on $N$, the base case being $N=0$, when $f=0$, for which it trivially holds.
    We can assume that $\lambda_i$ do not have a pole at $\tau = 1$ and 
    that~$\lambda_i(1) \neq 0$ for some $i$,  
    since any $f \in \cZ(\bH^{\fgl_n})$ is proportional to some $\widetilde{f} \in \cZ(\bH^{\fgl_n})$ that satisfies this assumption, and if the claim holds for $\widetilde{f}$, then it follows for $f$, too. At $\tau = 1$,
    we have
    $$
        \lambda_1(1) M_1^{\tau=1} + \cdots +  \lambda_N(1) M_N^{\tau=1} \in 
        \cZ(\cA) = \langle \prod_{i=1}^n t_i, \ \prod_{i=1}^n t_i^{-1} \rangle,
    $$
    where we used that the centre of $\cA \cong \C\fS_n \ltimes (U_q(\fgl_n)/I_q)$ was given in Proposition~\ref{prop: centre of cA}. Here $M_i^{\tau=1}$ are the monomials $M_i$ with $Y_l$ replaced by $t_l$, $T_w$ by~$w$, and $e_{jl}$ by $E_{jl}^q$.

    Thanks to the PBW basis of monomials~\eqref{eqn: PBW basis for cA} for the algebra $\cA$, we can conclude that in the expansion of $f$, the monomials $M_i$ for which $\lambda_i(1) \neq 0$ must have the form~$\tY^m$ ($m \in \Z$). By subtracting those terms from $f$, we get an element which, by Proposition~\ref{prop: tY in centre}, belongs to $\cZ(\bH^{\fgl_n})$ and which has fewer terms in its PBW expansion than $f$, and so it must be a Laurent polynomial in $\tY$ by the induction hypothesis, hence so is $f$. Thus, we arrive at the following theorem. 
    \begin{theorem}\label{thm: centre of Hgln}
        The centre $\cZ(\bH^{\fgl_n})$ is generated by $\tY^{\pm 1}$.
    \end{theorem}

Let us now consider the subalgebra $\fA$ of $\bH_n$ generated by 
\begin{equation*}
    \C[X_1, \dots, X_n], \qquad \C[D_1, \dots, D_n], \qquad \C[Y_1^{\pm 1}, \dots, Y_n^{\pm 1}],
\end{equation*}
and $T_1, \dots, T_{n-1}$. We note that~$\bH^{\fgl_n} \subset \fA$. 
In the limit $\tau = 1$, those generators of~$\fA$ that are not~$T_k$ reduce to the 
generators of the $q$-Weyl algebra considered by Hayashi in~\cite{Hayashi}. 
The algebra~$\fA$ can be thought of also as a $q$-analogue of the RCA~$H_n$. Indeed, in the trigonometric limit $q \to 1$ given by~\eqref{eqn: trigonometric degeneration}, the algebra $\fA$ reduces to $H_n \subset \bH_n^{\trig}$.

\begin{remark} \label{rem: l=2 cycDAHA}
By~\cite[Section~3.7]{BEF}, the cyclotomic DAHA $\HH_{n,t}^2(Z,q^{-1})$ for $Z = (1,-1)$ is the subalgebra of~$\HH_{n,t}(q^{-1})$ generated by~$T_k$ ($1 \leq k \leq n-1$), $X_i$, $Y_i^{\pm 1}$, and $D_i^{\text{BEF}}$ ($1 \leq i \leq n$) given by~\eqref{eqn: D_i^(2)}. By using the isomorphism $h \circ g \circ \varphi \colon \HH_{n,t}(q^{-1}) \to \bH_n$ from Remark~\ref{rem: BEF's D_i}, we get that this subalgebra~$\HH_{n,t}^2(Z,q^{-1})\subset \HH_{n,t}(q^{-1})$ is isomorphic to the subalgebra of $\bH_n$ generated by $T_k$, $X_i$, $Y_i^{\pm 1}$, and
\begin{equation*}
    D_i  Y_i^{-1} T_{i,n-1}^+ T_{n-1, i}^- 
\end{equation*}
(see equality~\eqref{eqn: D_i vs D_i^(2)}), which coincides with the algebra $\fA$. That is, $\fA \cong \HH_{n,t}^2((1,-1),q^{-1})$. 
\hfill \qedsymbol
\end{remark} 

We will need the following basis of the algebra $\fA$. 
Another basis of~$\fA \cong \HH_{n,t}^2((1,-1),q^{-1})$ was considered in \cite{BEF} (see paragraph above Proposition~3.32 therein). 
\begin{proposition}\label{prop: basis of fA}
    The algebra $\fA$ has a $\C(\tau)$-basis 
    consisting of the monomials 
    \begin{equation}\label{eqn: PBW basis of fA}
        T_w M_X M_D M_Y,
    \end{equation}
    where $w \in \fS_n$, $M_X$ is a monomial in $X_i$, $M_Y$ a monomial in $Y_i^{\pm 1}$, and $M_D$ a monomial in $D_i$ such that for all~$i$ $M_X$ does not contain $X_i$ or $M_D$ does not contain $D_i$.
\end{proposition}
\begin{proof}
    Consider any monomial $M$ in the generators $T_k$, $X_i$, $D_i$, $Y_i^{\pm 1}$. Firstly, we will show that we can write $M$ as a linear combination of terms of the form~\eqref{eqn: PBW basis of fA}. We will proceed recursively based on the total power of $X$'s that appear in $M$. We apply the following procedure to $M$. 
    We note that the relations that we use below are all linear in $X$, and thus preserve or decrease the total degree in $X$.

    In step one, we use the relations from Lemma~\ref{lemma: properties of D_i} (1) and (3), and Lemmas~\ref{lemma: moving Y_i to the right} and~\ref{lemma: moving Y_i inverse to the right} to move $Y_i^{\pm 1}$ to the right past any directly adjacent $X_j$ or~$D_j$. 
    In step two, we move all $T_k$'s completely to the left end of each term by using Lemma~\ref{lemma: T_k and D_i} and relations~\eqref{eqn: T_k and X_i} and~\eqref{eqn: T_k and Y_i}. 
    We then repeat steps one and two until in each monomial term all $Y$'s are to the right of any $X$'s and~$D$'s, and all $T$'s are at its left end. We achieve this in finitely many steps. Note that all the relations used preserve the total degree in $X$ variables (and separately in~$D$ variables). Thus, at the end, all the monomial terms that were produced have the same total power of $X$'s (and of $D$'s) as the original monomial $M$ 
    and 
    are of the form
    $T_w \times ($product of $X$'s and $D$'s$) \times M_Y$ for some (not necessarily same) $T_w$ and $M_Y$.

    We now apply Lemma~\ref{lemma: properties of D_i} (4). It gives that $S_{ij}^\tau = [D_i, X_j]$ can be expressed in terms of $Y$ and $T$ variables, hence we can commute~$D$ with~$X$ up to a term with a lower total power of $X$'s (which we can handle by recursion). Furthermore, whenever we encounter~$X_i D_i$, we can replace it with an expression containing $Y$ and~$T$ only, which again leads to a decrease in the total power of $X$'s. 
    It follows that the monomials~\eqref{eqn: PBW basis of fA} span the algebra~$\fA$. 
    
    It remains to show that these monomials are linearly independent over~$\C(\tau)$. It suffices to show that for $\tau = 1$ they are linearly independent over $\C$.  
    We will work with the faithful polynomial representation of the DAHA. 
    Recall that at~$\tau = 1$, the elements $D_i$ act as the operators $d_i$ from Section~\ref{sec: quantum group}, while $T_k$ and $Y_i^{\pm 1}$ act   as $s_k$ and $t_i^{\pm 1}$, respectively. Therefore, for any $a_i, b_i \in \Z_{\geq 0}$ and $c_i \in \Z$ we have 
    \begin{equation}\label{leading terms 2}
        T_w \prod_{i=1}^n X_i^{a_i} \prod_{i=1}^n D_i^{b_i}  \prod_{i = 1}^n  Y_i^{c_i}\bigg\rvert_{\tau = 1} \sim w \prod_{i=1}^n X_i^{a_i-b_i}
        \prod_{i = 1}^n  t_i^{b_i + c_i}
        + \, \dots,
    \end{equation}
    where $\dots$ denotes terms in which the overall sum of the exponents on the~$t_i$'s is lower than in the above leading term, and $\sim$ denotes proportionality by a non-zero factor, which may depend on $q$. 

    Assume a non-trivial linear dependence of some monomials~\eqref{eqn: PBW basis of fA} at $\tau = 1$. This implies a non-trivial linear dependence of their corresponding leading (with highest degree in $t_i$) terms, whose form is shown in the right-hand side of~\eqref{leading terms 2}. 
    By the assumptions on the monomials~\eqref{eqn: PBW basis of fA}, either $a_i = 0$ or $b_i = 0$, hence different monomials~\eqref{eqn: PBW basis of fA} lead to distinct leading terms.  However, operators
    $w\prod_{i=1}^n X_i^{n_i} \prod_{i=1}^n t_i^{n'_i}$ for different $n_i, n'_i \in \Z$ and $w \in \fS_n$ are linearly independent as operators on $\C[X_1^{\pm 1}, \dots, X_n^{\pm 1}]$. 
    We arrived at a contradiction, which completes the proof.
\end{proof}

\begin{lemma}\label{lemma: deg M_X = deg M_D}
    Any monomial~\eqref{eqn: PBW basis of fA} with $\deg M_X = \deg M_D$ belongs to~$\bH^{\fgl_n}$.
\end{lemma}
\begin{proof}
We permute the elements $X_i$ and~$D_j$ in such a monomial so as to pair them up into a product of elements $e_{ij}$. This can be done up to a combination of terms of lower degree in $X$ with equal degree in $D$, since $S_{ij}^\tau = [D_i, X_j]$ 
is by Lemma~\ref{lemma: properties of D_i} equal to an expression of degree 0 in both $X$ and $D$. We re-express the lower degree terms produced by applying these commutation relations via the basis~\eqref{eqn: PBW basis of fA}. Each resulting term will still have equal $X$ and~$D$ degree, so
the statement follows inductively by degree in~$X$.
\end{proof}

Remark~\ref{rem: l=2 cycDAHA} enables us to prove the following proposition. 
\begin{proposition}\label{prop: action on polys}
    The algebra~$\fA$ has an irreducible representation~$\psi$ on polynomials $\C(\tau)[X_1, \dots, X_n]$ given by 
    \begin{equation}\label{eqn: action on polys}
        \begin{aligned}
            &\psi(T_k) = \tau s_k + \frac{(\tau - \tau^{-1})X_{k+1}}{X_k - X_{k+1}}(s_k -1), \\
            &\psi(X_i) = X_i, \quad \psi(\tau) = \tau, \\
            &\psi(Y_i) = \tau^{n-1} \psi(T_{i,n-1}^+)\pi^{-1}\psi((T^{-1})_{1, i-1}^+), \\
            &\psi(D_i) = (q-q^{-1})^{-1}X_i^{-1}\psi\left((T^{-1})_{i, n-1}^+ (Y_n - Y_n^{-1})T_{n-1, i}^-\right),
        \end{aligned}
    \end{equation}
    where 
    $\pi^{-1} = (s_{n-1} \cdots s_2 s_1)t_1 = t_n (s_{n-1} \cdots s_2 s_1)$.
    (see formula~\eqref{eqn: action of pi}). 
\end{proposition}
\begin{proof}
    By Remark~\ref{rem: l=2 cycDAHA}, the algebra $\fA$ maps isomorphically via $\varphi^{-1} \circ g^{-1} \circ h$ to $\HH_{n,t}^2((1, -1), q^{-1})$. The latter has by~\cite[Proposition~3.6]{BEF} an action on $\C(\mathbf{t})[X_1, \dots, X_n]$ via $\rho_{\text{BEF}} \equiv \rho$ defined in~\cite[Proposition~3.3]{BEF} (with $q$ replaced by $q^{-1}$). Let $\psi' = \rho_{\text{BEF}} \circ \varphi^{-1} \circ g^{-1} \circ h$, which then gives an action of $\fA$ on $\C(\mathbf{t})[X_1, \dots, X_n]$. We have 
    \begin{align*}
        &\psi'(T_k) = \rho_{\text{BEF}}(T_{n-k}^{-1}) = 
        \mathbf{t}^{-1} s_{n-k} + \frac{(\mathbf{t}^{-1} - \mathbf{t})X_{n-k}}{X_{n-k+1} - X_{n-k}}(s_{n-k}-1), \\
        &\psi'(X_i) = X_{n-i+1},  \quad
        \psi'(\tau) = \mathbf{t}^{-1}, \\
        &\psi'(Y_i) =\rho_{\text{BEF}}(Y_{n-i+1}^{-1}) =\mathbf{t}^{1-n}\rho_{\text{BEF}}((T^{-1})_{n-i,1}^-) \pi'\rho_{\text{BEF}}(T_{n-1,n-i+1}^-), \\
        &\psi'(D_i) = (q-q^{-1})^{-1}X_{n-i+1}^{-1} \psi'\left((T^{-1})_{i, n-1}^+ (Y_n - Y_n^{-1})T_{n-1, i}^-\right),
    \end{align*}
    where $\pi' = (s_1s_2\cdots s_{n-1})t_n$.
    The representation~\eqref{eqn: action on polys} of $\fA$ is obtained from the module $\C(\mathbf{t})[X_1, \dots, X_n]$ by relabelling~$\mathbf{t}$ to~$\tau^{-1}$ and~$X_i$ to~$X_{n-i+1}$.

    The proof of irreducibility is similar to that of~\cite[Proposition~2.1]{Hayashi}. Let~$V$ be a non-trivial submodule, and choose in it a non-zero element 
    $v = \sum_{\mathbf{m}} a_{\mathbf{m}} X_1^{m_1} \cdots X_n^{m_n}$, 
    where $\mathbf{m} = (m_1, \dots, m_n) \in \Z_{\geq 0}^n$ and $a_{\mathbf{m}} \in \C(\tau)$. 
    We can assume that those $a_{\mathbf{m}}$ with maximal $\sum_{i=1}^n m_i$ among $\{\mathbf{m} \in \Z_{\geq 0}^n \colon a_{\mathbf{m}} \neq 0 \}$ do not have a pole at $\tau=1$, and that at least one of them is non-zero at $\tau=1$, say
    for $\mathbf{m}' = (m_1', \dots, m_n')$. Since the action of $\psi(D_i)$ reduces the degree of a polynomial, we get that 
    $
        \psi(D_1^{m_1'} \cdots D_n^{m_n'})v \in \C(\tau)
    $
    and is well-defined at $\tau = 1$.
    Moreover, it must be a non-zero element of $\C(\tau)$ because at $\tau = 1$ it equals
    \begin{equation*}
        d_1^{m_1'} \cdots d_n^{m_n'} (v) = a_{\mathbf{m}'}(1)[m_1']!_q \cdots [m_n']!_q,
    \end{equation*}
    which belongs to $\C^\times$ as $q$ is not a root of unity. Here we use for any $m \in \Z_{\geq 0}$ the notation
    \begin{equation*}
        [m]!_q = [m]_q [m-1]_q \cdots [2]_q [1]_q, \qquad [m]_q = \frac{q^m - q^{-m}}{q-q^{-1}}, 
    \end{equation*} 
    and that the operators $d_i$ are given by formula~\eqref{eqn: d_i}, which implies that each application of $d_i$ reduces the power $\widetilde{m}_i$ of $X_i$ by one  in every monomial term with $\widetilde{m}_i > 0$ and it annihilates every monomial with $\widetilde{m}_i = 0$.
    
    It follows that $1 \in V$, and by acting on $1$ by combinations of~$\psi(X_i)$, we get that $V = \C(\tau)[X_1, \dots, X_n]$.
\end{proof}

\begin{corollary}\label{cor: action on polys}
    The subalgebra $\bH^{\fgl_n} \subset \fA$ acts on $\C(\tau)[X_1, \dots, X_n]$. Moreover, this action preserves for all $k \in \Z_{\geq 0}$ the subspace $\C(\tau)[X_1, \dots, X_n]^{(k)}$ of homogeneous polynomials of degree $k$, and this is an irreducible $\bH^{\fgl_n}$-module.
\end{corollary}
\begin{proof}
    Irreducibility is proved similarly to the proof of Proposition~\ref{prop: action on polys}. Using the same notation, this time we have $\sum_{i=1}^n m_i' = k$. An arbitrary monomial $X_1^{a_1}\cdots X_n^{a_n} \in \C(\tau)[X_1, \dots, X_n]^{(k)}$ can be obtained as 
    \begin{equation*}
        \psi(cX_1^{a_1}\cdots X_n^{a_n}D_1^{m_1'}\cdots D_n^{m_n'})v
    \end{equation*}
    for suitable $c \in \C(\tau) \setminus \{0\}$,
    where $X_1^{a_1}\cdots X_n^{a_n}D_1^{m_1'}\cdots D_n^{m_n'} \in \bH^{\fgl_n}$ by Lemma~\ref{lemma: deg M_X = deg M_D} since $\sum_{i=1}^n a_i = k = \sum_{i=1}^n m_i'$.
\end{proof}
The preceding corollary generalises the fact that the polynomial representation of the algebra $H^{\fgl_n}$ preserves the space $\C[X_1, \dots, X_n]^{(k)}$, which is an irreducible module for it, and that this space is also preserved by the algebra~$\cA$ from Section~\ref{sec: quantum group} (cf.\ also~\cite[Theorem 4.1(A)]{Hayashi} for $\tau = 1$). 

\begin{remark}
    The assignments~\eqref{eqn: action on polys} almost coincide with those of the polynomial representation of the DAHA $\bH_n$ given in Section~\ref{sec: DAHA} above, except that the image of $Y_i$ in~\eqref{eqn: action on polys} has an extra factor of $\tau^{n-1}$ (the action of~$Y_i$ in the polynomial representation can be deduced from relations~\eqref{eqn: Y_i in terms of pi}). 
    
    A way to think about this is that the operators from the polynomial representation on $\C(\tau)[X_1^{\pm 1}, \dots, X_n^{\pm 1}]$ in Section~\ref{sec: DAHA} formally preserve also the space $(\prod_{i=1}^n X_i)^{\log_q \tau^{n-1}}\C(\tau)[X_1^{\pm 1}, \dots, X_n^{\pm 1}]$, which induces another action of $\bH_n$ on $\C(\tau)[X_1^{\pm 1}, \dots, X_n^{\pm 1}]$ under which the subalgebra $\fA$ preserves the subspace $\C(\tau)[X_1, \dots, X_n]$ and acts as given in Proposition~\ref{prop: action on polys}.
    \hfill \qedsymbol
    \end{remark}
    

We are now going to show that
\begin{equation}\label{eqn: double centraliser}
    \begin{aligned}
        &\bH^{\fgl_n} = C_{\fA}(\tY), \\
        &C_{\fA}(\bH^{\fgl_n}) = \cZ(\bH^{\fgl_n}),
    \end{aligned}
\end{equation}
where $C_A(B) = \{a \in A \colon [a,b]=0, \, \forall b \in B \}$ denotes the centraliser, and $\cZ(\bH^{\fgl_n}) = \langle \tY, \tY^{-1} \rangle$ by Theorem~\ref{thm: centre of Hgln}. 
This statement is a $q$-generalisation of the property that
\begin{align*}
    &H^{\fgl_n} = C_{H_n}(eu), \\
    &C_{H_n}(H^{\fgl_n}) = \cZ(H^{\fgl_n}) = \langle eu \rangle.
\end{align*}
The first of the latter equalities follows from the fact that the RCA~$H_n$ has a natural grading such that its faithful polynomial representation is a graded one. The element $eu$ acts (up to a constant) as the grading operator $\sum_{i=1}^n X_i \partial_{X_i}$, so it only commutes with the degree zero part of $H_n$, which is precisely $H^{\fgl_n}$. 
The second line follows from the fact that $\deg eu = 0$, hence $eu \in H^{\fgl_n}$, so the previous sentence implies that $C_{H_n}(H^{\fgl_n}) = \cZ(H^{\fgl_n})$, which equals $\langle eu \rangle$ by~\cite{FH}.

The fact that $eu$ is essentially the grading operator has a $q$-counterpart in the following property of $\tY$.
    Since $\tY = \pi^{-n}$ \cite[p.~101]{Cherednik}, by using formula~\eqref{eqn: action of pi} we get
    \begin{equation}\label{eqn: grading operator tY}
        \tY(X_1^{a_1} X_2^{a_2} \cdots X_n^{a_n}) = q^{\sum_{i=1}^n a_i}X_1^{a_1} X_2^{a_2} \cdots X_n^{a_n}.
    \end{equation}
    That is, $\tY$ acts in the polynomial representation as a grading operator.

Let us now provide a proof of relations~\eqref{eqn: double centraliser}.
From Theorem~\ref{thm: centre of Hgln}, it follows that $\bH^{\fgl_n} \subseteq C_{\fA}(\tY)$. We now prove the reverse inclusion.
Let $f \in C_\fA(\tY)$. 
Since $\tY D_i = q^{-1}D_i \tY$, we have 
\begin{equation*}
   \tY T_w M_X M_D M_Y = q^{\deg M_X - \deg M_D} T_w M_X M_D M_Y \tY.
\end{equation*}
This implies that the expansion of $f$ in the basis of $\fA$ from Proposition~\ref{prop: basis of fA} can contain only those monomials where $\deg M_X = \deg M_D$, as $q$ is not a root of unity. 
Hence, $f \in \bH^{\fgl_n}$ by Lemma~\ref{lemma: deg M_X = deg M_D}. We have proved that $\bH^{\fgl_n} = C_{\fA}(\tY)$. 

Suppose now that $f \in C_\fA(\bH^{\fgl_n})$. Then it must, in particular, commute with $\tY \in \bH^{\fgl_n}$. Thus, by the same argument as above, we get $f \in \bH^{\fgl_n}$. Therefore, $C_\fA(\bH^{\fgl_n}) = \cZ(\bH^{\fgl_n}) = \langle \tY, \tY^{-1} \rangle$ by Theorem~\ref{thm: centre of Hgln}, as required. 

Thus, we have established the following theorem.
\begin{theorem}
    We have $C_{\fA}(\tY) = \bH^{\fgl_n}$ and $C_{\fA}(\bH^{\fgl_n}) = \cZ(\bH^{\fgl_n})$.
\end{theorem}
This theorem implies that $\bH^{\fgl_n}$ coincides with the degree zero part of $\fA$, where the grading on $\fA$ is inherited from the DAHA.

Related to the previous considerations,
from Corollary~\ref{cor: action on polys} it follows that $V = \C(\tau)[X_1, \dots, X_n]$ is a $(\bH^{\fgl_n}, \cZ(\bH^{\fgl_n}))$-bimodule, which by Proposition~\ref{prop: action on polys} is an irreducible $\fA$-module. It admits the decomposition
\begin{equation*}
        V = \bigoplus_{k=0}^\infty W_k \otimes_{\C(\tau)} U_k,
\end{equation*}
where $W_k = \C(\tau)[X_1, \dots, X_n]^{(k)}$, which by Corollary~\ref{cor: action on polys} is an irreducible module of $\bH^{\fgl_n}$, and $U_k =  \C(\tau)$ is the irreducible (one-dimensional) module of~$\cZ(\bH^{\fgl_n})$ determined by $\tY \mapsto q^k$ (this is by formula~\eqref{eqn: grading operator tY} the action of $\tY$ on $W_k$). If~$k \neq l \in \Z_{\geq 0}$, then $W_k \ncong W_l$, because their dimensions as vector spaces differ, and $U_k \ncong U_l$ since $q$ is not a root of unity.

This can be compared with the usual (undeformed) $(\fgl_n, \fgl_1)$-Howe duality, which gives an analogous multiplicity-free decomposition of the bosonic Fock space $V$ corresponding to mutually centralising actions of $\fgl_n$ and $\fgl_1$. The former is acting via the Jordan--Schwinger (oscillator) representation, and the latter is generated by the Euler operator $\sum_{i=1}^n X_i \partial_{X_i}$. The physical interpretation is that $\fgl_1$ measures the total boson number, and the number operator eigenspaces (fixed boson number sectors of the Fock space) are preserved by the action of $\fgl_n$, since each generator of $\fgl_n$ creates one boson and annihilates one, and these eigenspaces are irreducible $\fgl_n$-modules.

\section{Related integrable systems}\label{sec: integrable systems}
  In Section~\ref{sec: D}, we considered a family of pairwise-commuting elements~$D_i$.
  We now introduce certain pairwise-commuting $\cD^{(l_1, l_2)}_i$ of a more general form  depending on additional parameters $l_1, l_2 \in \Z_{\geq 0}$, $a_j \in \C$ ($j=-l_1, \dots, l_2$). The action of symmetric combinations of $\cD^{(l_1, l_2)}_i$ on the space of symmetric Laurent polynomials $\C_\tau[X_1^{\pm 1}, \dots, X_n^{\pm 1}]^{\fS_n}$ will lead to families of commuting $q$-difference operators related to those of Macdonald--Ruijsenaars and Van Diejen. We recover $D_i = (q-q^{-1})^{-1}\cD_i^{(1,1)}$ for $a_{-1} = -1$, $a_0 = 0$, and $a_1 = 1$.
  
  We define 
  $$\cD_n = \cD^{(l_1, l_2)}_n = X_n^{-1}\sum_{j=-l_1}^{l_2} a_j Y_n^j,$$ 
  and for $1 \leq i \leq n-1$, we let 
\begin{align}\label{eqn: definition of cD_i^l}
    \cD_i = \cD^{(l_1, l_2)}_i &= T_{i, n-1}^+ \cD_n T_{n-1, i}^- \nonumber \\
    &= X_i^{-1}(T^{-1})_{i, n-1}^+ \left(\sum_{j=-l_1}^{l_2} a_j Y_n^j \right)T_{n-1, i}^-.
\end{align}

The following table summarises the notations for the various commuting elements of the DAHA considered in this paper. 
\begin{center}
\begin{tabular}{ |c|c| } 
 \hline
 $\cD^{(l_1, l_2)}_i$ & Elements~\eqref{eqn: definition of cD_i^l}.  \\ 
 $\cD_i$ & Shorthand for $\cD^{(l_1, l_2)}_i$ for implicit, fixed $l_1, l_2$, and $a_j$ ($j=-l_1, \dots, l_2$).  \\ 
 $D_i$ & Elements~\eqref{eqn: definition of D_i}, proportional to $\cD_i^{(1,1)}$ with $a_{-1} = -1$, $a_0 = 0$, $a_1 = 1$. \\
 $D_i^{(l)}$ & A subset of the generators of a cyclotomic DAHA from~\cite{BEF}. \\
 $D_i^{\text{BEF}}$ & $D_i^{(l)}$ for $l=2$. \\
 \hline
\end{tabular}
\end{center}

\vspace{1em}

We have $T_k^{-1}\cD_k T_k^{-1} = \cD_{k+1}$, and $[T_k, \cD_i] = 0$ for $i \neq k, k+1$ by an analogous proof as for Lemma~\ref{lemma: T_k and D_i}.
In Corollary~\ref{prop: D^l commute} below, we show that~$\cD_i$ pairwise commute. Firstly, we obtain a family of endomorphisms of a subalgebra of the DAHA, which we now define.

Let $\bH_n^-$ be the (unital, associative) $\C_\tau$-algebra generated by $T_k$ ($1 \leq k \leq n-1$) and $\C_\tau[Z_1, \dots, Z_n]$, $\C_\tau[Y_1^{\pm 1}, \dots, Y_n^{\pm 1}]$ subject to the following relations: 
\begin{align}
    &(T_k-\tau)(T_k +\tau^{-1}) = 0, \nonumber \\ 
    &T_l T_{l+1} T_l = T_{l+1} T_l T_{l+1} \quad (1 \leq l \leq n-2), \quad \ [T_k, T_l] = 0 \textnormal{ if } |k-l| > 1, \nonumber \\ 
    &T_k^{-1}Z_kT_k^{-1} = Z_{k+1}, \qquad [T_k, Z_i] = 0 \textnormal{ for } i \neq k, k+1, \label{eqn: T_k and X_i inverse}\\
    &T_k^{-1} Y_k T_k^{-1} = Y_{k+1}, \quad \, [T_k, Y_i] = 0 \textnormal{ for } i \neq k, k+1, \nonumber \\ 
    &\tY Z_i = q^{-1} Z_i \tY,  \label{eqn: Y tilde relation} \\ 
    &Y_2 Z_1 = Z_1 Y_2 T_1^2,  \label{eqn: Y_2 and X_1}
\end{align}
where $\tY = \prod_{i=1}^n Y_i$. Let us note that relations~\eqref{eqn: T_k and X_i inverse} imply that $T^+_{i,n-1}Z_n = Z_i (T^{-1})^+_{i,n-1}$ (cf.~Lemma~\ref{lem: DAHA identity}).

There is an algebra embedding $\phi \colon \bH_n^- \to \bH_n$ given by 
\begin{align*}
    \phi(T_k) = T_k, \quad \phi(Z_i) = X_i^{-1}, \quad \phi(Y_i^{\pm 1}) = Y_i^{\pm 1},
\end{align*}
whose image contains the elements $\cD_i$.
The next proposition gives a family of endomorphisms of the algebra $\bH_n^-$.
\begin{proposition}
    Let $f(z) \in \C[z, z^{-1}]$ be an arbitrary single-variable Laurent polynomial. There is an algebra endomorphism $\theta = \theta_f$ of $\bH_n^-$ determined by $\theta(T_k) = T_k$, $\theta(Y_i) = Y_i$, and 
    \begin{equation*}
        \theta(Z_i) = T_{i, n-1}^+ Z_n f(Y_n) T_{n-1, i}^-
        \qquad \left(= Z_i (T^{-1})_{i, n-1}^+ f(Y_n) T_{n-1, i}^-\right).
    \end{equation*}
\end{proposition}
\begin{proof}
    It suffices to check that $\theta$ preserves relations~\eqref{eqn: T_k and X_i inverse}--\eqref{eqn: Y_2 and X_1} and commutativity of $Z_i$.
    Firstly, 
    \begin{equation*}
        \theta(T_k^{-1}Z_kT_k^{-1}) = T_{k+1, n-1}^+ Z_n f(Y_n) T_{n-1, k+1}^- = \theta(Z_{k+1}), 
    \end{equation*}
    as required. Suppose now that $i \neq k, k+1$. Then either $i > k+1$, in which case it is easy to see that $\theta(T_k)$ commutes with $\theta(Z_i)$. Or $i < k$, in which case twice using Lemma~\ref{lemma: braid identities} and that $[T_{k-1}, Z_n] = 0 = [T_{k-1}, f(Y_n)]$, we get
    \begin{align*}
        \theta(T_kZ_i) &= T_k T_{i, n-1}^+ Z_n f(Y_n) T_{n-1, i}^- \\
        &= T_{i, n-1}^+ T_{k-1} Z_n f(Y_n) T_{n-1, i}^- \\
        &= T_{i, n-1}^+ Z_n f(Y_n) T_{k-1} T_{n-1, i}^- \\
        &= T_{i, n-1}^+ Z_n f(Y_n) T_{n-1, i}^- T_k 
        = \theta(Z_i T_k).
    \end{align*}
    This completes the proof that $\theta$ preserves relations~\eqref{eqn: T_k and X_i inverse}.

    Secondly, we have
    \begin{align*}
        &\theta(\tY Z_i) 
        = T_{i, n-1}^+ \tY Z_n f(Y_n) T_{n-1, i}^- 
        = q^{-1} T_{i, n-1}^+Z_n f(Y_n) T_{n-1, i}^- \tY 
        = q^{-1} \theta(Z_i\tY),
    \end{align*}
    hence $\theta$ preserves relations~\eqref{eqn: Y tilde relation}.

    Thirdly, since $\theta(Z_1) = Z_1 (T^{-1})_{1, n-1}^+ f(Y_n) T_{n-1, 1}^-$, we see, due to relation~\eqref{eqn: Y_2 and X_1}, that it will follow that $\theta(Y_2^{-1} Z_1 Y_2)  = \theta(Z_1T_1^{-2})$ if we show
    \begin{align*}
        T_1^{-1}Y_2^{-1} (T^{-1})_{1, n-1}^+ f(Y_n) T_{n-1, 1}^- Y_2 =  (T^{-1})_{2, n-1}^+ f(Y_n) T_{n-1, 2}^- T_1^{-1}.
    \end{align*}
    The left-hand side of that can be rearranged as
    \begin{align*}
        Y_1^{-1} (T^{-1})_{2, n-1}^+ f(Y_n) T_{n-1, 1}^- Y_2 &=
        (T^{-1})_{2, n-1}^+ f(Y_n) T_{n-1, 2}^- Y_1^{-1} T_1 Y_2 \\
        &= (T^{-1})_{2, n-1}^+ f(Y_n) T_{n-1, 2}^- T_1^{-1},
    \end{align*}
    as required. Thus, $\theta$ preserves the relation~\eqref{eqn: Y_2 and X_1} as well.

    Finally, we prove that $[\theta(Z_i), \theta(Z_j)] = 0$ for all $i,j$. From relations~\eqref{claim 1}, it follows that 
    $f(Y_n) T_{n-1}Z_n = T_{n-1}Z_n f(Y_{n-1})$. Also, by the same argument as for~\eqref{claim 2}, 
    we have $[f(Y_{n-1})f(Y_n), T_{n-1}] = 0$, since $f(Y_{n-1})f(Y_n)$ is symmetric in $Y_n$, $Y_{n-1}$. With that, the proof that $\theta(Z_{n-1})$, $\theta(Z_n)$ commute is fully analogous to that for $D_{n-1}$, $D_n$ in the proof of Lemma~\ref{lemma: properties of D_i} (2). The proof for other $i,j$ is completed just as in that lemma, using that $[T_k, \theta(Z_i)] = 0$ if $i \neq k, k+1$ (which is seen similarly to Lemma~\ref{lemma: T_k and D_i}).
\end{proof}

In particular, the elements $\cD_i \in \bH_n$ defined by~\eqref{eqn: definition of cD_i^l} commute.

\begin{corollary}\label{prop: D^l commute}
     We have $[\cD_i, \cD_j]=0$ for all $i, j$.
\end{corollary}
\begin{proof}
Let $f(z) = \sum_{j=-l_1}^{l_2} a_j z^j$. Then $(\phi \circ \theta_f)\left(Z_i\right) = \cD_i$.
Hence $[\cD_i, \cD_j]= (\phi \circ \theta_f )\left([Z_i, Z_j] \right) = 0$.
\end{proof}

\begin{remark} 
    Commutativity of $\cD_i$ in the special case of $l_2 = 0$ follows from~\cite[Corollary~3.22~(i)]{BEF}, where a different method was used. Indeed, the elements~$D_i^{(l)}$ considered there satisfy $a_{-l}^{-1}\cD_i = (hg\varphi)(D_{n-i+1}^{(l)})$ for $l_2 = 0$, $l = \max\{j \in \{0, \dots, l_1\}\colon a_{-j} \neq 0\}$, and~$Z_i$ expressed in terms of $a_i$. Here, $h$, $g$, and $\varphi$ are the isomorphisms from Remark~\ref{rem: BEF's D_i}.
    \hfill \qedsymbol
\end{remark}

\begin{remark}
    The algebra $\bH^{\fgl_n}$ is the subalgebra of $\bH_n$ generated by $T_k$, $Y_i^{\pm 1}$, and $X_i \cD_j$ ($i \neq j$) for $l_1=l_2=1$, and $a_{-1}=-1$, $a_0 = 0$, $a_1 = 1$, since then~$\cD_n = X_n^{-1}(Y_n - Y_n^{-1}) = (q-q^{-1})D_n$.
    It would be interesting to see if the subalgebra of $\bH_n$ generated by $T_k$, $Y_i^{\pm 1}$, and $X_i \cD_j$ ($i \neq j$) for more general~$l_1$, $l_2$, and $a_j$ --- equivalently, the degree zero subalgebra of a general cyclotomic DAHA ---  has good properties as well. \hfill \qedsymbol
\end{remark}

Recall the polynomial representation of $\bH_n$ on $\C_\tau[X_1^{\pm 1}, \dots, X_n^{\pm 1}]$, mentioned in Section~\ref{sec: DAHA}, in which the element $\pi^{-1}$ acts according to formula~\eqref{eqn: action of pi} as $(n, \dots, 1)t_1 = t_n (n, \dots, 1)$, the action of $X_i^{\pm 1}$ is by multiplication, and the Hecke generators $T_k$ act according to formula~\eqref{eqn: action of T_k} as
\begin{equation*}
    \tau s_k + \frac{\tau - \tau^{-1}}{X_kX_{k+1}^{-1} - 1}(s_k - 1) = \frac{\tau^{-1}X_{k+1}- \tau X_k}{X_{k+1}-X_k}s_k + \frac{(\tau - \tau^{-1})X_{k+1}}{X_{k+1}-X_k}.
\end{equation*}
It follows that the elements $T_k^{-1} = T_k + \tau^{-1} - \tau$ act as 
\begin{equation}\label{eqn: action of T_k inverse} 
    \frac{\tau^{-1}X_{k+1}- \tau X_k}{X_{k+1}-X_k}s_k 
    + \frac{(\tau - \tau^{-1})X_k}{X_{k+1}-X_k}.
\end{equation} 
By combining relations~\eqref{eqn: definition of cD_i^l} and~\eqref{eqn: Y_i in terms of pi}, we get 
\begin{equation}\label{eqn: cD_i^l i.t.o. pi}
    \cD_i = X_i^{-1} \left(\sum_{j=1}^{l_2} a_j \left((T^{-1})_{i, n-1}^+\pi^{-1}(T^{-1})_{1,i-1}^+\right)^j + \sum_{j=1}^{l_1}a_{-j} \left(T_{i-1,1}^- \pi T_{n-1, i}^- \right)^j + a_0\right). 
\end{equation}

We now prove that the action of symmetric combinations of $\cD_i$ preserves the subspace $\C_\tau[X_1^{\pm 1}, \dots, X_n^{\pm 1}]^{\fS_n}$. 
Let $\C[\cD_1, \dots, \cD_n]^{\fS_n}$ denote the set of all symmetric combinations of $\cD_i$, where $\fS_n$ acts by permuting the indices. 
We will make use of the following lemma.

\begin{lemma}\label{lemma: T_k and D}
    We have $[T_k, D] = 0$ for any $D \in \C[\cD_1, \dots, \cD_n]^{\fS_n}$ for all $k$.
\end{lemma}
\begin{proof}
   The subalgebra $\langle T_1, \dots, T_{n-1}, Y_1^{\pm 1},$ $ \dots, Y_n^{\pm 1} \rangle \subset \bH_n$ is an affine Hecke algebra of type $GL_n$, whose centre 
   contains $\C[Y_1, \dots, Y_n]^{\fS_n}$ (see e.g.~\cite[Lemma~1.3.12]{Cherednik} and a historical comment in~\cite{Lusztig}).  
    We have $[\cD_i, \cD_j]=0$ by Corollary~\ref{prop: D^l commute}; also, recall that $T_k^{-1}\cD_k T_k^{-1} = \cD_{k+1}$ and $[T_k, \cD_i] = 0$ for $i \neq k, k+1$. Thus, there is an epimorphism from the subalgebra $\langle T_1, \dots, T_{n-1}, Y_1, \dots, Y_n \rangle$ to the subalgebra
    $\langle T_1, \dots, T_{n-1}, \cD_1, \dots, \cD_n \rangle$ given by $T_k \mapsto T_k$, $Y_i \mapsto \cD_i$.
    The claim follows since $T_k$ commute with any element of $\C[Y_1, \dots, Y_n]^{\fS_n}$. 
\end{proof}

\begin{proposition}\label{prop: preservation of invariants}
    Let $D \in \C[\cD_1, \dots, \cD_n]^{\fS_n}$. Then the action of $D$ on $\C_\tau[X_1^{\pm 1}, \dots, X_n^{\pm 1}]$ preserves the space of invariants 
    $\C_\tau[X_1^{\pm 1}, \dots, X_n^{\pm 1}]^{\fS_n}$.
\end{proposition}
\begin{proof}
    From formula~\eqref{eqn: action of T_k}, it follows that $p \in \C_\tau[X_1^{\pm 1}, \dots, X_n^{\pm 1}]$ is $\fS_n$-invariant if and only if $T_k(p) = \tau p$ for all $k$. 
    The claim thus follows from the fact that $D$ commutes with all $T_k$ by Lemma~\ref{lemma: T_k and D}. 
\end{proof}

Let $f$ be any operator on $\C_\tau[X_1^{\pm 1}, \dots, X_n^{\pm 1}]$ of the form 
\begin{equation}\label{eqn: diff-refl operator}
    f = \sum_{\substack{ i \in \{1, \dots, n\}\\ j \in \Z, \, w \in \fS_n}} g_{i,j,w} t_i^j w, \qquad g_{i,j,w} \in \C_\tau(X_1, \dots, X_n).
\end{equation}
For instance, the action of any $D \in \C[\cD_1, \dots, \cD_n]^{\fS_n}$ can be written in that form. The operator $\Res(f)$ is defined by
\begin{equation*}
    \Res(f) \coloneqq \sum_{\substack{ i \in \{1, \dots, n\}\\ j \in \Z, \, w \in \fS_n}} g_{i,j,w} t_i^j. 
\end{equation*}
Thus, $\Res(f)$ is a $q$-difference operator with rational coefficients. On elements of the space $\C_\tau[X_1^{\pm 1}, \dots, X_n^{\pm 1}]^{\fS_n}$ it acts identically to~$f$. In particular, if the latter preserves this space then so does~$\Res(f)$. 

We note that the operators that represent the action of the elements $D$ are not invariant with regard to the action of the symmetric group~$\fS_n$ 
in the sense that they do not commute with~$\fS_n$ as elements of the crossed product of $\C \fS_n$ with the algebra of $q$-difference operators with rational coefficients, but the operators $\Res(D)$ are invariant.

\begin{theorem}\label{thm: Res D}
    The operators $\Res(D)$ for $D \in \C[\cD_1, \dots, \cD_n]^{\fS_n}$ are pairwise-commuting, $\fS_n$-invariant, and preserve the space
    $\C_\tau[X_1^{\pm 1}, \dots, X_n^{\pm 1}]^{\fS_n}$. Furthermore,
an algebraic basis $p_1, \ldots, p_n\in  \C[X_1, \ldots, X_n]^{\fS_n}$ gives $n$ algebraically independent operators $\Res p_i(\cD_1, \ldots, \cD_n)$.

\end{theorem}
\begin{proof}
Preservation of $\C_\tau[X_1^{\pm 1}, \dots, X_n^{\pm 1}]^{\fS_n}$ follows from Proposition~\ref{prop: preservation of invariants}.

Let $D, \tD \in  \C[\cD_1, \dots, \cD_n]^{\fS_n}$ and $p \in \C_\tau[X_1^{\pm 1}, \dots, X_n^{\pm 1}]^{\fS_n}$. By using Corollary~\ref{prop: D^l commute}, we get $\Res(D)\Res(\tD)p = D \tD p = \tD D p = \Res(\tD)\Res(D)p$. 
Thus $\Res(D)$ and $\Res(\tD)$ commute when restricted to $\C_\tau[X_1^{\pm 1}, \dots, X_n^{\pm 1}]^{\fS_n}$, which implies the commutativity of $\Res(D)$ and $\Res(\tD)$ (for the statement see~\cite[Theorem~4.5]{Kirillov}, \cite[Theorem~3.3]{Ch'95}, for a proof see \cite[Proposition 3.2]{KN}, and for the additive case~\cite[Lemma 3.7]{Ch'92}).

For any $w \in \fS_n$, we have $w \Res(D) w^{-1}p = Dp$ thanks to Proposition~\ref{prop: preservation of invariants}. Thus $w \Res(D) w^{-1}$ and $\Res(D)$ are equal as operators on $\C_\tau[X_1^{\pm 1}, \dots, X_n^{\pm 1}]^{\fS_n}$. 
As in the preceding paragraph, it follows that $w \Res(D) w^{-1} = \Res(D)$.

The final claim follows by specialisation $\tau=1$ which reduces $\cD_i$ to an operator in the variable $X_i$.
\end{proof}

Explicitly, for $l_1=l_2=1$, and for the degree one symmetric combination~$\sum_{i=1}^n\cD_i$, we get the following formula for the corresponding integrable Hamiltonian.
\begin{proposition}
    With $a=a_1$, $b=a_{-1}$, and $c=a_0$, we have
    \begin{equation}
        \begin{aligned}\label{eqn: vDE-like operator}
            M_{a,b,c} \coloneqq \Res&\left(\sum_{i=1}^n \cD_i^{(1,1)}\right) = a \tau^{1-n} \sum_{i=1}^n \frac{1}{X_i} \left( \prod_{\substack{j = 1 \\ j \neq i}}^n \frac{\tau^2 X_i - X_j}{X_i - X_j} \right)t_i \\
         &+ b\tau^{1-n} \sum_{i=1}^n \frac{1}{X_i} \left( \prod_{\substack{j = 1 \\ j \neq i}}^n \frac{ X_i - \tau^2 X_j}{X_i - X_j} \right)t_i^{-1} + c \sum_{i=1}^n \frac{1}{X_i}.
        \end{aligned}
    \end{equation}
\end{proposition}
The proof will follow from the next lemma. Let 
\begin{align*}
    \cD_i^+ &= X_i^{-1}(T^{-1})_{i, n-1}^+\pi^{-1}(T^{-1})_{1,i-1}^+, \\
    \cD_i^- &= X_i^{-1}T_{i-1,1}^- \pi T_{n-1, i}^-,
\end{align*}
so that by relation~\eqref{eqn: cD_i^l i.t.o. pi} we have
\begin{equation*}
    \Res\left(\sum_{i=1}^n \cD_i^{(1,1)}\right) = a \Res\left(\sum_{i=1}^n \cD_i^+\right) + b \Res\left(\sum_{i=1}^n \cD_i^-\right) + 
    c \sum_{i=1}^n \frac{1}{X_i}.
\end{equation*}
Then the following statement holds.
\begin{lemma}[\textnormal{cf.\ \cite[Lemma 5.3]{BF}}]
    For all $m \in \{1, \dots, n \}$, let
    \begin{equation*}
        E_m^+ = \tau^{1-n}\sum_{i=m}^n \frac{1}{X_i} A_{i,m} t_i, \quad \textnormal{ where } \quad  A_{i,m} = \prod_{\substack{j=m \\ j \neq i}}^n \frac{\tau^2 X_i - X_j}{X_i- X_j}.
    \end{equation*}
    Then 
    \begin{equation}\label{eqn: Em+}
        \Res\left(\sum_{i=m}^n \cD_i^+\right) = E_m^+.
    \end{equation}
    Furthermore,
    \begin{equation}\label{eqn: Em-}
        \Res\left(\sum_{i=1}^m \cD_i^-\right) = E_m^-,
    \end{equation}
    where 
    \begin{equation*}
        E_m^- = \tau^{n-2m+1}\sum_{i=1}^m \frac{1}{X_i} B_{i,m} t_i^{-1} \quad \textnormal{ with } \quad  B_{i,m} = \prod_{\substack{j=1 \\ j \neq i}}^m \frac{X_i - \tau^2 X_j}{X_i- X_j}.
    \end{equation*}
\end{lemma}
The proof is analogous to that of~\cite[Lemma 5.3]{BF}. For convenience, we indicate here how to adapt that proof in our context.
\begin{proof}
    We give the proof of equality~\eqref{eqn: Em-}, since~\eqref{eqn: Em+} works similarly. 
    By using formulas~\eqref{eqn: action of T_k} and \eqref{eqn: action of pi}, we get
    \begin{equation*}
        \Res(\cD_i^-) = \tau^{n-i}\Res(X_i^{-1}T_{i-1,1}^- t_1^{-1}) = \tau^{n-i}\Res\left((T^{-1})_{i-1,1}^- X_1^{-1} t_1^{-1}\right).
    \end{equation*}
    In particular, $\Res(\cD_1^-) = \tau^{n-1}X_1^{-1}t_1^{-1}$ , from which equality~\eqref{eqn: Em-} for $m=1$ follows. Thus, it now suffices to show that for all $m=1, \dots, n-1$ we have
    \begin{equation}\label{eqn: Em- equivalent}
        \Res(\cD_{m+1}^-) = E_{m+1}^- - E_m^-. 
    \end{equation}
    For $i \neq m+1$, we have
    \begin{equation*}
        B_{i, m+1} = \frac{X_i - \tau^2 X_{m+1}}{X_i- X_{m+1}} \prod_{\substack{j=1 \\ j \neq i}}^m \frac{X_i - \tau^2 X_j}{X_i- X_j} = 
        \left(1 + \frac{(1-\tau^2)X_{m+1}}{X_i - X_{m+1}} \right) B_{i,m}. 
    \end{equation*}
    Hence, relation~\eqref{eqn: Em- equivalent} is equivalent to 
    \begin{equation}
    \label{eqn: Em- equivalent 2}
         \Res(\cD_{m+1}^-) = \frac{\tau^{n-2m-1}}{X_{m+1}}B_{m+1, m+1} t_{m+1}^{-1} + \sum_{i=1}^m \frac{\tau^{n-2m-1}(1-\tau^2)}{X_i-X_{m+1}} B_{i,m} t_i^{-1}.
    \end{equation}
    Let the right-hand side of equality \eqref{eqn: Em- equivalent 2} be the definition of $R_{m+1}$ for $m=0,$ $1,\dots, n-1$. We trivially have $R_1 = \Res(\cD_1^-)$.
    We note that 
    $$\Res(\cD_{m+1}^-) = \Res(T_m^{-1}\cD_m^- T_m^{-1} ) = \tau^{-1}\Res(T_m^{-1}\cD_m^-).$$
    Thus, to prove equality \eqref{eqn: Em- equivalent 2} for all $m=1, \dots, n-1$, it suffices to prove that $\Res(T_m^{-1} R_m) = \tau R_{m+1}$. Indeed, then we will get 
    $$\Res(\cD_2^-) = \tau^{-1} \Res(T_1^{-1} \cD_1^-) =  \tau^{-1} \Res(T_1^{-1} \Res(\cD_1^-)) = \tau^{-1} \Res(T_1^{-1} R_1) = R_2,$$
    as required; similarly for $\Res(\cD_3^-)$, etc.

    By using formula~\eqref{eqn: action of T_k inverse} for the action of $T_m^{-1}$, we compute
    \begin{align*}
        \Res(T_m^{-1} R_m) &= \frac{\tau^{n-2m}}{X_{m+1}}B_{m+1, m+1} t_{m+1}^{-1} 
        + \frac{\tau^{n-2m}(1-\tau^2)}{X_m-X_{m+1}} B_{m,m} t_m^{-1} \\
        &\quad + \sum_{i=1}^{m-1} \frac{\tau^{n-2m}(1-\tau^2)(X_{m+1} - \tau^2X_m)}{(X_{m+1} - X_m)(X_i-X_{m+1})} B_{i,m-1} t_i^{-1} \\
        &\quad -  \sum_{i=1}^{m-1} \frac{\tau^{n-2m}(1-\tau^2)^2X_m}{(X_{m+1} - X_m)(X_i-X_m)} B_{i,m-1} t_i^{-1}.
    \end{align*}
    The proof that $\Res(T_m^{-1} R_m) = \tau R_{m+1}$ is completed by using in the preceding equality that
    \begin{align*}
        \frac{1}{X_{m+1} - X_m}&\left( \frac{X_{m+1} - \tau^2X_m}{X_i-X_{m+1}} 
        - \frac{(1-\tau^2)X_m}{X_i - X_m} \right)B_{i,m-1} \\
        &= \frac{X_i-\tau^2 X_m}{(X_i-X_{m+1})(X_i-X_m)} B_{i,m-1} = \frac{1}{X_i - X_{m+1}}B_{i,m}
    \end{align*}
    for $i \neq m$. This completes the proof of the lemma.
\end{proof}

\begin{remark}\label{rem: vDE (3.13a)}
    The operator $M_{a,b,c}$ given by~\eqref{eqn: vDE-like operator} for a special choice of the parameters $a, b, c$ can be related to a particular limit of the operator (3.13a) from \cite{vDE} as follows. In the latter operator, let us make a translation of the center-of-mass of the form $q^{x_j} \to \kappa^{-1} q^{x_j}$ ($j = 1,...,n$) for a constant~$\kappa$, make the change of variables $X_j = q^{-x_j}$ (in particular, the additive shift operators $T_j$, $T_j^{-1}$ become respectively $t_j^{-1}$, $t_j$ in our notation), put $t = \tau^2$, and multiply the whole operator by $\kappa$. Then in the limit $\kappa \to 0$, one obtains the operator~\eqref{eqn: vDE-like operator} for $a = - \tau^{n-1} \htt_0$, $b = - \tau^{1-n}\htt_1$, and $c = \htt_0 + \htt_1$. Further specialisation of this operator at $\htt_1 = 0$ appeared in~\cite[Example~3.24]{BEF}. \hfill \qedsymbol
\end{remark}

For more general values of $l_1$ and $l_2$, and the degree one symmetric combination~$\sum_{i=1}^n\cD_i$, the following proposition takes place. 
\begin{proposition}
    We have
    \begin{align*}
        \Res&\left(\sum_{i=1}^n \cD_i^{(l_1, l_2)}\right) = \tau^{l_2(1-n)} a_{l_2}  \sum_{i=1}^n \frac{1}{X_i} \left( \prod_{k=0}^{l_2-1}  \prod_{\substack{j = 1 \\ j \neq i}}^n \frac{q^k \tau^2 X_i - X_j}{q^k X_i - X_j} \right) t_i^{l_2} \\
         &+ \tau^{l_1(1-n)} a_{-l_1} \sum_{i=1}^n \frac{1}{X_i} \left( \prod_{k=0}^{l_1-1} \prod_{\substack{j = 1 \\ j \neq i}}^n \frac{ X_i - q^k \tau^2 X_j}{X_i - q^k X_j} \right)t_i^{-l_1} + \ldots,
    \end{align*}
    where $\ldots$ denotes ``non-leading terms'', that is terms with shifts
    $\prod_{j=1}^n t_j^{k_j}$ such that 
    $-l_1 < k_j < l_2$ for all $j$. Moreover, in each term, either all $k_j$ are non-negative with $\sum_{j=1}^n k_j \leq l_2$, or all $k_j$ are non-positive with $\sum_{j=1}^n k_j \geq -l_1$. 
\end{proposition}
The proof is similar to the calculation of the leading term of a general Macdonald operator, polynomial in $Y$ variables, from \cite[Proposition~3.4]{Ch'95}.
\begin{proof}    
By using~\eqref{eqn: cD_i^l i.t.o. pi}, we get
\begin{align*}
    \Res&\bigg(\sum_{i=1}^n \cD_i^{(l_1, l_2)}\bigg) = \Res\bigg(\sum_{i=1}^n X_i^{-1} \sum_{j=1}^{l_2} a_j \left((T^{-1})_{i, n-1}^+t_n(n, \dots, 1)(T^{-1})_{1,i-1}^+\right)^j \\
    &\quad+ \sum_{i=1}^n X_i^{-1}\sum_{j=1}^{l_1}a_{-j} \left(T_{i-1,1}^- t_1^{-1}(1, \dots, n) T_{n-1, i}^- \right)^j + a_0\sum_{i=1}^n X_i^{-1} \bigg).
\end{align*}
From there we see, due to formulas~\eqref{eqn: action of T_k inverse}, that the term containing $t_1^{l_2}$ can only come from
\begin{equation*}
    \Res\bigg(a_{l_2} X_1^{-1} \left((T^{-1})_{1, n-1}^+ t_n(n, \dots, 1)\right)^{l_2}\bigg).
\end{equation*}
Hence, by using~\eqref{eqn: action of T_k inverse}, we can compute this $t_1^{l_2}$ term to be 
\begin{align*}
    \tau^{l_2(1-n)}&a_{l_2} \frac{1}{X_1} 
     \bigg(\prod_{j = 2}^n \frac{\tau^2 X_1 - X_j}{X_1 - X_j}t_1 \bigg)^{l_2}  \\
     &=
     \tau^{l_2(1-n)}a_{l_2} \frac{1}{X_1} 
     \bigg(\prod_{k=0}^{l_2-1}\prod_{j = 2}^n \frac{q^k\tau^2 X_1 - X_j}{q^kX_1 - X_j}\bigg)t_1^{l_2}.
\end{align*}
We can use $\fS_n$-invariance (see Theorem~\ref{thm: Res D}) to deduce the coefficient at $t_i^{l_2}$ for any $i$.
Similarly, one can compute explicitly the coefficient at~$t_n^{-l_1}$, and then use $\fS_n$-invariance again to complete the proof of the proposition.
\end{proof}

For example, for $l_1 = 1$ and $l_2 = 2$, we get the following integrable Hamiltonian
\begin{align*}
    &\Res\left(\sum_{i=1}^n \cD_i^{(1, 2)}\right) = \tau^{2(1-n)} a_2  \sum_{i=1}^n \frac{1}{X_i} \left( \prod_{\substack{j = 1 \\ j \neq i}}^n \frac{(\tau^2 X_i - X_j)(q \tau^2 X_i - X_j)}{(X_i - X_j)(q X_i - X_j)} \right) t_i^2 \\
         &+ q\tau^{2(1-n)} a_2\sum_{1 \leq i < j \leq n} \frac{(\tau^2-1)(\tau^2-q)(X_i + X_j)}{(qX_i - X_j)(qX_j - X_i)} 
         \left(\prod_{\substack{l = 1 \\ l \neq i,j}}^n \frac{(\tau^2 X_i - X_l)(\tau^2 X_j - X_l)}{(X_i- X_l)(X_j- X_l)} \right) t_i t_j \\
         &+ \tau^{1-n} a_{1}  \sum_{i=1}^n \frac{1}{X_i} \left( \prod_{\substack{j = 1 \\ j \neq i}}^n \frac{\tau^2 X_i - X_j}{X_i - X_j} \right) t_i
         + \tau^{1-n} a_{-1} \sum_{i=1}^n \frac{1}{X_i} \left( \prod_{\substack{j = 1 \\ j \neq i}}^n \frac{ X_i - \tau^2 X_j}{X_i -  X_j} \right)t_i^{-1} + a_0 \sum_{i=1}^n \frac{1}{X_i} 
\end{align*}
which is a generalisation of the operator~\cite[(5.18)]{ChalykhFairon}, to which it reduces for 
$a_{-1} = 0$.
If we put $a_0 = a_{-1} = a_2 = 0$ (with $a_1\neq 0$), then the resulting operator is gauge-equivalent to the standard Macdonald--Ruijsenaars operator (see \cite{BEF}).

\subsection{Systems with two types of particles}\label{sec: 2 types of particles}
In this subsection, we obtain a generalisation of the Macdonald--Ruijsenaars system with Morse term~\cite[(2.1)]{vDE} (a particular limit of which, \cite[(3.13a)]{vDE}, was mentioned in Remark~\ref{rem: vDE (3.13a)} above). That system was introduced by Van Diejen in~\cite{vD} and studied further by Van Diejen and Emsiz in~\cite{vDE}. Our generalisation introduces into the system a second, different set of particles interacting with each other and also with the original set of particles. In the case of Macdonald--Ruijsenaars systems, such two-types-of-particles generalisations were considered in \cite{Ch'00, SV}.

One way to obtain such a generalisation is to make use of representation theory of the DAHA of type $GL_n$ (see Remark~\ref{rem: restriction} below). There are though advantages in taking another approach as follows. 
The operator~\cite[(2.1)]{vDE} can be obtained from the Koornwinder operator (i.e.\ the operator of Macdonald--Ruijsenaars type for the root system $(C_n^\vee, C_n)$) by a limit in which the centre of mass is sent to infinity in such a way that only the type~$A$ permutation symmetry of the operator survives and the Morse terms appear~\cite{vD}. We now take the generalised Koornwinder operator~\cite[(5.12)]{FS} introduced by Silantyev and one of the authors, and apply to it an analogous centre-of-mass-to-infinity limit so that only the $A_{N_1} \times A_{N_2}$ symmetry of the operator survives, where $N_1$, $N_2$ are the numbers of the particles of each of the two types. In particular, we want the terms depending on the relative positions of the particles to survive and we want to preserve the asymmetry between the two types.

As we need multiple versions of shift operators in this section, let us summarise in the following table the notations used.

\begin{center}
\begin{tabular}{ |c|c| } 
 \hline
 Notation & Meaning  \\ 
 \hline
 $\cT_{A}^{a}$ & additive shift by $a$ in variable $A$  \\ 
 $\cT_{A}$ & unit additive shift in variable $A$  \\ 
 $t_{A}^a$ & multiplicative $a$-shift in variable $A$ \\
 $t_i$ & multiplicative $q$-shift in variable $X_i$ \\
 \hline
\end{tabular}
\end{center}

\vspace{1em}
Let the parameter $q$ in \cite{FS} be denoted $\textbf{q}$ here in order to distinguish it. 
In the operator~\cite[(5.12)]{FS}, let us make the substitutions $x_i \to x_i + R$, $y_i \to y_i + R + \log(\textbf{q}s^{-1})$, $a \to a e^R$, $b \to b e^R$, $c \to c e^{-R}$, $d \to d e^{-R}$, and then take $R \to \infty$. Recall that $\textbf{q} = e^{\hbar/2}$ and $s = e^{\xi/2}$ in~\cite{FS}. In order to make connection with the notations used in~\cite{vDE}, let us in the resulting limit make the replacements $\textbf{q} = q^{-1/2}$, $s = t^{1/2}$ (so that $q = e^{-\hbar}$ and $t = e^{\xi}$), $a=t_1$, $b=t_2$, $c = t_0^{-1}$, $d = t_3^{-1}$, $x_i \to \hbar x_i$, and $y_i \to \hbar y_i$. The latter two replacements mean that the shift operators $\cT_{x_i}^{\e \hbar}$ and $\cT_{y_i}^{\e \xi}$ for $\e \in \{\pm 1\}$ become $\cT_{x_i}^{\e}$ and $\cT_{y_i}^{-\e \log(t)/\log(q)}$, respectively. Then we get the following Hamiltonian
\begin{align}\label{eqn: generalised vDE operator}
        &H_{t_0,t_1,t_2,t_3}=\sum_{i=1}^{N_1} (1-t_1 q^{x_i})(1-t_2 q^{x_i}) \left(\prod_{\substack{j=1 \\ j \neq i}}^{N_1} \frac{t^{-1} - q^{x_i - x_j}}{1-q^{x_i - x_j}}  \right)
        \left( \prod_{j=1}^{N_2} \frac{q-q^{x_i - y_j}}{1-q^{x_i - y_j}} \right)(\cT_{x_i}-1) \nonumber \\
        &+ \frac{t_1 t_2}{qt_0 t_3} \sum_{i=1}^{N_1}(1-t_0 q^{x_i})(1-t_3 q^{x_i}) 
        \left(\prod_{\substack{j=1 \\ j \neq i}}^{N_1} \frac{t - q^{x_i - x_j}}{1-q^{x_i - x_j}} \right) \left( \prod_{j=1}^{N_2} \frac{t - q^{x_i - y_j}}{qt- q^{x_i - y_j}} \right)(\cT_{x_i}^{-1}-1) \nonumber\\
        &+ \frac{1-q}{1-t^{-1}} \sum_{i=1}^{N_2} (1-t_1 q^{y_i})(1-t_2 q^{y_i})
        \left( \prod_{j=1}^{N_1} \frac{t^{-1} - q^{y_i - x_j}}{1-q^{y_i - x_j}} \right)   \nonumber\\
        &\hspace{11em}\times \left( \prod_{\substack{j=1 \\ j \neq i}}^{N_2} \frac{q-q^{y_i - y_j}}{1-q^{y_i - y_j}} \right)
        (\cT_{y_i}^{-\log(t)/\log(q)}-1) \nonumber\\
        &+\frac{t_1 t_2(1-q^{-1})}{qt_0 t_3(1-t)} \sum_{i=1}^{N_2}
        (1-t_0 t q^{y_i+1})(1-t_3 t q^{y_i+1}) 
        \left( \prod_{j=1}^{N_1} \frac{q^{-1} - q^{y_i - x_j}}{q^{-1}t^{-1} - q^{y_i - x_j}} \right)  \nonumber\\
         &\hspace{11em}\times
        \left( \prod_{\substack{j=1 \\ j \neq i}}^{N_2} 
         \frac{q^{-1} - q^{y_i-y_j}}{1-q^{y_i-y_j}} \right)
        (\cT_{y_i}^{\log(t)/\log(q)}-1).
\end{align}

For $N_2 = 0$ and $N_1 = n$, the operator~\eqref{eqn: generalised vDE operator} reduces to the Van Diejen--Emsiz operator~\cite[(2.1)]{vDE} up to a factor of $\sqrt{qt_0t_3/(t_1t_2)}$. Also, if in~\eqref{eqn: generalised vDE operator} we put $n = N_1 + N_2$, $y_i = x_{N_1+i}$, and $t=q^{-1}$ (that is $\xi = \hbar$), then we get the operator~\cite[(2.1)]{vDE} with $t=q^{-1}$.

By applying the same limiting procedure to the set of quantum integrals of the generalised Koornwinder operator found in~\cite[Proposition~5.6]{FS}, we get quantum integrals for the Hamiltonian~\eqref{eqn: generalised vDE operator}.

Let us now explain what the differential limit of the operator~\eqref{eqn: generalised vDE operator} is. Make the replacements $x_i \to \frac{x_i}{\hbar}$, $y_i \to \frac{y_i}{\hbar}$, let $q=e^\hbar$, $t = e^{-\hbar c}$, and define $-(t_1+t_2) = a_1 \hbar$, $t_1 t_2 = a_2 \hbar$, $-(t_0+t_3) = a_3 \hbar$, and $t_0 t_3 = a_4 \hbar$, where $a_1, a_3, a_4 \in \C$ do not depend on $\hbar$ and $a_2 = a_4(1+\a \hbar)$ with $\alpha \in \C$. Then, the expansion of the operator~\eqref{eqn: generalised vDE operator} in powers of $\hbar$ gives 
\begin{equation}
\begin{aligned}\label{eqn: hbar expansion}
    &\hbar^{-2} H_{t_0,t_1,t_2,t_3} = \sum_{i=1}^{N_1} \partial_{x_i}^2 + c \sum_{i=1}^{N_2} \partial_{y_i}^2 -
    \sum_{i < j}^{N_1} \frac{2c}{e^{x_i}-e^{x_j}}(e^{x_i} \partial_{x_i} - e^{x_j} \partial_{x_j}) \\
    & -\sum_{i=1}^{N_1}\sum_{j=1}^{N_2} \frac{2}{e^{x_i}-e^{y_j}}(e^{x_i} \partial_{x_i} - c e^{y_j} \partial_{y_j}) - \sum_{i < j}^{N_2} \frac{2}{e^{y_i}-e^{y_j}}(e^{y_i} \partial_{y_i} - e^{y_j} \partial_{y_j}) \\
    & + (a_1-a_3)\bigg(\sum_{i=1}^{N_1} e^{x_i} \partial_{x_i} + \sum_{i=1}^{N_2} e^{y_i} \partial_{y_i} \bigg) + \beta_1 \sum_{i=1}^{N_1} \partial_{x_i} + \beta_2 \sum_{i=1}^{N_2} \partial_{y_i} + O(\hbar)  
\end{aligned}
\end{equation}
for $\beta_1 = 2c(N_1-1) + 2N_2 + 1 -\a$ and $\beta_2 = 2N_2 + 2cN_1 - c - \a$. 
The operator~\eqref{eqn: hbar expansion} is gauge-equivalent via conjugation by the function
\begin{align*}
    &\exp\bigg(\frac{\beta_1}{2} \sum_{i=1}^{N_1}x_i + \frac{\beta_2}{2c} \sum_{i=1}^{N_2}y_i+
    \frac{a_1-a_3}{2}\sum_{i=1}^{N_1} e^{x_i} + \frac{a_1-a_3}{2c}\sum_{i=1}^{N_2} e^{y_i}
    \bigg) \\
    &\quad \times\bigg( \prod_{i < j}^{N_1} (e^{x_i} - e^{x_j})^{-c} \bigg) \bigg( \prod_{i=1}^{N_1}\prod_{j=1}^{N_2} (e^{x_i} - e^{y_j})^{-1} \bigg)  \prod_{i < j}^{N_2} (e^{y_i} - e^{y_j})^{-1/c} 
\end{align*}
to the operator
\begin{align*}
    &\sum_{i=1}^{N_1} \partial_{x_i}^2 + c \sum_{i=1}^{N_2} \partial_{y_i}^2 - \sum_{i<j}^{N_1} \frac{c(c+1)}{2\sinh^2(\frac{x_i-x_j}{2})} - \sum_{i=1}^{N_1}\sum_{j=1}^{N_2}\frac{c+1}{2\sinh^2(\frac{x_i-y_j}{2})}
    \\
    &\quad -\sum_{i<j}^{N_2} \frac{c+1}{2c \sinh^2(\frac{y_i-y_j}{2})} + \sum_{i=1}^{N_1}(A e^{x_i} + Be^{2x_i}) + \frac{1}{c} \sum_{i=1}^{N_2}(A e^{y_i} + Be^{2y_i})
\end{align*}
with $A = \frac12(a_1-a_3)(\alpha-2)$ and $B=-\frac14(a_1-a_3)^2$,
which is a generalisation of the hyperbolic Calogero--Moser Hamiltonian with Morse terms~\eqref{eqn: Inozemtsev} introduced by Inozemtsev in \cite{Inozemtsev} to the case of two types of particles.

From the operator \eqref{eqn: generalised vDE operator}, we also obtain a generalisation to the case of two types of particles of the Hamiltonian~\cite[(3.13a)]{vDE}. Indeed, if in~\eqref{eqn: generalised vDE operator} we put $t_3 = 1$, define $\htt_i$ ($i=0,1,2$) by $t_0 = q^{-1}\htt_1\htt_2$, $t_1 = \htt_0 \htt_2$, and $t_2 = \htt_0 \htt_1$ following~\cite[(3.12b)]{vDE}, and then take the limit $\htt_2 \to 0$, we get

\begin{align}\label{eqn: generalised vDE operator (3.13a)}
        H_{\htt_0, \htt_1} &= \sum_{i=1}^{N_1} (1-\htt_0\htt_1 q^{x_i})\left(\prod_{\substack{j=1 \\ j \neq i}}^{N_1} \frac{t^{-1} - q^{x_i - x_j}}{1-q^{x_i - x_j}}  \right)
        \left( \prod_{j=1}^{N_2} \frac{q-q^{x_i - y_j}}{1-q^{x_i - y_j}} \right)(\cT_{x_i}-1) \nonumber \\
        &+ \htt_0 {}^2 \sum_{i=1}^{N_1}(1-q^{x_i}) 
        \left(\prod_{\substack{j=1 \\ j \neq i}}^{N_1} \frac{t - q^{x_i - x_j}}{1-q^{x_i - x_j}} \right)\left( \prod_{j=1}^{N_2} \frac{t - q^{x_i - y_j}}{qt- q^{x_i - y_j}} \right)(\cT_{x_i}^{-1}-1)\nonumber \\
        &+ \frac{1-q}{1-t^{-1}} \sum_{i=1}^{N_2} (1-\htt_0\htt_1 q^{y_i})
        \left( \prod_{j=1}^{N_1} \frac{t^{-1} - q^{y_i - x_j}}{1-q^{y_i - x_j}} \right)  \left( \prod_{\substack{j=1 \\ j \neq i}}^{N_2} \frac{q-q^{y_i - y_j}}{1-q^{y_i - y_j}} \right) \nonumber \\
        &\hspace{22em}\times
        (\cT_{y_i}^{-\log(t)/\log(q)}-1) \nonumber \\
        &+\frac{\htt_0 {}^2(1-q^{-1})}{1-t} \sum_{i=1}^{N_2}
        (1- t q^{y_i+1}) 
        \left( \prod_{j=1}^{N_1} \frac{q^{-1} - q^{y_i - x_j}}{q^{-1}t^{-1} - q^{y_i - x_j}} \right)  
        \left( \prod_{\substack{j=1 \\ j \neq i}}^{N_2} 
         \frac{q^{-1} - q^{y_i-y_j}}{1-q^{y_i-y_j}} \right) \nonumber \\
        &\hspace{22em}\times
        (\cT_{y_i}^{\log(t)/\log(q)}-1).
\end{align}
For $N_2 = 0$ and $N_1 = n$, this reduces to the operator~\cite[(3.13a)]{vDE} up to a factor of $\htt_0$.

Let us now consider the analogous limit of the operator~\eqref{eqn: generalised vDE operator (3.13a)} as the limit described in Remark~\ref{rem: vDE (3.13a)}, just additionally making also the change $q^{y_i} \to \kappa^{-1}q^{y_i}$ and introducing the variables $Y_i = q^{-y_i}$. Let us denote a $q$-multiplicative shift operator in the variables $X_i$ by $t_{X_i}^q$ and analogously for the variables $Y_i$ (so that $\cT_{x_i}^{\pm 1}$ and $\cT_{y_i}^{\pm \log(t)/\log(q)}$ become respectively $t_{X_i}^{q^{\mp 1}}$ and $t_{Y_i}^{\tau^{\mp 2}}$). Then, in this limit, we obtain the multiplicative operator 

    \begin{align}\label{eqn: vDE-like operator generalised}
        \widetilde{H}_{\htt_0, \htt_1} &=\ta \sum_{i=1}^{N_1}\frac{1}{X_i}
        \left(\prod_{\substack{j=1 \\ j \neq i}}^{N_1} \frac{\tau^2X_i - X_j}{X_i - X_j} \right)\left( \prod_{j=1}^{N_2} \frac{\tau^2X_i - Y_j}{q\tau^2 X_i - Y_j} \right)t_{X_i}^q  \nonumber \\
        &\quad+\frac{\ta(1-q)}{1-\tau^{-2}} \sum_{i=1}^{N_2}
        \frac{1}{Y_i}
        \left( \prod_{j=1}^{N_1} \frac{q^{-1}Y_i - X_j}{q^{-1}\tau^{-2}Y_i - X_j} \right)  
        \left( \prod_{\substack{j=1 \\ j \neq i}}^{N_2} 
         \frac{q^{-1}Y_i -Y_j}{Y_i - Y_j} \right)
        t_{Y_i}^{\tau^{-2}} \nonumber \\
        &\quad+ \tb \sum_{i=1}^{N_1} \frac{1}{X_i}\left(\prod_{\substack{j=1 \\ j \neq i}}^{N_1} \frac{\tau^{-2}X_i - X_j}{X_i - X_j}  \right)
        \left( \prod_{j=1}^{N_2} \frac{qX_i-Y_j}{X_i - Y_j} \right)t_{X_i}^{q^{-1}} \nonumber \\
        &\quad+ \frac{\tb(1-q)}{1-\tau^{-2}} \sum_{i=1}^{N_2} \frac{1}{Y_i}
        \left( \prod_{j=1}^{N_1} \frac{\tau^{-2}Y_i - X_j}{Y_i - X_j} \right)   \left( \prod_{\substack{j=1 \\ j \neq i}}^{N_2} \frac{qY_i-Y_j}{Y_i - Y_j} \right)
        t_{Y_i}^{\tau^2} \nonumber \\
        &\quad + \tc \sum_{i=1}^{N_1} \frac{1}{X_i} + \frac{\tc (1-q)}{1-\tau^{-2}} \sum_{i=1}^{N_2}\frac{1}{Y_i},
    \end{align}
where 
$\ta = -\htt_0 {}^2$, $\tb = -\htt_0\htt_1$, and $\tc = \htt_0(\htt_0 + \htt_1) = -(\ta + \tb)$,
and where we twice used the following identity (which generalises to two types of particles the identity from the top of~\cite[p.~1621]{vDE}):
\begin{equation}\label{eqn: polynomial identity}
\begin{aligned}
    &\sum_{i=1}^{N_1} z_i
    \left( \prod_{\substack{j=1 \\ j \neq i}}^{N_1} \frac{t z_j - z_i}{z_j - z_i} \right) \prod_{j=1}^{N_2} \frac{s w_j - z_i}{w_j - z_i} \\
    &+
    \frac{1-s}{1-t} \sum_{i=1}^{N_2} w_i 
    \left(\prod_{j=1}^{N_1} \frac{t z_j - w_i}{z_j - w_i} \right)
    \prod_{\substack{j=1 \\ j \neq i}}^{N_2} \frac{s w_j - w_i}{w_j - w_i} 
    = \sum_{i=1}^{N_1}z_i + \frac{1-s}{1-t}\sum_{i=1}^{N_2} w_i.
\end{aligned}
\end{equation}
More specifically, each term $A(\cT-1)$ in the operator~\eqref{eqn: generalised vDE operator (3.13a)} splits as $A \cT - A\cdot1$, and in the limit all the $A \cdot 1$ contributions corresponding to the shifts $t_{X_i}^q$ and~$t_{Y_i}^{\tau^{-2}}$ collapse via the identity~\eqref{eqn: polynomial identity} with $z_i = X_i^{-1}$, $w_i = Y_i^{-1}$, $t = \tau^{-2}$, $s = q$ to 
\begin{equation}\label{eqn: 1st 2 terms}
    -\ta \sum_{i=1}^{N_1} \frac{1}{X_i} - \frac{\ta (1-q)}{1-\tau^{-2}} \sum_{i=1}^{N_2}\frac{1}{Y_i}.
\end{equation}
Similarly,  in the limit, all the $A \cdot 1$ contributions corresponding to the shifts $t_{X_i}^{q^{-1}}$ and~$t_{Y_i}^{\tau^2}$ collapse via the identity~\eqref{eqn: polynomial identity} with $z_i = X_i^{-1}$, $w_i = q\tau^2Y_i^{-1}$, $t = \tau^2$, $s = q^{-1}$ to~\eqref{eqn: 1st 2 terms} with $\ta$ replaced by $\tb$.

The operator~\eqref{eqn: vDE-like operator generalised} generalises the Hamiltonian $M_{a,b,c}$ given by~\eqref{eqn: vDE-like operator}, for special values of $a, b, c$, 
to a Hamiltonian of a system containing two types of particles. 
Indeed, if we put $N_2 = 0$, $N_1 = n$ in~\eqref{eqn: vDE-like operator generalised}, we recover $M_{a,b,c}$ with $a = - \tau^{n-1} \htt_0 {}^2$, $b = - \tau^{1-n}\htt_0\htt_1$, and $c = \htt_0 (\htt_0 + \htt_1)$.

\begin{remark}\label{rem: restriction}
    An alternative way to arrive at the operator~\eqref{eqn: vDE-like operator generalised} and also a version of it with arbitrary parameters $\ta$, $\tb$, $\tc$ is to apply to the Hamiltonian~\eqref{eqn: vDE-like operator} a restriction procedure, similar to those considered in~\cite{FS}, for a suitable submodule of the polynomial representation of the $GL_n$-type DAHA.

Moreover, this approach should lead to integrable generalisations of the Hamiltonians 
$\Res\left(\sum_{i=1}^n \cD_i^{(l_1, l_2)}\right)$ for general $l_1, l_2$ to the case of two types of particles. 
\end{remark}

\subsection*{Acknowledgements}
We are very grateful to G.\,Bellamy, O.\,Chalykh, A.\,Khoroshkin, C.\,Korff, J.\,Lamers, D.\,Muthiah, P.\,Samuelson, A.\,Shapiro, A.\,Silantyev, and B.\,Vlaar for useful discussions and comments.
We are also very grateful to the anonymous referees for a very  thorough reading of the paper and suggesting multiple improvements. 

The work of M.\,F.\ was supported by the Engineering and Physical Sciences Research Council [grant number  EP/W013053/1].  The work of M.\,V.\ was funded by a Carnegie-Caledonian PhD scholarship from the Carnegie Trust for the Universities of Scotland held at the University of Glasgow and a CRM-ISM postdoctoral fellowship at Universit\'e de Montr\'eal.

{\small
\subsection*{Data availability}
Data sharing is not applicable to this article as no datasets were generated or analysed.
}

{\small
\subsection*{Conflict of interest}
The authors have no conflict of interest to declare.}


\begin{thebibliography}{}
\bibitem{Adler} Adler, M. ``Some finite dimensional integrable systems and their scattering behavior'', \textit{Comm. Math. Phys.} \textbf{55}, (1977) pp.~195--230. 

\bibitem{AST} Arakawa, T., Suzuki, T., Tsuchiya, A. ``Degenerate double affine Hecke algebra and conformal field theory''. In: Kashiwara, M., Matsuo, A., Saito, K., Satake, I. (eds) \textit{Topological field theory, primitive forms and related topics},  Progress in Mathematics 160, Birkh\"auser, Boston (1998), pp.~1--34. 

\bibitem{AHL} 
Atai, F.,  Halln\"as, M.,  Langmann, E. 
``Super-Macdonald polynomials: orthogonality and Hilbert space interpretation'', \textit{Comm. Math. Phys.} \textbf{388}, (2021) pp.~435--468.

\bibitem{Avan} Avan, J. ``Integrable extensions of the rational and trigonometric $A_N$ Calogero--Moser potentials'', \textit{Phys. Lett. A} \textbf{185}, 3 (1994), pp.~293--303.

\bibitem{BF} Baker, T.\,H., Forrester, P.\,J. ``A $q$-analogue of the type $A$ Dunkl operator and integral kernel'', \textit{Int.~Math.~Res.~Not.} \textbf{1997}, 14 (1997), pp.~667--686. 

\bibitem{BFH} Bellamy, G.,  Feigin, M., Hird, N. ``Two invariant subalgebras of rational Cherednik algebras'',  \textit{J. of Pure and Applied Algebra}, \textbf{230}, 3, (2026).

\bibitem{BM}
Bellamy, G., Martino, M. ``Affinity of Cherednik algebras on projective space'', \textit{Algebra and Number Theory} \textbf{8}, 5 (2014), pp.~1151--1177.

\bibitem{BEF} Braverman, A.,  Etingof, P., Finkelberg, M. ``Cyclotomic double affine Hecke algebras'' (with an appendix by  Nakajima, H., and Yamakawa, D.), \textit{Ann.~Sci.~de l’ENS}  \textbf{53}, 5 (2020), pp.~1249--1312.  


\bibitem{BGe}
Braverman, A., Gaitsgory, D. ``Poincar\'e--Birkhoff--Witt theorem for quadratic algebras of Koszul type'', \textit{J.~Algebra} \textbf{181} (1996).

\bibitem{BG}
Brown, K.\,A., Gordon, I. ``Poisson orders, symplectic reflection algebras and representation theory'', \textit{J.~reine angew.~Math.} \textbf{559} (2003), pp.~193--216.

\bibitem{Ch'00}
Chalykh, O. ``Bispectrality for the quantum Ruijsenaars model and its integrable deformation'', \textit{J.~Math.~Phys.} \textbf{41}, 8 (2000), pp.~5139--5167.

\bibitem{Chalykh} Chalykh, O., Etingof, P. ``Orthogonality relations and Cherednik identities for multivariable
Baker--Akhiezer functions'' (with an appendix by Chalykh, O.), \textit{Adv.~Math.}, \textbf{238} (2013), pp.~246--289.

\bibitem{ChalykhFairon} Chalykh, O., Fairon, M. ``Multiplicative quiver varieties and generalised Ruijsenaars--Schneider models'', \textit{J.~Geom.~Phys.} \textbf{121} (2017), pp.~413--437.

\bibitem{CFV}
Chalykh, O.,  Feigin, M.,  Veselov, A. 
``New integrable generalizations of Calogero--Moser quantum problem'', \textit{
J.~Math.~Phys.}, \textbf{39} (1998), pp.~695--703.

\bibitem{Ch'91} Cherednik, I.\,V. ``A unification of Knizhnik--Zamolodchikov and Dunkl operators via affine Hecke algebras'', \textit{Invent.~Math.} \textbf{106}, 2 (1991), pp.~411--432. 


\bibitem{Ch'92} Cherednik, I.\,V. ``Quantum Knizhnik--Zamolodchikov equations and affine root systems'', \textit{Comm.~Math.~Phys.} \textbf{150}, 1 (1992), pp.~109--136. 

\bibitem{Ch'95} Cherednik, I.\,V. ``Double affine Hecke algebras and Macdonald's conjectures'', \textit{Ann.~Math.} \textbf{141} (1995), pp.~191--216.

\bibitem{Cherednik} Cherednik, I.\,V. \textit{Double affine Hecke algebras}, London Mathematical Society Lecture Note Series 319, Cambridge University Press, New York (2005).

\bibitem{DKKV} Derkachov, S.\,E., Karakhanyan, D.\,R., Kirschner, R., Valinevich, P. ``Iterative construction of $U_q(\fsl(n+1))$ representations and Lax matrix factorisation'', \textit{Lett.~Math.~Phys.} \textbf{85} (2008), pp.~221--234.

\bibitem{vD} van Diejen, J.\,F. ``Difference Calogero--Moser systems and finite Toda chains'', \textit{J.~Math.~Phys.} \textbf{36}, 3 (1995).


\bibitem{vDE} van Diejen, J.\,F., Emsiz, E. ``Spectrum and eigenfunctions of the lattice hyperbolic Ruijsenaars--Schneider system with exponential Morse term'', \textit{Ann.~Henri Poincar\'e} \textbf{17} (2016), pp.~1615--1629.


\bibitem{Dunkl} Dunkl, C.\,F. ``Differential-difference operators associated to reflection groups'', \textit{Trans.~Amer.~Math.~Soc.} \textbf{311} (1989), pp.~167--183.

\bibitem{Etingof} Etingof, P. \textit{Calogero--Moser systems and representation theory}, Zurich Lectures in Advanced Mathematics, EMS Press (2007).

\bibitem{Etingof2}
Etingof, P. ``Cherednik and Hecke algebras of varieties with a finite group action'', \textit{Mosc.~Math.~J.} \textbf{17}, 4 (2017), pp.~635--666.


\bibitem{EG} Etingof, P., Ginzburg, V. ``Symplectic reflection algebras, Calogero--Moser space, and deformed Harish-Chandra homomorphism'', \textit{Invent.~Math.} \textbf{147}, 243 (2002). 

\bibitem{F} 
Feigin, M. ``Generalized Calogero--Moser systems from rational Cherednik algebras'', \textit{Sel. Math. New Ser.} \textbf{18}, (2012) pp.~253--281.

\bibitem{FH} Feigin, M., Hakobyan, T. ``On Dunkl angular momenta algebra'', \textit{J.~High Energy Phys.} \textbf{2015}, 107 (2015).

\bibitem{FS} Feigin, M., Silantyev, A. ``Generalized Macdonald--Ruijsenaars systems'', \textit{Adv.~Math.} \textbf{250} (2014), pp.~144--192.


\bibitem{Hayashi} Hayashi, T. ``$Q$-Analogues of Clifford and Weyl algebras --- spinor and oscillator representations of quantum enveloping algebras'', \textit{Comm.~Math.~Phys.} \textbf{127} (1990), pp.~129--144.  

\bibitem{Heckman}
Heckman, G. ``A remark on the Dunkl differential-difference operators''. In: Barker B., Sally, P. (eds) \textit{Proceedings of the conference on harmonic analysis on reductive groups}, Progress in
Mathematics, Birkh\"auser 101 (1991), pp.~181--191.

\bibitem{Inozemtsev} Inozemtsev, V.\,I. ``New completely integrable multiparticle
dynamical systems'', \textit{Phys. Scr.} \textbf{29}, (1984) pp.~518--520.

\bibitem{IM} Inozemtsev, V.\,I., Meshcheryakov, D.\,V.  {``The discrete spectrum states of finite-dimensional quantum systems connected with
Lie algebras''}, \textit{Phys. Scr.} \textbf{33}, (1986) pp.~99--104.





\bibitem{Jimbo} Jimbo, M. ``A $q$-analogue of $U(\fgl(N+1))$, Hecke algebra, and the Yang--Baxter equation'', \textit{Lett.~Math.~Phys.} \textbf{11} (1986), pp.~247--252. 

\bibitem{KN} Kajihara, Y., Noumi, M. ``Raising operators of row type for Macdonald polynomials'', \textit{Compos. Math.} \textbf{120}, 2 (2000), pp.~119--136.

\bibitem{Kirillov} Kirillov, A.\,A. ``Lectures on affine Hecke algebras and Macdonald's conjectures'', \textit{Bull. Am. Math. Soc.} \textbf{34}, 3 (1997), pp.~251--292.

\bibitem{KS} Klimyk, A., Schm\"udgen, K. \textit{Quantum groups and their representations}, Texts and Monographs in Physics, Springer-Verlag (1997).

\bibitem{Koorn} Koornwinder, T.\,H. ``Askey--Wilson polynomials for root systems of type~$BC$'', \textit{Contemp. Math.} \textbf{138}, (1992) pp.~189--204.

\bibitem{KP} Kulkarni, M., Polychronakos, A. ``Emergence of the Calogero family of models in external potentials: duality, solitons and hydrodynamics'', \textit{J. Phys. A: Math. Theor.} \textbf{50}, 45 (2017), p.~455202.

\bibitem{Lusztig} Lusztig, G. ``Affine Hecke algebras and their graded version'', \textit{J.~Am.~Math.~Soc.} \textbf{2}, 3 (1989).


\bibitem{Polychronakos2} Polychronakos, A.\,P. ``New integrable systems from unitary matrix models'', \textit{Phys. Lett. B} \textbf{277}, 1--2 (1992), pp.~102--108.

\bibitem{Ruij} Ruijsenaars, S.\,N.\,M. ``Complete integrability of relativistic Calogero--Moser systems and elliptic function identities'', \textit{Comm. Math. Phys.},
  \textbf{110}, 2 (1987), pp.~191--213.

\bibitem{SVdifferential}
Sergeev, A.\,N., Veselov, A.\,P.,
``Deformed quantum Calogero--Moser problems and Lie superalgebras'', 
\textit{Comm. Math. Phys.},
\textbf{245}, 2 (2004), pp.~249--278.

\bibitem{SV}
Sergeev, A.\,N., Veselov, A.\,P., ``Deformed Macdonald--Ruijsenaars operators and super Macdonald polynomials'', \textit{Comm.~Math.~Phys.} \textbf{288}, 2 (2009), pp.~653--675.

\bibitem{Suzuki} Suzuki, T. ``Rational and trigonometric degeneration of the double affine Hecke algebra of type $A$'', \textit{Int.~Math.~Res.~Not.} \textbf{2005}, 37 (2005), pp.~2249--2262. 



\end{thebibliography}
\end{document}